\documentclass[12pt]{article}
\usepackage{amsfonts, amsmath, amssymb, amsgen, amsthm, amscd,
latexsym}
\usepackage{color}
\usepackage[all]{xy}

\def\Diff{\mathop{\rm Diff}\nolimits}

\def\GL{\mathop{\rm GL}\nolimits}

\def\act{\mathop{\rm act}\nolimits}
\def\Id{\mathop{\rm Id}\nolimits}
\def\id{\mathop{\rm Id}\nolimits}

\def\Ad{\mathop{\rm Ad}\nolimits}
\def\ad{\mathop{\rm ad}\nolimits}

\def\Ker{\mathop{\rm Ker}\nolimits}

\def\Tr{\mathop{\rm Tr}\nolimits}

\def\Hom{\mathop{\rm Hom}\nolimits}
\def\Tot{\mathop{\rm Tot}\nolimits}
\def\Ch{\mathop{\rm Ch}\nolimits}
\def\cs{\mathop{\rm cs}\nolimits}

\def\Cb{{\mathbb C}}

\def\Nb{{\mathbb N}}
\def\Rb{{\mathbb R}}

\def\Zb{{\mathbb Z}}

\def\Ac{{\cal A}}
\def\Bc{{\cal B}}
\def\Ec{{\cal E}}
\def\Fc{{\cal F}}

\def\Hc{{\cal H}}

\def\Jc{{\cal J}}
\def\Ic{{\cal I}}
\def\Lc{{\cal L}}

\def\Sc{{\cal S}}
\def\Uc{{\cal U}}
\def\Vc{{\cal V}}
\def\Qc{{\cal Q}}
\def\Kc{{\cal K}}
\def\Lc{{\cal L}}
\def\Cc{{\cal C}}

\def\Dc{{\cal D}}

\def\a{\alpha}
\def\b{\beta}
\def\d{\delta}
\def\D{\Delta}
\def\g{\gamma}

\def\G{\Gamma}
\def\lb{\lambda}
\def\Lba{\Lambda}
\def\om{\omega}
\def\Om{\Omega}
\def\s{\sigma}

\def\t{\theta}
\def\z{\zeta}
\def\ve{\varepsilon}
\def\vp{\varphi}

\def\x{\xi}

\def\bash{\backslash}

\def\fl{\forall}

\def\ot{\otimes}

\def\ra{\rightarrow}

\def\sbs{\subset}

\def\rt{\triangleright}

\def\lt{\triangleleft}
\def\cl{\blacktriangleright\hspace{-4pt} < }
\def\al{>\hspace{-4pt}\vartriangleleft}

\def\acl{\blacktriangleright\hspace{-4pt}\vartriangleleft }

\def\brt{\blacktriangleright}

\def\hd{\overset{\ra}{\partial}}

\def \vd{\uparrow\hspace{-4pt}\partial}
\def\hs{\overset{\ra}{\sigma}}
\def \vs{\uparrow\hspace{-4pt}\sigma}
\def\hta{\overset{\ra}{\tau}}
\def \vta{\uparrow\hspace{-4pt}\tau}
\def\hb{\overset{\ra}{b}}
\def \vb{\uparrow\hspace{-4pt}b}

\def\hB{\overset{\ra}{B}}
\def \vB{\uparrow\hspace{-4pt}B}
\def\p{\partial}
\def\dbar{{^-\hspace{-7pt}\d}}
\def\gbar{{-\hspace{-9pt}\g}}
\def\0D{\Delta^{(0)}}
\def\1D{\Delta^{(1)}}
\def\Db{\blacktriangledown}

\def\ts{\times}
\def\wg{\wedge}

\def\td{\tilde}
\def\cop{{^{\rm cop}}}

\newcommand{\FD}{\mathfrak{D}}

\newcommand{\Fa}{\mathfrak{a}}
\newcommand{\Fd}{\mathfrak{d}}
\newcommand{\Fg}{\mathfrak{g}}
\newcommand{\Fh}{\mathfrak{h}}

\newcommand{\Fl}{\mathfrak{l}}

\newcommand{\Fn}{\mathfrak{n}}

\newcommand{\Fq}{\mathfrak{q}}

\newcommand{\FN}{\mathfrak{N}}

\def\tpsi{\tilde{\psi}}

\newtheorem{theorem}{Theorem}[section]
\newtheorem{remark}[theorem]{Remark}
\newtheorem{proposition}[theorem]{Proposition}
\newtheorem{lemma}[theorem]{Lemma}
\newtheorem{corollary}[theorem]{Corollary}

\def\limproj{\mathop{\oalign{lim\cr
\hidewidth$\longleftarrow$\hidewidth\cr}}}

\def\build#1_#2^#3{\mathrel{
\mathop{\kern 0pt#1}\limits_{#2}^{#3}}}
\newcommand{\ps}[1]{~\hspace{-4pt}_{^{(#1)}}}
\newcommand{\ns}[1]{~\hspace{-4pt}_{_{{<#1>}}}}

\newcommand{\lu}[1]{\;^{#1}\hspace{-2pt}}

\def\odots{\ot\cdots\ot}
\def\wdots{\wedge\dots\wedge}

\def\one{{\bf 1}}

\numberwithin{equation}{section}
 \newcommand{\eg}{{\it e.g.\/}\ }
 \newcommand{\ie}{{\it i.e., }\ }
 \def\cf{{\it cf.\/}\ }

\def\a{\alpha}
\def\b{\beta}

\def\d{\delta}

\def\g{\gamma}

\def\i{\iota}
\def\lb{\lambda}
\def\om{\omega}
\def\s{\sigma}
\def\t{\theta}
\def\ve{\varepsilon}
\def\ph{\phi}
\def\vp{\varphi}

\def\z{\zeta}

\def\D{\Delta}
\def\G{\Gamma}
\def\Lb{\Lambda}
\def\Om{\Omega}

\def\dt{\left.\frac{d}{dt}\right|_{_{t=0}}}

\def\fl{\forall}

\def\ot{\otimes}
\def\part{\partial}

\def\sbs{\subset}

\def\ts{\times}
\def\wdg{\wedge}

\def\ra{\rightarrow}

\def\text{\hbox}

\def\bash{\backslash}

\def\fl{\forall}

\def\ot{\otimes}

\def\ra{\rightarrow}

\def\sbs{\subset}

\def\ts{\times}
\def\wdg{\wedge}

\def\Ad{\mathop{\rm Ad}\nolimits}

\def\Diff{\mathop{\rm Diff}\nolimits}

\def\Hom{\mathop{\rm Hom}\nolimits}

\def\GL{\mathop{\rm GL}\nolimits}

\def\Id{\mathop{\rm Id}\nolimits}
\def\exp{\mathop{\rm exp}\nolimits}

\def\Ker{\mathop{\rm Ker}\nolimits}

\def\build#1_#2^#3{\mathrel{
\mathop{\kern 0pt#1}\limits_{#2}^{#3}}}

\def\limproj{\mathop{\oalign{lim\cr
\hidewidth$\longleftarrow$\hidewidth\cr}}}

\numberwithin{equation}{section}
\newcommand{\comment}[1]{\relax}

\def\sb{{\bf s}}
\def\tb{{\bf t}}
\def\Gb{{\bf G}}

\begin{document}
%%%%%%%%%%%%%%%%%%%%%%%%%%%%%%%%%%%%%%%%%%%%%%%
\title{\bf Hopf cyclic cohomology and transverse characteristic classes}
\author{
\begin{tabular}{cc}
Henri Moscovici \thanks{Research
    supported by the National Science Foundation
    awards DMS-0652167 and DMS-0969672.}~~\thanks{Department of Mathematics,
    Ohio State University, Columbus, OH 43210, USA }\quad and \quad  Bahram Rangipour
    \thanks {Department of Mathematics  and   Statistics,
     University of New Brunswick, Fredericton, NB, Canada}
      \end{tabular}}

\date{ \ }

\maketitle

%%%%%%%%%%%%%%%%%%%%%%%%%%%%%%%%%%%
\begin{abstract}
We refine the cyclic cohomological apparatus for computing the Hopf
cyclic cohomology of the Hopf algebras associated to infinite primitive
Cartan-Lie pseudogroups, and for the transfer of their characteristic
classes to foliations. The main novel feature is the precise
  identification as a Hopf cyclic complex of the image of the
 canonical homomorphism from the Gelfand-Fuks complex to the Bott complex
for equivariant cohomology. This provides a convenient new model for the Hopf cyclic cohomology of the geometric Hopf algebras, which
allows for an efficient transport of the Hopf cyclic classes
via characteristic homomorphisms.
We illustrate the latter aspect
by indicating how to realize the universal Hopf cyclic Chern classes in terms
of explicit cocycles in
the cyclic cohomology of \'etale foliation groupoids.
\end{abstract}

%%%%%%%%%%%%%%%%%%%%%%%%%%%%%%%%%%%%%%%%%%%%%%%%%%%
\section*{Introduction}

Hopf algebras and their Hopf cyclic cohomology
emerged as a new geometric tool
in the work~\cite{cm2} of A. Connes and of the first named author computing
the local index formula for transversely hypoelliptic operators
on foliations. One of the key results was the construction
of a canonical isomorphism~\cite[Theorem 11]{cm2} from the Gelfand-Fuks
cohomology~\cite{GF} of the Lie algebra of formal vector fields on $\Rb^n$ to
the Hopf cyclic cohomology of the Hopf algebra $\Hc (n)$ associated
to the pseudogroup of local diffeomorphisms of $\Rb^n$. This isomorphism
made it possible to recognize the cyclic cocycle expressing the Connes-Chern character
of the hypoelliptic signature operator as representing a classical transverse
characteristic class~\cite{BottHaef}
in the cyclic cohomology of the transverse \'etale groupoid associated to the
foliation.

Extending the construction of
$\Hc (n)$ to all classical types of transverse geometries,
in~\cite{mr09} we have associated a similar Hopf algebra $\Hc(\Gb)$ to each
 infinite primitive Cartan-Lie pseudogroup~\cite{ECar, Guill}, or equivalently
 to each classical subgroup of diffeomorphisms $\Gb \sbs \Diff (\Rb^n)$.
 We showed that all these
 geometric Hopf algebras can be reconstituted as bicrossed products
 of traditional Hopf algebras, with one of the factors corresponding
 to a Lie algebra while
 the other is associated to a group. The bicrossed product structure was
 used in turn to express the Hopf cyclic cohomology of $\Hc(\Gb)$
 in terms of adapted bicocyclic complexes.
\medskip

The present paper clarifies
the relationship between the Hopf cyclic cohomology of the Hopf algebras
$\Hc(\Gb)$ and the
classical characteristic classes (\cf~\cite{BottHaef}, \cite{BernRos}, \cite{Bott})
of foliations with transverse $\Gb$-structure.
At the level of universal classes, we extend
the canonical cohomological isomorphism
 of~\cite{cm2} to all classical types of transverse geometries.
More significantly, by refining the cyclic cohomological apparatus for computing the Hopf
cyclic cohomology of $\Hc(\Gb)$, we succeed in
explicitly identifying as a Hopf cyclic complex the image of the
Gelfand-Fuks complex through the canonical homomorphism of~\cite{cm2} into the Bott complex
for equivariant cohomology. Distinct from the original realization
of the Hopf cyclic
cohomology of $\Hc(\Gb)$ as differentiable cyclic cohomology~\cite{cm2, cm4},
this provides a front-end model for the Hopf cyclic cohomology of $\Hc(\Gb)$.
Furthermore, we produce
characteristic homomorphisms from the newly developed models for
Hopf cyclic characteristic classes to the cyclic cohomology of the convolution
algebras of \'etale holonomy groupoids, which work in the relative case with
no compactness restriction. Finally, we concretely illustrate the latter feature by indicating
how to realize the universal Chern classes in terms of explicit cocycles in
the cyclic cohomology of the convolution
algebras of action groupoids.
\medskip

We now outline the plan of the paper and
highlight its main results. Section \ref{S: HCC} starts by
recalling from~\cite{mr07, mr09} the construction
of the geometric Hopf algebras $\Hc(\Gb)$ and their bicrossed product
decomposition $\Hc(\Gb) =  \Fc (\FN)\acl \Uc (\Fg)$ (\S \ref{Ss: geom HA}), as well as
the  Hopf cyclic bicomplexes associated to them (\S \ref{Ss:Hopf bic}) and their
reduction to a Chevalley-Eilenberg-type complex (\S \ref{Ss: H-C-E}). We then
give in \S \ref{Ss:gH} a new Hopf cyclic interpretation to the latter, as a
 $\Fg$-equivariant Hopf cyclic complex for the abelian Hopf factor $\Fc (\FN)$
 with coefficients in the Koszul resolution of  $\wdg^\bullet \Fg^\ast$, viewed as a
 differential graded stable anti-Yetter-Drinfeld module (\cf Theorem \ref{thm:dgH}).
 \smallskip

 Section~\ref{S: van est} is devoted to the construction (\S\S \ref{Ss: GF to EC}-\ref{Ss: EC to HC})
of an explicit quasi-isomorphism between the Gelfand-Fuks
cohomology complex of the Lie algebra of formal vector fields $\Fa (\Gb) = \Fg\bowtie\Fn$
and the Hopf cyclic cohomology of  $\Hc(\Gb)$. While this is patterned on
the corresponding construction in~\cite{cm2}, the novel feature consists in
the exact identification in Hopf cyclic terms
of the target of the canonical quasi-isomorphism.
Theorem \ref{thm:KvE}  in \S \ref{Ss: thm.vE}
identifies it as being precisely the
 $\Fg$-equivariant Hopf cyclic complex of Theorem \ref{thm:dgH}.
 \smallskip

The transfer of the universal Hopf cyclic classes from the Hopf cyclic cohomology
of $\Hc(\Gb)$ to the cyclic cohomology of
the convolution algebras of $\Gb$-holonomy \'etale groupoids is treated
in Section~\ref{S: HCmaps}.  In \S \ref{Ss: Hact}
we construct  characteristic maps adapted to the bicrossed product
models for Hopf cyclic cohomology, and landing in the Getzler-Jones models
\cite{gj} for the cyclic cohomology of crossed product algebras.
Theorem \ref{thm:actmap} reconciles these maps with the original characteristic
map of~\cite{cm2}, by showing that they are all equivalent.
With the goal of facilitating the transport of characteristic classes from the Gelfand-Fuks
cohomology via the Hopf cyclic cohomology,
in \S \ref{Ss: dgchar} we introduce a
differential graded characteristic map from the
$\Fg$-equivariant Hopf cyclic complex of $\Hc(\Gb)$ to
 differential  \'etale groupoids.
 Remark \ref{rem:dgactmap} explains how this transfer can be concretely implemented,
 by relating the above characteristic map to
 Connes' map~\cite[Theorem 14, p. 220]{book} from the
 Bott bicomplex for equivariant cohomology to the cyclic cohomology bicomplex,
 and to Gorokhovsky's differential graded analogue~\cite[Theorem 19]{Sasha02}.

Finally, we treat the relative case in Section \ref{S:rel}. Unlike
the original characteristic map (\cf ~\cite{cm2}), the graded
characteristic map does not require an invariant trace and thus
descends to the relative complex even when
 the underlying reductive subgroup of the linear isotropy group is not compact.
In particular, this allows to realize the Hopf cyclic version of
the universal Chern classes both in terms of
 the Bott bicomplex,
and as cocycles in the cyclic cohomology complex of an
 \'etale foliation groupoid.

 %%%%%%%%%
\tableofcontents

\section{Hopf cyclic cohomological models} \label{S: HCC}
\subsection{Geometric Hopf algebras and their
bicrossed product structure} \label{Ss: geom HA}

Inspired by the G. I. Kac procedure~\cite{Kac} for producing
non-commutative and non-cocommutative quantum groups
out of  matched pairs of finite groups, the construction of the
Hopf algebra associated to an infinite primitive Cartan-Lie pseudogroup
(\cf~\cite{cm2} for the general pseudogroup and~\cite{mr09} for all others)
relies on splitting the corresponding group $\Gb$ of
globally defined diffeomorphisms of $\Rb^n$
as a set-theoretical product of two subgroups:
$\, \Gb = G \cdot N$, $\,G \cap N= \{e\}$. When the primitive pseudogroup
is flat, \ie contains all the Euclidean translations,
$G$ is the subgroup of the affine transformations of $\Rb^n$ that belong to
$\Gb$, and $N \sbs \Gb$ is the subgroup of those diffeomorphisms
that preserve the origin to order $1$. The only primitive pseudogroup
which is not flat is the pseudogroup of contact transformations.
In that case (\cf~\cite[\S 2.2]{mr09}),
 the Euclidean translations are replaced by Heisenberg
translations. More precisely, regarding
$\Rb^{2n+1}$ as the underlying manifold of the Heisenberg group $H_n$,
$G$ is the semidirect product $H_n \ltimes G_0$, with $H_n$ acting by
 left translations, and $G_0$ consisting of the linear contact transformations.
 The complementary factor $N$ is the subgroup
of all contact diffeomorphisms preserving the origin
to order $1$ in the sense of Heisenberg calculus, \ie whose differential
at $0$ is the identity map of
the Heisenberg tangent bundle.

The splitting $\Gb= G \cdot N$ allows to represent uniquely
any $\phi \in \Gb$ as a product
$\phi  = \varphi \cdot \psi$, with $\varphi \in G$ and  $\psi \in N$.
Factorizing the product of any two elements $\varphi \in G$ and  $\psi \in N$
in the reverse order,
$\, \psi \cdot \varphi \, = \, (\psi\rt
\varphi)\cdot(\psi \lt \varphi)$,
one obtains a left action $\psi  \mapsto \tpsi (\varphi) :=
\psi \rt \varphi$ of $N$ on $G$, along with
a right action $\lt$
of $G$ on $N$. Equivalently, these actions are the restrictions of
the natural actions of $ \Gb$ on the coset spaces
$\Gb \slash N \cong G$ and $G \bash \Gb \cong N$.

Denoting by $\Fg$ the Lie algebra of $G$ and by $\Fn$ the Lie algebra of $N$,
we abbreviate $\Uc = \Uc (\Fg)$ and let $\Fc$ stand for the
algebra of functions on $N$ generated by the components of the jet
of $ \tpsi$ at $e \in G$, as $\psi$ runs through $N$.  $\Uc$ is endowed with the
standard Hopf algebra structure of a universal enveloping algebra, while
 $\Fc$ is a restricted dual to $ \Uc (\Fn)$ and is a Hopf algebra with respect
 to the coproduct inherited from the group multiplication on $N$. While referring
 the reader to~\cite[\S 2.1-2.2]{mr09} for details, we succinctly recall the
 salient points regarding the construction of the Hopf algebra $\, \Hc = \Hc (\Gb)$.

  First of all,  $ \, \Uc$ is a right  $\Fc$-comodule coalgebra with respect to the coaction
$\Db:\Uc \ra\Uc \ot \Fc$ which can be concretely described as follows.
Let $\, \{X_1, \ldots , X_m \}$ be a basis of $\Fg$, and let
$\{X_I = X_1^{i_1} \cdots X_m^{i_m} \, ; \, i_1, \ldots , i_m \in \Zb^+   \}$
be the corresponding Poincar\'e-Birkhoff-Witt  basis  of $\Uc$. Then
\begin{align}  \label{gij}
\begin{split}
&\Db (X_I)\,  = \, \sum_J  X_J \ot \eta_I^J  , \quad
\text{with} \quad  \eta_I^J (\cdot) := \,
\g_I^J (\cdot)(e) , \\
&\text{where} \qquad
 U_\psi \, X_I \, U_{\psi^{-1}} \, = \, \sum_J  \g_I^J (\psi)\, X_J, \qquad \psi \in N .
 \end{split}
\end{align}
The properties which make
$\Db:\Uc\ra\Uc\ot \Fc$ a right coaction are:
\begin{align} \label{Fcoact}
\begin{split}
&u\ns{0}\ps{1}\ot u\ns{0}\ps{2}\ot u\ns{1}= u\ps{1}\ns{0}\ot
u\ps{2}\ns{0}\ot u\ps{1}\ns{1}u\ps{2}\ns{1}\\
&\epsilon(u\ns{0})u\ns{1}=\epsilon(u)1 , \qquad \qquad \forall \, u\in \Uc.
\end{split}
\end{align}
One can then  forms a  cocrossed product coalgebra $\Fc\cl\Uc$ that has
$\Hc\ot \Kc$ as underlying vector space and the coalgebra
structure given by:
\begin{align} \label{bicrosscoalg}
\begin{split}
&\Delta(f\cl u)= f\ps{1}\cl u\ps{1}\ns{0}\ot  f\ps{2}u\ps{1}\ns{1}\cl u\ps{2}, \\
&\epsilon(h\cl k)=\epsilon(h)\epsilon(k).
\end{split}
\end{align}

 On the other hand, the
  right action $\lt$ of $G$ on $N$ induces a left action of $G$ on $\Fc$, which in turn
 gives rise to a left action $\rt$ of $\Uc$ on $\Fc$; that makes  $\Fc$
 a  left $\Uc$-module algebra, \ie
  for any $u\in \Uc$, and
$f,g\in \Fc$,
\begin{align}
\begin{split}
 &u\rt(fg)=(u\ps{1}\rt
f)(u\ps{2}\rt g) , \\
&u\rt 1=\epsilon(u)1.
\end{split}
\end{align}
This allows to endow the underlying vector space $\Fc\ot \Uc$
 an algebra structure, which we denote it by $\Fc\al \Uc$; it
 has  $1\al 1$ as its unit, and the product
\begin{equation*}
(f\al u)(g\al v)=f \;u\ps{1}\rt g\al u\ps{2}v .
\end{equation*}
Furthermore, $\Uc$ and $\Fc$ form  a matched pair
of Hopf algebras in the sense of \cite{maj}. More precisely,
the above action and coaction satisfy the
following properties:
\begin{align}  \label{mp2}
\begin{split}
&\epsilon(u\rt f)=\epsilon(u)\epsilon(f), \qquad \forall \, u\in\Uc , \, f\in \Fc ,\\
&\Delta(u\rt f)=u\ps{1}\ns{0} \rt f\ps{1}\ot
u\ps{1}\ns{1}(u\ps{2}\rt f\ps{2}), \\
 &\Db(1)=1\ot 1, \\
 &\Db(uv)=u\ps{1}\ns{0} v\ns{0}\ot
u\ps{1}\ns{1}(u\ps{2}\rt v\ns{1}),\\
&u\ps{2}\ns{0}\ot
(u\ps{1}\rt f)u\ps{2}\ns{1}=u\ps{1}\ns{0}\ot
u\ps{1}\ns{1}(u\ps{2}\rt f) .
\end{split}
\end{align}
Superimposing the coalgebra structure $\Fc\al \Uc$ and the algebra structure
$\Fc\cl \Uc$, one obtains
the bicrossed product Hopf
algebra $\Hc = \Fc\acl \Uc$, which has as antipode
\begin{equation*}
S(f\acl u)=(1\acl S(u\ns{0}))(S(fu\ns{1})\acl 1) .
\end{equation*}

The following auxiliary results will be needed later in this paper.

\begin{lemma}\label{S-linear}
For any $u\in \Uc$, and $f\in \Fc$ one has
\begin{equation*}
S(u\rt f)= u\ns{0}\rt S(f)\; S(u\ns{1}).
\end{equation*}
\end{lemma}

\begin{proof}
Let us  endow $\Uc \ot \Fc$  with tensor coproduct, and  consider
the convolution algebra $\Hom(\Uc\ot \Fc, \Fc)$.  Define three
elements $K, L$ and $M$ of the latter algebra as follows:
\begin{align*}
& K(u\ot f)= u\rt f, \quad
 L(u\ot f)=u\rt S(f),\quad
M(u\ot f)= S(u\ns{0}\rt f)\; u\ns{1}.
\end{align*}
Using the fact that $\Fc$ is a $\Uc$-module algebra,  one checks that
$K$ and $L$ are left and right inverse of one another; indeed,
\begin{align*}
&(K\ast L)(u\ot f)= (u\ps{1}\rt f\ps{1})( u\ps{2}\rt S(f\ps{2}))=\\
& u\rt(f\ps{1}S(f\ps{2}))=u\rt\epsilon(f)=\epsilon(u)\epsilon(f) ,
\end{align*}
and the verification of other identity is similar.
On the other hand, making use of \eqref{mp2}, one sees  that $M\ast
K=\Id$. So $M=L$, which is equivalent to the claimed identity.
\end{proof}
%%%%%%%%%%%%%%%%%%%%%%%%%%%%%%%%%%%%%%%%%%%%%%%%%%%
Although $\Uc$ is not $\Fc$ comodule algebra, \ie  $\Db$ is not
algebra map,
 the properties \eqref{mp2} seem to suggest that there could be an algebra
 structure on $\Uc\ot \Fc$ which would make $\Db$ an algebra map.
With this in mind, we define a new action
 of $\Uc$ on $\Fc$,  $\, \brt:\Uc\ot \Fc\ra \Fc$, which convolutes the
above action and coaction, by
\begin{align} \label{eq:action}
u\brt f=S^{-1}(u\ns{1})\; u\ns{0}\rt f , \qquad u \in \Uc, \, f \in \Fc .
\end{align}

\begin{lemma}\label{action}
The above formula defines an action which makes $\Fc$ a $\Uc\cop$-module
algebra.
\end{lemma}
\begin{proof}
To prove that it is an action, we note that by Lemma \ref{S-linear},
\begin{align*}
&S(u\brt f)= S(S^{-1}(u\ns{1})\; u\ns{0}\rt f)=\\
 &S(u\ns{0}\rt f)\;u\ns{1}= u\ns{0}\rt S(f) S(u\ns{1})u\ns{2}= u\rt
 S(f) .
\end{align*}
This implies that
\begin{equation*}
S(uv\brt f)= uv\rt S(f)=u\rt (v\rt S(f))= u\rt S(v\brt f)= S(u\brt(v
\brt f)).
\end{equation*}
Next, we show  $\Uc^{cop}$ acts on $\Fc$ by a Hopf action, as follows:
\begin{align*}
&S(u\brt (fg))= u\rt S(fg)= u\ps{1}\rt S(g)\; u\ps{2}\rt S(f)= \\
&S(u\ps{1}\brt g)S(u\ps{2}\brt f)=S((u\ps{2}\brt f)( u\ps{1}\brt
g)).
\end{align*}
\end{proof}

In view of the above lemma, one can now form a new crossed product
algebra, which we denote by  $\Fc \al \Uc^{\rm cop}$.

Let $\Db^{\top}= \top_{1,2}\circ\Db:\Uc\ra \Fc\ot \Uc$, where
$\top_{1,2}: \Uc\ot \Fc\ra \Fc\ot \Uc$ stands for the flip map.

\begin{lemma}\label{D-algebra}
The map $\Db^{\top}: \Uc\ra \Fc\al \Uc^{\rm cop}$ is an algebra
map.
\end{lemma}
\begin{proof} Indeed, using \eqref{mp2} we can write
\begin{align*}
&\Db^{\top}(u)\Db^{\top}(v)= (u\ns{1}\ot u\ns{0})(v\ns{1}\ot v\ns{0})= \\
&u\ns{1}( u\ns{0}\ps{2}\brt v\ns{1})\ot u\ns{0}\ps{1}v\ns{0}=\\
& u\ps{1}\ns{1}u\ps{2}\ns{1}( u\ps{2}\ns{0}\brt v\ns{1})\ot u\ps{1}\ns{0}v\ns{0}=\\
&u\ps{1}\ns{1}u\ps{2}\ns{2}  S^{-1}(u\ps{2}\ns{1}) (u\ps{2}\ns{0}\rt v\ns{1})\ot u\ps{1}\ns{0}v\ns{0}=\\
& u\ps{1}\ns{1}(u\ps{2}\rt v\ns{1})\ot u\ps{1}\ns{0}v\ns{0}=(uv)\ns{1}\ot (uv)\ns{0}=\Db^{\top}(uv). \\
\end{align*}
\end{proof}

\begin{lemma}\label{antipode}
The antipode of $\Uc$ satisfies the following colinearity property:
\begin{equation*}
\Db(S(u))= S(u\ps{2}\ns{0})\ot S(u\ps{1})\rt u\ps{2}\ns{1}  , \qquad \fl \, u \in \Uc .
\end{equation*}
\end{lemma}
\begin{proof}
We first show that the operator
\begin{equation*}
L(u)= S(u\ps{1})\rt u\ps{2}\ns{1}\ot S(u\ps{2}\ns{0}) .
\end{equation*}
is the right inverse of $\Db^\top$ in
the convolution algebra  $ \Hom(\Uc, \Fc\al \Uc^{\text{~cop}})$.
Indeed, one has
\begin{align*}
&(\Db^\top\ast L)(u)= \Db^\top(u\ps{1})L(u\ps{2})= \\
& (u\ps{1}\ns{1}\ot u\ps{1}\ns{0})(S(u\ps{2})\rt u\ps{3}\ns{1}\ot
S(u\ps{3}\ns{0}))=\\
& u\ps{1}\ns{1} u\ps{1}\ns{0}\ps{2}\brt (S(u\ps{2})\rt
u\ps{3}\ns{1})\ot u\ps{1}\ns{0}\ps{1}S(u\ps{3}\ns{0})= \\
&u\ps{1}\ps{1}\ns{1} u\ps{1}\ps{2}\ns{1}u\ps{1}\ps{2}\ns{0}\brt
(S(u\ps{2})\rt
u\ps{3}\ns{1})\ot u\ps{1}\ps{1}\ns{0}S(u\ps{3}\ns{0})= \\
&u\ps{1}\ns{1}u\ps{2}\ns{1}u\ps{2}\ns{0}\brt (S(u\ps{3})\rt
u\ps{4}\ns{1})\ot u\ps{1}\ns{0}S(u\ps{4}\ns{0})=\\
&u\ps{1}\ns{1}u\ps{2}\ns{2} S^{-1}(u\ps{2}\ns{1})u\ps{2}\ns{0}\rt
(S(u\ps{3})\rt
u\ps{4}\ns{1})\ot u\ps{1}\ns{0}S(u\ps{4}\ns{0})=\\
&u\ps{1}\ns{1} (u\ps{2}S(u\ps{3}))\rt u\ps{4}\ns{1}\ot
u\ps{1}\ns{0}S(u\ps{4}\ns{0})=\\
&u\ps{1}\ns{1}u\ps{2}\ns{1}\ot u\ps{1}\ns{0}S(u\ps{2}\ns{0})=
u\ns{1}\ot u\ns{0}\ps{1}S(u\ns{0}\ps{2})=\\
&u\ns{1}\ot \epsilon(u\ns{0})=\epsilon(u)\ot 1.
\end{align*}
On the other hand, since according to Lemma \ref{D-algebra}, $\Db^\top$ is
an algebra map, it is automatically two sided invertible, with
$\Db^\top\circ S$ as its inverse. It follows that $\Db^\top \circ S=L$,
which proves the claim.
\end{proof}
%%%%%%%%%%%%%%%%%%%%%

\subsection{Hopf cyclic complexes based on bicrossed product structure} \label{Ss:Hopf bic}

The Hopf algebra $\Hc = \Hc (\Gb)$  comes equipped with the character
$\d \in \Hc^\ast$, the unique extension of the
infinitesimal modular character of $\Fg$, $ \d (X) = \Tr (\ad X)$.
Moreover, $\d:\Hc\ra \Cb$ satisfies $S_\d^2=\Id$, where $S_\d(h)=\d(h\ps{1})S(h\ps{2})$.
This means that the
 corresponding module $\Cb_\d$, viewed also as a trivial comodule,
defines a {\em modular pair in involution} (\cf~\cite{cm3}), and a particular case
of  {\em SAYD module-comodule} (\cf~\cite{hkrs1}).
Such a datum allows to
specialize Connes'  $Ext_\Lba$-definition~\cite{Cext} of cyclic cohomology
 to the context of Hopf algebras
(\cf~\cite{cm2, cm3}). The resulting Hopf
cyclic cohomology, introduced in~\cite{cm2} and extended to SAYD coefficients
in~\cite{hkrs2},
incorporates both Lie algebra and group
cohomology and provides the appropriate cohomological tool
for the treatment of symmetry in noncommutative geometry.
However, its computation for general
Hopf algebras poses quite a challenge.

For the convenience of the reader, we recall below
from~\cite{mr07, mr09} how one can replace
 the standard Hopf cyclic $(b, B)$-complex $(C^\ast (\Hc, \Cb_\d), b, B)$
of~\cite{cm2, cm3}
 by cohomological models which exploit
the bicrossed product structure. In fact,
 such bicocyclic bicomplexes are associated in~\cite[\S 3.2-3.3]{mr09} to
 any bicrossed product Hopf algebra equipped with an SAYD
 module. A bicocyclic bicomplex   is a bigraded module
whose rows and  columns  are  cocyclic modules and, in addition, the
vertical operators and the horizontal operators  commute.

 In the particular case of $\Hc=\Hc (\Gb)$ with
 coefficients in $\Cb_\d$, the first bicocyclic complex $C_\Hc (\Uc, \Fc, \Cb_\d)$
consists of the collection of vector spaces
 \begin{equation}\label{FX}
C_\Hc^{(p,q)}(\Uc, \Fc, \Cb_\d) :=   \Cb_\d \ot_\Hc (\Uc^{\ot p+1}\ot \Fc^{ \ot q+1}) , \quad
(p, q) \in \Zb^+ \times \Zb^+,
\end{equation}
the tensor product over $\Hc$ being taken with respect
to the following action of $\Hc$ on $\Uc$ and $\Fc$ :
\begin{align*}
&(f\acl u)v=\ve(f)uv,\quad (f\acl u)g= fu\rt g, \quad u,v\in\Uc,\quad f,g\in\Fc .
\end{align*}
To simplify the formulas for the boundary operators, we introduce the abbreviated
notation
\begin{align} \label{short}
\begin{split}
&\tilde{u}=u^0\odots  u^p\, , \qquad \qquad
 \tilde{f}=f^0\odots  f^q \, ,\\
 &\td{u}\ns{0}\ot \td{u}\ns{1}= u^1\ns{0}\odots u^p\ns{0}\ot u^1\ns{1}\dots u^p\ns{1}.
  \end{split}
 \end{align}
 With this convention, the horizontal  operators
are defined by

\begin{align}
\begin{split}
&\hd_{i}(\one\ot \td{u} \ot \td{f})= \one\ot u^0\odots
\D(u^i)\odots u^p\ot \td f,\\
&\hd_{p+1}(\one\ot \td{u} \ot \td{f})=\\
&\one \ot u^0\ps{2}\odots
u^p\ot u^0\ps{1}\ns{0} \ot f^0u\ps{1}\ns{1}\odots f^qu\ps{1}\ns{q+1},\\
&\hs_j(\one\ot \td{u} \ot \td{f})= \one\ot  u^0\odots
\ve(u^{j+1})\odots u^p\ot  \td{f},\\
&\hta(\one\ot \td{u} \ot \td{f})=\\
 &\one\ot u^1\odots u^p\ot u^0\ns{0}\ot  f^0u\ns{1}\odots f^qu\ns{q+1} ;
 \end{split}
\end{align}
and the vertical operators
 are defined by
\begin{align}
\begin{split}
 &\vd_{i}(\one\ot \td{u} \ot \td{f})=\one\ot \td u\ot f^0\odots \D(f^i)\odots f^q,\\
&\vd_{q+1}(\one\ot \td{u} \ot \td{f})= \\
&~~\one\ot u^0\ns{0}\odots u^q\ns{0}\ot f^0\ps{2}\odots f^q\ot f^0\ps{1} S(u^0\ns{1}\dots u^p\ns{1}),\\
 &\vs_j(\one\ot \td{u} \ot \td{f})= \one \ot \td u \ot  f^0\odots \ve(f^{j+1})\odots f^q,\\
 &\vta(\one\ot \td{u} \ot \td{f})=\\
 &~~~~~\one\ot u^0\ps{0}\odots u^q\ps{0}\ot f^1\odots f^q\ot f^0 S(u^0\ns{1}\dots u^p\ns{1}).
 \end{split}
\end{align}

The Hopf algebra $\Uc$ admits the following
 action on on $\Fc^{\ot q}$, which plays
a key role in the definition of the next bicocyclic module:
 \begin{align} \nonumber
& u\bullet( f^1\odots f^n)= (1\acl u)\cdot( f^1\odots f^n) :=\\ \label{specact}
&(1\acl u)\ps{1}f^1\odots (1\acl u)\ps{n}f^n=\\ \nonumber
& u\ps{1}\ns{0}\rt f^1\ot u\ps{1}\ns{1}u\ps{2}\ns{0}\rt f^2\odots
u\ps{1}\ns{n-1}\dots u\ps{n-1}\ns{1} u\ps{n}\rt  f^n.
 \end{align}
 It is worth mentioning that,
although this action does not make $\Fc$ a $\Uc$-module coalgebra,
it does commute with the coproduct.

With this understood, the second bicocyclic module $C (\Uc, \Fc, \Cb_\d) $
 is given by the following diagram:

\begin{align}\label{UF}
\begin{xy} \xymatrix{  \vdots\ar@<.6 ex>[d]^{\uparrow B} & \vdots\ar@<.6 ex>[d]^{\uparrow B}
 &\vdots \ar@<.6 ex>[d]^{\uparrow B} & &\\
\Cb_\d \ot \Uc^{\ot 2} \ar@<.6 ex>[r]^{\hb}\ar@<.6
ex>[u]^{  \uparrow b  } \ar@<.6 ex>[d]^{\uparrow B}&
  \Cb_\d\ot \Uc^{\ot 2}\ot \Fc   \ar@<.6 ex>[r]^{\hb}\ar@<.6 ex>[l]^{\hB}\ar@<.6 ex>[u]^{  \uparrow b  }
   \ar@<.6 ex>[d]^{\uparrow B}&\Cb_\d\ot \Uc^{\ot 2}\ot\Fc^{\ot 2}
   \ar@<.6 ex>[r]^{~~\hb}\ar@<.6 ex>[l]^{\hB}\ar@<.6 ex>[u]^{  \uparrow b  }
   \ar@<.6 ex>[d]^{\uparrow B}&\ar@<.6 ex>[l]^{~~\hB} \hdots&\\
\Cb_\d \ot \Uc \ar@<.6 ex>[r]^{\hb}\ar@<.6 ex>[u]^{  \uparrow b  }
 \ar@<.6 ex>[d]^{\uparrow B}&  \Cb_\d \ot \Uc \ot\Fc \ar@<.6 ex>[r]^{\hb}
 \ar@<.6 ex>[l]^{\hB}\ar@<.6 ex>[u]^{  \uparrow b  } \ar@<.6 ex>[d]^{\uparrow B}
 &\Cb_\d\ot \Uc \ot \Fc^{\ot 2}  \ar@<.6 ex>[r]^{~~\hb}\ar@<.6 ex>[l]^{\hB}\ar@<.6 ex>[u]^{  \uparrow b  }
  \ar@<.6 ex>[d]^{\uparrow B}&\ar@<.6 ex>[l]^{~~\hB} \hdots&\\
\Cb_\d  \ar@<.6 ex>[r]^{\hb}\ar@<.6 ex>[u]^{  \uparrow b  }&
\Cb_\d\ot\Fc \ar@<.6 ex>[r]^{\hb}\ar[l]^{\hB}\ar@<.6
ex>[u]^{  \uparrow b  }&\Cb_\d\ot\Fc^{\ot 2}  \ar@<.6
ex>[r]^{~~\hb}\ar@<.6 ex>[l]^{\hB}\ar@<1 ex >[u]^{  \uparrow b  }
&\ar@<.6 ex>[l]^{~~\hB} \hdots& ,}
\end{xy}
\end{align}
whose horizontal maps are given by
\begin{align}
\begin{split}
&\hd_0(\one\ot \td{u}\ot \td{f})= \one\ot 1\ot u^1\ot\dots\ot u^p\ot  \td{f}\\
&\hd_j(\one\ot \td{u}\ot \td{f})= \one\ot u^1\ot\dots\ot
\Delta(u^i)\ot\dots \ot u^p\ot  \td{f}\\
&\hd_{p+1}(\one\ot \td{u}\ot \td{f})=\one\ot u^1\ot \dots \ot
u^p\ot 1\ot \td{f}\\
 &\hs_j(\one\ot \td u\ot \td f)=\one\ot u^1\ot \dots \ot \epsilon(u^{j+1})\ot\dots\ot u^p\ot
 \td{f}\\
&\hta(\one\ot \td{u}\ot  \td{f})=\\
 &\d(u^1\ps{1})\ot S(u^1\ps{3})\cdot(u^2\ot\dots\ot u^p\ot 1)\ot S(u^1\ps{2})\bullet \td{f} ,
 \end{split}
\end{align}
while the vertical operators
 are defined by
\begin{align}
\begin{split}
&\vd_0(\one\ot \td{u}\ot \td{f})= \one\ot \td u\ot  1\ot \td{f},\\
 &\vd_j(\one\ot \td{u}\ot \td{f})= \one \ot \td u \ot    f^1\odots\Delta(f^j)\odots f^q,\\
&\vd_{q+1}(\one\ot \td{u}\ot \td{f})=\one \ot \td u\ns{0}\ot \td{f}\ot S(\td u\ns{1}),\\
 &\vs_j(\one\ot \td u\ot  \td f)=\one\ot \td u\ot  f^1\odots \epsilon(f^{j+1})\odots f^n, \\
 &\vta(m\ot \td{u}\ot  \td{f})=\\
 &\one\ot \td u\ns{0} \ot  S(f^1)\cdot(f^2\odots f^n\ot S(\td u\ns{1})) .
  \end{split}
\end{align}

The next step consists in identifying the standard Hopf cocyclic module $C^\ast(\Hc,
\Cb_\d)$ to the diagonal subcomplex  $D^\ast$ of $C^{\ast,\ast}$.
This is achieved by means of the map
$\, \Psi_{\acl}:  D^\ast \longrightarrow  C^\ast(\Hc,\Cb_\d)$
together with its inverse $\Psi^{-1}_{\acl}:C^\ast(\Hc, \Cb_\d) \longrightarrow D^{\ast}$ .
They are explicitly defined as follows:
 \begin{align}\label{PSI-1}
 \begin{split}
&\Psi_{\acl}(\one\ot u^1\odots u^n\ot f^1\odots f^n)=\\
&\one\ot f^1\acl u^1\ns{0}\ot f^2  u^1\ns{1}\acl u^2\ns{0}\odots \\
& \odots f^n u^1\ns{n-1} \dots u^{n-1}\ns{1}\acl u^n ,
 \end{split}
\end{align}
respectively
 \begin{align}\label{PSI}
 \begin{split}
&\Psi^{-1}_{\acl}(\one\ot   f^1\acl u^1\ot \dots\ot f^n\acl u^n)=\\
& \one\ot u^1\ns{0}\odots u^{n-1}\ns{0}\ot u^n\ot f^1\ot\\
&\ot f^2S(u^1\ns{n-1})\ot f^3S(u^1\ns{n-2}u^2\ns{n-2}) \odots
f^nS(u^1\ns{1} \dots u^{n-1}\ns{1})  .
 \end{split}
\end{align}

With $Tot(C)$ designating the total mixed complex
$\, Tot(C)_n=\bigoplus_{p+q=n} \Cb_\d \ot  \Uc^{p}\ot \Fc^{q}$, we
denote by $\Tot(C)$ the associated \textit{normalized} subcomplex,
obtained by retaining only the elements annihilated by all degeneracy
operators. Its total boundary is $b_T+B_T$, with $b_T$ and $B_T$
defined as follows:
\begin{align} \notag
&\hb_p=\sum_{i=0}^{p+1} (-1)^{i}\hd_i,
&&\vb_q=\sum_{i=0}^{q+1} (-1)^{i}\vd_i, \\ \label{bT}
&b_T=\sum_{p+q=n}\hb_p+\vb_q,\\  \notag
&\hB_p=(\sum_{i=0}^{p-1}(-1)^{(p-1)i}\hta^i)\hs_{p-1}\hta,  &&\vB_q=
(\sum_{i=0}^{q-1}(-1)^{(q-1)i}\vta^i)\hs_{q-1}\vta,\\ \label{BT}
&B_T=\sum_{p+q=n}\hB_p+\vB_q.\
\end{align}

The total complex of a bicocyclic module is a mixed complex, and so
 its cyclic cohomology is well-defined. By means of the analogue of
 the Eilenberg-Zilber theorem for bi-paracyclic
modules~\cite{gj, kr2}, the diagonal mixed complex
$(D, b, B)$ and
the total mixed complex $(\Tot C,b_T, B_T)$ can be seen to
be quasi-isomorphic in both
Hochschild and cyclic cohomology.
At the level of Hochschild cohomology the quasi-isomorphism is
implemented by the Alexander-Whitney map
$\, AW:= \bigoplus_{p+q=n} AW_{p,q}: \Tot(C)^n\ra D^n $,
\begin{align}\label{AW}
\begin{split}
&AW_{p,q}: \Cb_\d\ot \Uc^{\ot p}\ot \Fc^{\ot q}\longrightarrow \Cb_\d\ot  \Uc^{\ot
p+q}\ot \Fc^{\ot p+q} \\
&AW_{p,q}=(-1)^{p+q}\underset{q\,\text{times}}{\underbrace{\vd_0\vd_0\dots
\vd_0}}\hd_n\hd_{n-1}\dots \hd_{q+1} \, .
\end{split}
\end{align}
Using a standard homotopy operator $H$, this can be supplemented by
a cyclic Alexander-Whitney map
$AW':= AW\circ B\circ H:  D^n\ra \Tot(C)^{n+2}$, and thus
upgraded to an $S$-map $\overline{AW} = (AW, AW')$,
 of mixed complexes. The inverse quasi-isomorphisms are
provided by the shuffle maps $\, Sh:= D(C)^n\ra \Tot(C)^n$, resp.
$\overline{Sh} = (Sh, Sh')$, which are discussed in detail in~\cite{kr2}.

  We refer the reader to
~\cite[Theorem 3.16 and \S 3.3]{mr09} for a complete proof of
 the statement recorded below.

\begin{proposition} \label{mix}
The mixed complex $(\Tot C^\ast (\Uc, \Fc, \Cb_\d), b_T, B_T)$ is quasi-isomorphic to
the standard total Hopf cyclic complex $(\Tot C^\ast (\Hc, \Cb_\d), b, B)$.
\end{proposition}

%%%%%%%%%%%%%%%%%%%%

The bicocyclic module \eqref{UF} is an amalgam between the Lie algebra homology
of $\Fg$ with coefficients in $\Cb_\d\ot \Fc^{\ot \bullet}$
 and Hopf cyclic cohomology of $\Fc$ with
coefficients in the comodule $\Uc(\Fg)^{\ot \bullet}$.
It can be further reduced to  the bicomplex
$C^{\bullet, \bullet}(\Fg,\Fc, \Cb_\d):=\Cb_\d\ot \wedge^{\bullet}\Fg\ot \Fc^{\ot \bullet} $,
obtained by replacing the tensor algebra
 of $\Uc (\Fg)$ with the exterior algebra of  $\Fg$. To this end,
 recall the action \eqref{specact}, which restricted to $\Fg$ reads
 \begin{align} \label{specactg}
&X\bullet (f^1 \ot \cdots \ot f^q) = (\one \acl X)(\td f)= \\ \notag
&X\ps{1}\ns{0}\rt f^1\ot X\ps{1}\ns{1}(X\ps{2}\ns{0}\rt f^2)\odots X\ps{1}\ns{q-1}\dots X\ps{q-1}\ns{1}(X\ps{q}\rt f^q),
\end{align}
and define
 $\,  \p_\Fg : \Cb_\d\ot \wg^p\Fg\ot\Fc^{\ot q} \ra \Cb_\d\ot \wg^{p-1}\Fg\ot\Fc^{\ot q} $
as being the Lie algebra homology boundary with respect to the action of
$\Fg$ on $\Cb_\d\ot\Fc^{\ot q}$ defined by
\begin{equation} \label{Liebdact}
(\one \ot\td f)\lt X= \d(X)\ot \td{f}-\one \ot X\bullet \td f .
\end{equation}
Via the antisymmetrization map
\begin{align} \label{antsym1}
&\td \a_\Fg:\Cb_\d\ot
\wg^q\Fg \ot \Fc^{\ot p}\ra \Cb_\d\ot \Uc^{\ot q} \ot \Fc^{\ot p} ,
\qquad \td\a_\Fg= \a\ot  \Id, \\ \notag
&\a(\one \ot X^1\wdots X^p)= \frac{1}{p!} \sum_{\s\in S_p}(-1)^\s  \one \ot X^{\s(1)}\odots X^{\s(p)} ,
\end{align}
the pullback of the vertical $b$-coboundary in \eqref{UF}  vanishes, while
the vertical $B$-coboundary becomes
$\p_\Fg$ (\cf~\cite[Proposition 7]{cm2}).

On the other hand, since $\Fc$ is commutative, the coaction \eqref{gij} extends
 from $\Fg$ to
 a unique coaction $\Db_\Fg:\wg^p\Fg\ra \wg^p\Fg\ot \Fc$ satisfying
\begin{align} \label{wgcoact}
\Db_\Fg(X^1\wdots X^q)\,=\, X^1\ns{0}\wdots X^q\ns{0}\ot
X^1\ns{1}\dots X^q\ns{1} .
\end{align}
\medskip
The coboundary $b_\Fc$ is given by
\begin{align} \label{b-cobd}
\begin{split}
&b_\Fc(\one\ot\a\ot   f^1\odots f^q )=\one \ot\a\ot  1\ot f^1\odots f^q+\\
&\sum_{1\ge i\ge q} (-1)^i \one \ot\a\ot  f^1\odots \D(f^i)\odots f^q+\\
&(-1)^{q+1}\one\ot\a\ns{0}\ot   f^1\odots f^q\ot  S(\a\ns{1})  ,
\end{split}
\end{align}
while  the $B$-boundary is
\begin{align} \label{B-bd}
& B_\Fc= \left(\sum_{i=0}^{q-1}(-1)^{(q-1)i}\tau_\Fc^{i}\right) \s \tau_\Fc , \qquad \qquad
\text{where} \\ \notag
&\tau_\Fc(\one\ot \a\ot f^1\odots f^q)= \one\ot\a\ns{0}\ot
S(f^1)\cdot(f^2\odots f^q\ot S(\a\ns{1})) ,\\ \notag
&\s(\one\ot\a\ot f^1\odots f^q)= \ve(f^q)\ot\a\ot f^1\odots f^{q-1} .
\end{align}
Thus we arrive at the bicomplex $C^{\bullet, \bullet}(\Fg, \Fc, \Cb_\d)$, described by
the diagram
\begin{align}\label{UF+}
\begin{xy} \xymatrix{  \vdots\ar[d]^{\p_\Fg} & \vdots\ar[d]^{\p_\Fg}
 &\vdots \ar[d]^{\p_\Fg} & &\\
\Cb_\d \ot \wg^2\Fg \ar@<.6 ex>[r]^{b_{\Fc}} \ar[d]^{\p_\Fg}&\ar@<.6 ex>[l]^{~~B_\Fc}
 \Cb_\d\ot \wg^2\Fg\ot\Fc  \ar@<.6 ex>[r]^{b_{\Fc}} \ar[d]^{\p_\Fg}
 &\ar@<.6 ex>[l]^{~~B_\Fc}\Cb_\d\ot \wg^2\Fg\ot\Fc^{\ot 2} \ar[r]^{~~~~~~b_{\Fc}}   \ar[d]^{\p_\Fg}& \ar@<.6 ex>[l]^{~~B_\Fc}\hdots&\\
\Cb_\d \ot \Fg \ar@<.6 ex>[r]^{b_{\Fc}} \ar[d]^{\p_\Fg}&\ar@<.6 ex>[l]^{~~B_\Fc}  \Cb_\d\ot \Fg\ot\Fc \ar@<.6 ex>[r]^{b_{\Fc}} \ar[d]^{\p_\Fg}&\ar@<.6 ex>[l]^{~~B_\Fc}\Cb_\d\ot \Fg\ot\Fc^{\ot 2}
   \ar[d]^{\p_\Fg} \ar[r]^{~~~~~~~b_{\Fc}}&\ar@<.6 ex>[l]^{~~B_\Fc} \hdots&\\
\Cb_\d \ot \Cb \ar@<.6 ex>[r]^{b_{\Fc}}& \ar@<.6 ex>[l]^{~~B_\Fc} \Cb_\d\ot \Cb\ot\Fc \ar@<.6 ex>[r]^{b_{\Fc}}&\ar@<.6 ex>[l]^{~~B_\Fc}\Cb_\d\ot \Cb\ot\Fc^{\ot 2} \ar[r]^{~~~~~b_{\Fc}} & \ar@<.6 ex>[l]^{~~B_\Fc}\hdots&  . }
\end{xy}
\end{align}
Actually, since $\Fc$ is commutative and its action on $\Cb_\d\ot \wg^q\Fg$ is trivial,
by~\cite[Theorem 3.22]{kr1}  $B_\Fc$ vanishes in Hochschild cohomology
and therefore can be omitted from the above diagram.

Referring to~\cite[Prop. 3.21 and \S 3.3]{mr09} for more details, we state the
conclusion as follows.

\begin{proposition} \label{mixCE}
The map \eqref{antsym1} induces a quasi-isomorphism
between the total complexes  $\Tot C^{\bullet, \bullet}(\Fg, \Fc, \Cb_\d)$
and $\Tot C^{\bullet, \bullet} (\Uc, \Fc, \Cb_\d) $.
\end{proposition}

\bigskip

\subsection{Passage to Lie algebra cohomology} \label{Ss: H-C-E}

In order to convert the Lie algebra homology into Lie algebra cohomology,
 we shall resort to the Poincar\'e isomorphism
\begin{align}\label{theta}
\begin{split}
&\FD_\Fg =  \Fd_\Fg \ot \Id : \wedge^q\Fg^\ast\ot \Fc^{\ot p} \ra
  \wedge^{m-q}\Fg \ot  \wedge^{m}\Fg^\ast \ot \Fc^{\ot p} \\
 &\Fd_\Fg(\eta) \, = \,   \iota(\eta) \varpi \ot \varpi^\ast ,
 \end{split}
\end{align}
where $\varpi$ is a covolume element and
$\varpi^\ast $ is the
dual volume element. The contraction operator is defined as follows:
for $\lambda \in \Fg^\ast $,
$\iota(\lambda) : \wedge^{\bullet}\Fg \ra \wedge^{\bullet-1}\Fg$ is
the unique derivation of degree $-1$ which on $\Fg$ is the evaluation map
\begin{align*}
\iota(\lambda) (X) \, = \, \langle \lambda, X \rangle , \qquad \fl \, X \in \Fg ,
\end{align*}
while for $\eta = \lambda_1 \wdots \lambda_q \in \wedge^{q}\Fg^\ast$,
 $\iota(\eta):\wedge^{\bullet}\Fg \ra \wedge^{\bullet-q}\Fg$ is given by
\begin{align*}
\iota (\lambda_1 \wdots \lambda_q) := \iota (\lambda_q) \circ \ldots
 \circ \iota(\lambda_1) , \qquad \fl \, \lambda_1, \ldots , \lambda_q \in \Fg^\ast .
\end{align*}
Noting that the coadjoint action of $\Fg$ induces on $ \wedge^{m}\Fg^\ast$ the action
\begin{align}
\ad^\ast (X) \varpi^\ast = \d (X) \varpi^\ast , \qquad \fl \, X \in \Fg ,
\end{align}
we shall identify $\wedge^{m}\Fg^\ast$ with $\Cb_\d$ as $\Fg$-modules.

%%%%%%%%%%%%%%%%%%%%%%%%%%%%%%%%%%%%%%
To define the coaction of $\Fc$ on $\wg^\bullet\Fg^\ast$, we proceed in a manner similar to
the definition of its coaction on $\Uc (\Fg)$. First, we identify  $\wg^\bullet\Fg^\ast$
with the graded algebra $\Om^\bullet (G)^G$ of left-invariant differential forms on $G$.
We fix a basis
$\{\om_1 , \ldots , \om_m\}$ of $\Om^1 (G)^G$,
and denote by $\{\om_A = \om_{a_1} \wdots  \om_{a_p} \, ;
a_1< \ldots < a_p \}$ the corresponding basis of  $\Om^p (G)^G$. Given any $\om \in \wg^p \Fg^\ast \equiv
\Om^p (G)^G$, we define a polynomial function on $N$ with values in  $\wg^p \Fg^\ast$
by restricting the (not necessarily invariant) form $ \psi^\ast (\om)$ to the tangent space at $e
\in G$:
   \begin{equation} \label{coact3}
(\Db_{\Fg^\ast} \om) (\psi) \, = \,  (\psi^{-1})^\ast (\om)\mid_e  , \qquad  \psi \in N .
 \end{equation}
Explicitly,  ${\Db_{\Fg^\ast}} \om  \in \wg^p \Fg^\ast \ot \Fc (N)$ is defined as follows.
 Express $\om \in \wg^p \Fg^\ast$  in the form
 \begin{equation} \label{coact1}
(\psi^{-1})^\ast (\om) \, = \, \sum_B  \g_\om^B (\psi)\, \om_B, \qquad \psi \in N ,
 \end{equation}
with $\g_A^B  (\psi)$ in
  the algebra of functions on $G$ generated by $\{ \g^\bullet_{\bullet \ldots
  \bullet} (\psi) \}$.
Then
 \begin{equation} \label{coact2}
 \Db_{\Fg^\ast} \om = \sum_B\om_B  \ot \eta_\om^B  , \qquad \text{where} \quad
  \eta_\om^B(\psi) = \g_\om^B  (\psi) (e) .
 \end{equation}
%%%%%%%%%%%%%%%%%%%%%%%%%

%In particular, in view of  \eqref{omdisp} and of naturality of the forms $\theta^k$ are canonical, one has
%\begin{align}\label{omthetacoact}
% &\Db_{\Fg^\ast}(\theta^k)=\theta^k\ot1\\\notag
%&\Db_{\Fg^\ast}(\om^i_j)=\om^i_j\ot 1-\sum_k \theta^k\ot \eta^i_{jk}
%\end{align}
%%%%%%%%%%%%%%%%%%%%%%%%

\begin{lemma}
\begin{equation}\label{F-coaction-g*}
\Db_{\Fg^\ast}:\wg^p\Fg^\ast\ra \wg^p \Fg^\ast \ot \Fc (N),
\end{equation} defines a right coaction.
\end{lemma}

\proof On applying \eqref{coact1} twice one obtains
\begin{align*} \notag
&(\psi_1^{-1})^{\ast} \left((\psi_2^{-1})^{\ast}(\om)\right) \,  =\,(\psi_1^{-1})^{\ast}
 \big(\sum_B  \g_\om^B (\psi_2)\, \om_B\big) \\
 &\qquad \qquad =
 \,  \sum_{C} \left(\sum_{B} \big(\g_\om^B (\psi_2)\circ (\psi_1^{-1})\big) \, \g_B^C (\psi_1)\right) \, \om_C  .
  \end{align*}
On the other hand,
 \begin{align*}
\left((\psi_1\circ  \psi_2)^{-1}\right)^\ast  (\om)  \, =\, \sum_C  \g_\om^C(\psi_1 \psi_2)\, \om_C .
\end{align*}
Equating the two results, one obtains
\begin{align*}
 \g_\om^C (\psi_1  \psi_2) \, =\,\sum_{B} \left(\g_\om^B (\psi_2)\circ (\psi_1^{-1})\right) \, \g_B^C (\psi_1) ,
\end{align*}
which by evaluation at $e \in G$ yields
\begin{align*}
&\D \eta_\om^C \, = \, \eta_B^C \ot \eta_\om^B .
\end{align*}
In turn, this gives
\begin{align*}
(\Db_{ \Fg^\ast}  \ot \Id) ( \Db_{\Fg^\ast} \om)& \, = \, (\Db_{ \Fg^\ast}  \ot \Id)
 \sum_B  \om_B \ot \eta_\om^B  \\
 &=  \sum_{B, C} \om_C \ot  \eta_B^C \ot \eta_\om^B
\, = \, \sum_{C}  \om_C \ot \D \left(\eta_\om^C\right) \,
 =\, (\Id \ot \D)  ( \Db_{\Fg^\ast} \om ) .
\end{align*}
 \endproof

\begin{remark}\label{remark} \label{ractg*}
{\rm The above coaction
gives rise to a left action
 $ \rt : \wg^p\Fg^\ast\ts N\ra \wg^p\Fg^\ast$
of $N$ on $\wg^\bullet\Fg^\ast$, via
\begin{align}\label{Naction}
{\Db_{\Fg^\ast}} \om = \om\ns{0} \ot \om\ns{1}\quad \Longrightarrow \quad \psi\rt \om\,
=\, \om\ns{1}(\psi) \, \om\ns{0} , \quad \fl \, \psi \in  N
\end{align}
This action coincides in fact with the natural action of $N$ on $\wg^\bullet\Fg^\ast$, that is
 \begin{equation}  \label{natcoact}
 \psi \rt \om \, = \, (\psi^{-1})^\ast  \om \mid_e.
 \end{equation}
 }
 \end{remark}
Indeed, by \eqref{coact2}, for any basis element one has
 \begin{equation*}
\psi \rt \om_A = \sum_B  \g_A^B  (\psi^{-1}) (e) \om_B \,
  = \, (\psi^{-1})^\ast  \om_A \mid_e.
 \end{equation*}

\begin{lemma} The coactions $\Db_{\Fg^\ast}:\wg^\bullet \Fg^\ast\ra \wg^\bullet\Fg^\ast\ot \Fc$ and
$\Db_\Fg:\wg^\bullet \Fg\ra \wg^\bullet\Fg\ot \Fc$ are Poincar\'e dual to each other, {\rm i.e.}
for any $\z \in \wg^\bullet \Fg^\ast$,
\begin{align}\label{PDcoaction}
\Db_{\Fg^\ast}(\z)=\FD_\Fg^{-1}\circ\Db_\Fg\circ\FD_\Fg(\z)
\end{align}
\end{lemma}

\proof This follows from the obvious identity
\begin{align} \label{dualaction}
 \langle (\psi^{-1})^\ast (\om )\mid_e , \psi_{\ast e} (X) \rangle
 \, = \,  \langle \om  , X \mid_e \rangle , \qquad
 \om \in \Fg^\ast , \, X \in \Fg , \, \psi \in N ,
\end{align}
combined with the fact that any covolume element $\varpi \in \wg^{m} \Fg$ is  $\Gb$-invariant,
\begin{align*}
  \vp^\ast (\varpi) = \varpi , \qquad
 \fl \, \vp \in  \Gb \sbs \Diff (\Rb^n).
\end{align*}
 \endproof

\begin{remark}\label{Db-inverse}
{\rm The equality \eqref{dualaction} together with the definitions \eqref{gij} and
\eqref{coact1} imply}
\begin{align*}
\begin{split}
 \langle \om , X \rangle \, &= \, \langle (\psi^{-1})^\ast (\om )\mid_e , \psi_{\ast e} (X) \rangle
\, = \, \langle \sum_B  \eta_\om^B (\psi)\, \om_B , \sum_C \eta_X^C (\psi) X_C \rangle \\
 &= \, \sum_A  \eta_\om^A (\psi) \, \eta_X^A (\psi)  .
\end{split}
\end{align*}
{\rm This shows that for $\om$ and $X$ dual to each other the corresponding matrices
$\left(\eta^A_{\om, B} \right)$ and $\left(\eta^A_{X, B} \right)$ are
inverse to each other.}
\end{remark}

By transfer of structure,  the bicomplex  \eqref{UF+} becomes
$(C^{\bullet,\bullet}(\Fg^\ast,  \Fc) ,\p_{\Fg^\ast}, b^\ast_\Fc)$,
 \begin{align}\label{UF+*}
\begin{xy} \xymatrix{  \vdots & \vdots
 &\vdots &&\\
 \wdg^2\Fg^\ast  \ar[u]^{\p_{\Fg^\ast}}\ar@<.6 ex>[r]^{b^\ast_\Fc~~~~~~~}& \ar@<.6 ex>[l]^{~~B^\ast_\Fc} \wdg^2\Fg^\ast\ot\Fc \ar[u]^{\p_{\Fg^\ast}} \ar@<.6 ex>[r]^{b^\ast_\Fc}& \ar@<.6 ex>[l]^{~~B^\ast_\Fc}\wdg^2\Fg^\ast\ot\Fc^{\ot 2} \ar[u]^{\p_{\Fg^\ast}} \ar@<.6 ex>[r]^{~~~~~~~~~b^\ast_\Fc} & \ar@<.6 ex>[l]^{~~B^\ast_\Fc}\hdots&  \\
 \Fg^\ast  \ar[u]^{\p_{\Fg^\ast}}\ar@<.6 ex>[r]^{b^\ast_\Fc~~~~~}& \ar@<.6 ex>[l]^{~~B^\ast_\Fc} \Fg^\ast\ot\Fc \ar[u]^{\p_{\Fg^\ast}} \ar@<.6 ex>[r]^{b^\ast_\Fc}& \ar@<.6 ex>[l]^{~~B^\ast_\Fc} \Fg^\ast\ot \Fc^{\ot 2} \ar[u]^{\p_{\Fg^\ast}} \ar@<.6 ex>[r]^{~~~~~b^\ast_\Fc }&\ar@<.6 ex>[l]^{~~B^\ast_\Fc} \hdots&  \\
 \Cb  \ar[u]^{\p_{\Fg^\ast}}\ar@<.6 ex>[r]^{b^\ast_\Fc~~~~~~~}& \ar@<.6 ex>[l]^{~~B^\ast_\Fc} \Cb\ot \Fc \ar[u]^{\p_{\Fg^\ast}}\ar@<.6 ex>[r]^{b^\ast_\Fc}& \ar@<.6 ex>[l]^{~~B^\ast_\Fc} \Cb\ot \Fc^{\ot 2} \ar[u]^{\p_{\Fg^\ast}} \ar@<.6 ex>[r]^{~~~~~b^\ast_\Fc} &\ar@<.6 ex>[l]^{~~B^\ast_\Fc} \hdots&. }
\end{xy}
\end{align}
 The vertical coboundary $\p_{\Fg^\ast}:C^{p,q}\ra C^{p,q+1}$ is the Lie algebra
cohomology coboundary of the Lie algebra $\Fg$ with coefficients in $\Fc^{\ot p}$,
 where the action of $\Fg$ is
\begin{equation}
( f^1\odots f^p) \blacktriangleleft X=-X\bullet(f^1\odots f^p).
\end{equation}
The horizontal $b$-coboundary $b^*_\Fc$, resp.
the horizonal  $B$-boundary $B_\Fc^\ast$  are
 the Hochschild coalgebra cohomology coboundary
 resp. the $B$-boundary operator of $\Fc$  with
coefficients in the comodule $\wedge^q\Fg^\ast$ with respect to the coaction of $\Fc$ defined by
\eqref{coact2}, given by formulae similar to \eqref{b-cobd}, resp. \eqref{B-bd}.

For future reference, we record the conclusion.

\begin{proposition} \label{mixCE*}
The map \eqref{theta} induces a quasi-isomorphism
between the total complexes  $\Tot C^{\bullet, \bullet} (\Uc, \Fc, \Cb_\d) $
and $\Tot C^{\bullet, \bullet} (\Fg^\ast,  \Fc) $.
\end{proposition}
\medskip

\subsection{$\Fg$-equivariant Hopf cyclic interpretation} \label{Ss:gH}

We shall now give a genuine Hopf cyclic interpretation to
the bicomplex \eqref{UF+*}, by recasting it as a
 $\Fg$-equivariant Hopf cyclic complex for $\Fc$ with coefficients in a graded
 SAYD built out of the Koszul resolution.

We recall that the Koszul resolution associated to the  Lie algebra $\Fg$ is the complex
\begin{equation}
\xymatrix{
V(\Fg^\ast):& \Uc\ar[r]^{\p_K} &\Fg^\ast\ot \Uc\ar[r]^{\p_K}&\wdg^2\Fg^\ast\ar[r]^{\p_K}\ot \Uc&\dots  ,}
\end{equation}
where the coboundary $\p_K$ is the usual
Lie algebra cohomology coboundary of $\Fg$ with coefficients in $\Uc$,
and $\Fg$ acts on $\Uc$ from the right via
\begin{equation}\label{Koszul-coplex}
 \lt:\Uc\ot\Fg\ra \Uc, \quad u\lt X=-Xu, \quad X\in \Fg, \;\;u\in \Uc.
\end{equation}

We endow $V^\oplus(\Fg^\ast) = \wdg^\bullet \Fg^\ast \ot \Uc $ with a coaction
$\Db_K:V^\oplus(\Fg^\ast)\ra  \Fc\ot V^\oplus(\Fg^\ast)$ defined as follows:
\begin{align}\label{Db-Koszul}
 \Db_K(\om\ot u)= S(u\ns{0}\ps{2})\rt S(u\ns{1}\om\ns{1})\ot\om\ns{0}\ot u\ns{0}\ps{1} .
\end{align}
Equivalently, using \eqref{Fcoact} we can rephrase the definition as
\begin{equation}\label{Db-Koszul-=}
\Db_K(\om\ot u)= S(u\ps{2}\ns{0})\rt S(u\ps{1}\ns{1}u\ps{2}\ns{1}\om\ns{1})\ot\om\ns{0}\ot u\ps{1}\ns{0} ;
\end{equation}
here $\om\ns{0}\ot\om \ns{1}$ refers to
 the same coaction of $\Fc$ on $\wdg^\bullet\Fg^\ast$ used in \eqref{F-coaction-g*}.
\begin{lemma}
The map $\Db_K:V^\oplus(\Fg^\ast)\ra  \Fc\ot V^\oplus(\Fg^\ast)$ is a coaction.
\end{lemma}
\begin{proof}
We need to check that
$(\id\ot \Db_K)\circ \Db_K= (\D\ot \id)\circ \Db$.
Using  \eqref{Fcoact} we can write
\begin{align*}
\begin{split}
&(\id\ot \Db_K)\circ \Db_K(\om\ot u)=\\
&(\id\ot \Db_K)(S(u\ns{0}\ps{2})\rt S(u\ns{1}\om\ns{1})\ot\om\ns{0}\ot u\ns{0}\ps{1})=\\
&S(u\ns{0}\ps{2})\rt S(u\ns{1}\om\ns{2})\ot S(u\ns{0}\ps{1}\ns{0}\ps{2})\rt S(u\ns{0}\ps{1}\ns{1}\om\ns{1})  \ot\\
& \om\ns{0}\ot u\ns{0}\ps{1}\ns{0}\ps{1}=\\
& S(u\ns{0}\ps{3})\rt S(u\ns{1}\om\ns{2})\ot S(u\ns{0}\ps{2}\ns{0})\rt S(u\ns{0}\ps{1}\ns{1} u\ns{0}\ps{2}\ns{1}\om\ns{1})  \ot\\
& \om\ns{0}\ot u\ns{0}\ps{1}\ns{0}=\\
& S(u\ps{3}\ns{0})\rt S(u\ps{1}\ns{2}u\ps{2}\ns{2}u\ps{3}\ns{1}\om\ns{2})\ot S(u\ps{2}\ns{0})\rt S(u\ps{1}\ns{1}u\ps{2}\ns{1} \om\ns{1})  \ot\\
& \om\ns{0}\ot u\ps{1}\ns{0}.
\end{split}
\end{align*}
On the other hand, using \eqref{Db-Koszul} and Lemma \ref{antipode} we see that
\begin{align*}
\begin{split}
&(\D\ot \id)\circ \Db_K(\om\ot u)= \\
&(\D\ot \id)(S(u\ps{2}\ns{0})\rt S(u\ps{1}\ns{1}u\ps{2}\ns{1}\om\ns{1})\ot\om\ns{0}\ot u\ps{1}\ns{0})=\\
&S(u\ps{2}\ns{0}\ps{2})\ns{0}\rt S(u\ps{1}\ns{2}u\ps{2}\ns{2}\om\ns{2})\ot\\
&S(u\ps{2}\ns{0}\ps{2})\ns{1} S(u\ps{2}\ns{0}\ps{1})\rt S(u\ps{1}\ns{1}u\ps{2}\ns{1}\om\ns{1})\ot \om\ns{0}\ot u\ps{1}\ns{0}=\\
&S(u\ps{2}\ns{0}\ps{3}\ns{0})\rt S(u\ps{1}\ns{2}u\ps{2}\ns{2}\om\ns{2})\ot\\
&(S(u\ps{2}\ns{0}\ps{2})\rt\big( u\ps{2}\ns{0}\ps{3}\ns{1}S(u\ps{1}\ns{1}u\ps{2}\ns{1}\om\ns{1})\big)\ot \om\ns{0}\ot u\ps{1}\ns{0}=\\
& S(u\ps{3}\ns{0})\rt S(u\ps{1}\ns{2}u\ps{2}\ns{2}u\ps{3}\ns{1}\om\ns{2})\ot S(u\ps{2}\ns{0})\rt S(u\ps{1}\ns{1}u\ps{2}\ns{1} \om\ns{1})  \ot\\
& \om\ns{0}\ot u\ps{1}\ns{0}.
\end{split}
\end{align*}
\end{proof}

Since $\Fc$ is commutative,
$V^\oplus(\Fg^\ast)$ endowed
with the trivial $\Fc$-action can be viewed as an SAYD over $\Fc$,
hence its Hopf cyclic complex $C^{\bullet}(\Fc, V^\bullet(\Fg^\ast))$
is well-defined \cite{hkrs1}.
On the other hand, recall that $\Uc(\Fg)$ acts on $\Fc^{\ot q}$ via the action $\bullet$
 defined in \eqref{specactg}. We denote the graded module $V^\oplus(\Fg^\ast)\ot_{\Uc}\ot \Fc^{\ot q}$
 by $C^q_\Uc(\Fc,V^\oplus(\Fg^\ast))$.

\begin{proposition}
The cyclic structure of $C^{\bullet}(\Fc, V^\oplus(\Fg^\ast))$ descends to the quotient
complex $C^{\bullet}_\Uc(\Fc, V^\oplus(\Fg^\ast))$.
\end{proposition}

\begin{proof}
The action $\bullet$ commutes with comultiplication of $\Fc$,
 although  $\Fc$ is not  $U(\Fg)$-module coalgebra. This ensures that  $\bullet$
 commutes with all cofaces of the Hopf cyclic complex  $C^{\bullet}(\Fc, V^\oplus(\Fg^\ast))$,  except possibly  the very last one. On the other hand, using
 the fact that $\ve(u\rt f)=\ve(u)\ve(f)$ one sees that $\bullet $ also commutes with codegeneracies.
It remains to check that $\bullet$ commutes with $\tau$, more precisely to show that
 \begin{equation*}
 \tau(\om\ot u\ot f^1\odots f^q)=\tau(\om\ot 1\ot u\bullet(f^1\odots f^q)).
 \end{equation*}
By the very definition,
 \begin{align*}
 \begin{split}
 &\tau(\om\ot u\ot f^1\odots f^q)=\\
 &\om\ns{0}\ot_\Uc u\ns{0}\ps{1}\ot S(f^1)\cdot(f^2\odots f^q \ot S(u\ns{0}\ps{2})\rt S(u\ns{1}\om\ns{1})).
 \end{split}
 \end{align*}
 On the other hand  by using \eqref{Fcoact}, \eqref{mp2}, and Lemma \ref{S-linear},
  we see that
 \begin{align*}
 \begin{split}
 &\tau(\om\ot 1\ot u\bullet(f^1\odots f^q))=\\
 &\om\ns{0}\ot 1\ot_\Uc S(u\ps{1}\ns{0}\rt f^1))\cdot\big( u\ps{1}\ns{1} u\ps{2}\ns{0}\rt f^2\ot\dots\\
  &~~~~~~~~~~~~~~~~~~~~~~~\dots \ot u\ps{1}\ns{q-1}\dots u\ps{q-1}\ns{1} u\ps{1}\ns{0}\rt f^q \ot S(\om\ns{1})\big)=\\
 &\om\ns{0}\ot 1\ot_\Uc (u\ps{1}\ns{0}\rt S(f^1))\cdot\big( u\ps{2}\ns{0}\rt f^2\ot\dots\\
  &~~~~~~~~~~~~~\dots \ot u\ps{2}\ns{q-2}\dots u\ps{q-1}\ns{1} u\ps{1}\ns{0}\rt f^q \ot S(u\ps{1}\ns{1})S(\om\ns{1})\big)=\\
  &\om\ns{0}\ot 1\ot_\Uc u\ps{1}\ns{0}\ps{1}\ns{0}\rt S(f^1\ps{q}) u\ps{2}\ns{0}\rt f^2\ot\dots\\
  &\dots \ot u\ps{1}\ns{0}\ps{1}\ns{q-1}\dots u\ps{1}\ns{0}\ps{q-1}\ns{1} u\ps{1}\ns{0}\ps{q}\ns{0}\rt S(f^1\ps{2}) \\
  &u\ps{1}\ns{0}\ps{1}\ns{0}\rt S(f^1\ps{q}) u\ps{2}\ns{q-2}\dots u\ps{q-1}\ns{1} u\ps{1}\ns{0}\rt f^q \ot \\
  &u\ps{1}\ns{0}\ps{1}\ns{q}\dots u\ps{1}\ns{0}\ps{q}\ns{1} u\ps{1}\ns{0}\ps{q+1}\rt  S(u\ps{1}\ns{1})S(\om\ns{1})=\\
  &\om\ns{0}\ot 1\ot_\Uc u\ns{0}\ps{1}\bullet(S(f^1)\cdot(f^2\odots f^q\ot S(u\ns{0}\ps{2})\rt
S(u\ns{1}\om\ns{1}))) .
   \end{split}
 \end{align*}
 \end{proof}

As a result, via the induced structure,
 the graded module $C^q_\Uc(\Fc, V^{\oplus}(\Fg^\ast))$ becomes a cocyclic module.
 To emphasize the differential graded structure of the coefficients
 $V^{\oplus}(\Fg^\ast)$, we
 now switch to denoting it $V^{\bullet}(\Fg^\ast)$. Accordingly,
 $C^\bullet_\Uc(\Fc, V^{\bullet}(\Fg^\ast))$ will be viewed as a $(b, B)$-cyclic
 complex
\begin{equation}
\xymatrix{
  V^{\bullet}(\Fg^\ast)\ot _\Uc\Cb \ar@<.6 ex>[r]^{b_{\Fc,K}}&\ar@<.6 ex>[r]^{b_{\Fc,K}} \ar@<.6 ex>[l]^{B_{\Fc,K}} V^{\bullet}(\Fg^\ast)\ot_\Uc \Fc& \ar@<.6 ex>[l]^{B_{\Fc,K}}\ar@<.6 ex>[r]^{~~~~~b_{\Fc,K}} V^{\bullet}(\Fg^\ast)\ot_\Uc \Fc^{\ot 2}&  \ar@<.6 ex>[l]^{~~~~~~~B_{\Fc,K}}\cdots
}
\end{equation}
 with the induced boundaries
\begin{align*}
&b_{\Fc,K}: C^q_\Uc(\Fc, V^{p}(\Fg^\ast))\ra C^{q+1}_\Uc(\Fc, V^{p}(\Fg^\ast)),\\
&B_{\Fc,K}: C^q_\Uc(\Fc, V^{p}(\Fg^\ast))\ra C^{q-1}_\Uc(\Fc, V^{p}(\Fg^\ast)),
\end{align*}
supplemented by the Koszul coboundary
\begin{align*}
\p_{\Fc,K} :=
  \p_K\ot \id_\Fc: C^q_\Uc(\Fc, V^{p}(\Fg^\ast))\ra C^{q+1}_\Uc(\Fc, V^{p+1}(\Fg^\ast)).
\end{align*}

\begin{theorem} \label{thm:dgH}
$C^\bullet_\Uc(\Fc, V^{\bullet}(\Fg^\ast))$ is
a $\Fg$-equivariant Hopf cyclic complex of $\Fc$ with coefficients in the
differential graded SAYD  $V^\bullet(\Fg^\ast)$, and
the map $\kappa:\; C^\bullet_\Uc(\Fc, V^\bullet(\Fg^\ast))\ra C^\bullet(\Fc,\Fg^\ast)$ defined by
\begin{align*}
 \kappa(\om\ot  u\ot_\Uc f^1\odots f^q)= \om\ot u\bullet(f^1\odots f^q) ,
\end{align*}
induces an isomorphism of total complexes.
\end{theorem}
\begin{proof}
Indeed, the inverse map $\kappa^{-1}$ is given by
\begin{equation*}
\kappa^{-1}(\om\ot \td f)= \om\ot 1\ot_\Uc \td f.
\end{equation*}
Evidently, $\kappa^{-1}$ commutes with the cyclic structure. In addition,
\begin{align*}
&\kappa(\p_{\Fc, K}(\om\ot u\ot_\Uc\td f))=\kappa(\p\om\ot u\ot_\Uc \td f-
\sum_i \om\wdg \t^i\ot X_iu\ot_\Uc \td f)=\\
&\p\om \ot u\bullet\td f-\sum  \om\wdg \t^i\ot  X_iu\bullet \td f=\p_{\Fg^\ast}(\kappa(\om\ot u\ot_\Uc \td f)).
\end{align*}
\end{proof}

  \subsection{Linkage with group cohomology} \label{Ss: Group C}

Recalling that $\Fc$ can be viewed as an algebra of regular functions on $N$,
it makes sense to try to relate the complex $C^{\bullet,\bullet}(\Fg^\ast,  \Fc)$ to the
group cohomology of $N$ with coefficients in $\wedge^\bullet \Fg^\ast$. To this end
we introduce the homogeneous version of the complex \eqref{UF+*},
\begin{equation}\label{coinv}
C^{p,q}_{\rm coinv}(\Fg^\ast, \Fc):= ( \wg^p\Fg^\ast\ot  \Fc^{\ot q+1})^\Fc  ,
\end{equation}
defined as follows. An element  $ \sum \a\ot \td f\in ( \wg^p\Fg^\ast\ot  \Fc^{\ot q+1})^\Fc $
if it satisfies the $\Fc$-coinvariance condition:
\begin{equation}\label{coinv-condition}
\sum \a\ns{0}\ot \;\td{f}\ot S(\a\ns{1})=\sum \; \a\ot\td{f}\ns{0}\ot \td{f}\ns{1} ;
\end{equation}
here for  $\td f=f^0\odots f^q$, we have denoted
\begin{align} \label{coinv-not}
\td f\ns{0}\ot \td f\ns{1}\, = \, f^0\ps{1}\odots
f^q\ps{1}\ot f^0\ps{2}\cdots f^q\ps{2}.
\end{align}
One can identify the cochains of \eqref{UF+*} and \eqref{coinv} via the Connes duality  isomorphism defined originally in \cite{kr3} and is recalled here by
$\Ic:  \wg^p\Fg^\ast\ot\Fc^{\ot q}\ra (\wg^p\Fg^\ast\ot\Fc^{\ot q+1})^\Fc$ given by
\begin{align}\label{I-1}
\begin{split}
 &\Ic( \a\ot \td f) \,=\\
 &\, =a\ns{0}\ot f^1\ps{1}\ot S(f^1\ps{2})f^2\ps{1}\odots S(f^{q-1}\ps{2})f^q\ps{1}  \ot S(\a\ns{1} f^q\ps{2}).
 \end{split}
\end{align}
To check that the image of  $\Ic$ is formed of coinvariant elements, let
$\t\ot \td g:= \a\ns{0}\ot f^1\ps{1}\ot S(f^1\ps{2})f^2\ps{1}\odots
S(f^{q-1}\ps{2})f^q\ps{1}  \ot S(\a\ns{1} f^q\ps{2})$, we compute both sides of the
condition \eqref{coinv-condition}.
On the one hand,
\begin{align*}
&\t\ns{0}\ot \td g\ot S(\t\ns{1})= \\
&\a\ns{0}\ot f^1\ps{1}\ot S(f^1\ps{2})f^2\ps{1}\odots S(f^{q-1}\ps{2})f^q\ps{1}  \ot S(\a\ns{2} f^q\ps{2})\ot S(\a\ns{1})
\end{align*}
On the other hand,
\begin{align*}
&\t\ot \td g\ns{0}\ot \td g\ns{1}=\\
&\a\ns{0}\ot f^1\ps{1}\ot S(f^1\ps{4})f^2\ps{1}\ot S(f^2\ps{4})f^3\ps{1} \ot\dots \ot S(f^{q-1}\ps{4})f^q\ps{1}  \ot  S(\a\ns{2} f^q\ps{4})\ot \\
&  f^1\ps{2}S(f^1\ps{3})f^2\ps{2}S(f^2\ps{3})f^3\ps{2} \dots S(f^{q-1}\ps{3})f^q\ps{2}  S(\a\ns{1} f^q\ps{3})=\\
&\a\ns{0}\ot f^1\ps{1}\ot S(f^1\ps{2})f^2\ps{1}\odots S(f^{q-1}\ps{2})f^q\ps{1}  \ot S(\a\ns{2} f^q\ps{2})\ot S(\a \ns{1}).
\end{align*}
The fact that $\Ic$ is indeed an isomorphism can be seen by verifying that the formula
\begin{align}\label{I-1-inverse}
\Ic^{-1}(\a\ot \td f)= \a\ot f^0\ps{1}\ot  f^0\ps{2}f^1\ps{1}\odots f^0\ps{q}\dots f^{q-2}\ps{2}f^{q-1}\ve(f^q)
\end{align}
gives its inverse $\Ic^{-1}:  (\wg^p\Fg^\ast\ot\Fc^{\ot q+1})^\Fc \ra \wg^p\Fg^\ast\ot\Fc^{\ot q}$.

The isomorphism $\Ic$ turns the action \eqref{specactg} into the diagonal action
\begin{align} \label{diagact}
X\rt (f^0\odots f^q )=\sum_{i=0}^q f^0\odots X\rt f^i\odots f^q  .
\end{align}
Specifically, using
 the fact that  $\Fg$ acts on $\Fc$ by derivations in conjunction with the relations
\eqref{mp2} and  Lemma \eqref{S-linear}, one checks the identity
\begin{align} \label{switch}
&\Ic(\a\ot X\bullet \td f)=\\ \notag
& \a\ns{0}\ot X\rt( f^1\ps{1}\ot S(f^1\ps{2})f^2\ps{1}\odots S(f^{q-1}\ps{2})f^q\ps{1}  \ot S(\a\ns{1} f\ps{2})).
\end{align}

By transfer of structure, $ C^{\bullet,\bullet}_{\rm coinv}(\Fg^\ast, \Fc)$
can be equipped with the boundary operators
\begin{align*}
 &\p_\Fg^{\rm coinv}:C^{p,q}_{\rm coinv}(\Fg^\ast, \Fc)\longrightarrow  C^{p+1,q}_{\rm coinv}(\Fg^\ast, \Fc),\\\notag
 &\\
& b_\Fc^{\rm coinv}: C^{p,q}_{\rm coinv}(\Fg^\ast, \Fc)\longrightarrow  C^{p,q+1}_{\rm coinv}(\Fg^\ast, \Fc), \\\notag
&\\
&B_\Fc^{\rm coinv}: C^{p,q}_{\rm coinv}(\Fg^\ast, \Fc)\ra C^{p,q-1}_{\rm coinv}(\Fg^\ast, \Fc),
\end{align*}
as follows:
 \quad $\p_\Fg^{\rm coinv}$ is the restriction of  the Lie algebra cohomology coboundary of $\Fg$ with coefficients in $ \Fc^{\ot q+1}$ on which now $\Fg$ acts  from the right by
\begin{align*}
&( f^0\odots f^{q})\lt X:= - X \rt (f^0\odots  f^{q});
\end{align*}
 \begin{align*}
 &b_\Fc^{\rm coinv}( \a\ot f^0\odots f^q):=\\
 &~~~~~~~~~~~~~~~~\sum_{i=0}^{q+1}(-1)^i \a\ot f^0\odots f^{i-1}\ot 1\ot f^i\odots f^q ; \\
& B_\Fc^{\rm coinv}:=\sum_{i=0}^{q-1}(-1)^{(q-1)i}\tau^i_{\rm coinv}\s_{\rm coinv}\tau_{\rm coinv}, \\
 &\tau_{\rm coinv}(\a\ot f^0\odots f^q):=  \a\ot f^1\odots f^q\ot f^0,\\
 & \s_{\rm coinv}(\a\ot f^0\odots f^q):= \a\ot f^0\odots f^{q-1}f^{q}.
  \end{align*}
 One obtains the bigraded module
\begin{align}\label{g*F}
\begin{xy} \xymatrix{  \vdots & \vdots
 &\vdots &&\\
 (\wdg^2\Fg^\ast\ot \Fc)^\Fc  \ar[u]^{\p_\Fg^{\rm coinv}}\ar@<.6 ex>[r]^{b_\Fg^{\rm coinv}~~~~~~~}& \ar@<.6 ex>[l]^{~~B_\Fg^{\rm coinv}} (\wdg^2\Fg^\ast\ot\Fc^{\ot 2})^\Fc \ar[u]^{\p_\Fg^{\rm coinv}} \ar@<.6 ex>[r]^{b_\Fg^{\rm coinv}}& \ar@<.6 ex>[l]^{~~B_\Fg^{\rm coinv}}(\wdg^2\Fg^\ast\ot\Fc^{\ot 3})^\Fc \ar[u]^{\p_\Fg^{\rm coinv}} \ar@<.6 ex>[r]^{~~~~~~~~~b_\Fg^{\rm coinv}} & \ar@<.6 ex>[l]^{~~B_\Fg^{\rm coinv}}\hdots&  \\
 (\Fg^\ast \ot \Fc)^\Fc  \ar[u]^{\p_\Fg^{\rm coinv}}\ar@<.6 ex>[r]^{b_\Fg^{\rm coinv}~~~~~}& \ar@<.6 ex>[l]^{~~B_\Fg^{\rm coinv}} (\Fg^\ast\ot\Fc^{\ot 2})^\Fc \ar[u]^{\p_\Fg^{\rm coinv}} \ar@<.6 ex>[r]^{b_\Fg^{\rm coinv}}& \ar@<.6 ex>[l]^{~~B_\Fg^{\rm coinv}} (\Fg^\ast\ot \Fc^{\ot 3})^\Fc \ar[u]^{\p_\Fg^{\rm coinv}} \ar@<.6 ex>[r]^{~~~~~b_\Fg^{\rm coinv} }&\ar@<.6 ex>[l]^{~~B_\Fg^{\rm coinv}} \hdots&  \\
 (\Cb\ot \Fc)^\Fc  \ar[u]^{\p_\Fg^{\rm coinv}}\ar@<.6 ex>[r]^{b_\Fg^{\rm coinv}~~~~~~~}& \ar@<.6 ex>[l]^{~~B_\Fg^{\rm coinv}} (\Cb\ot \Fc^{\ot 2})^\Fc \ar[u]^{\p_\Fg^{\rm coinv}}\ar@<.6 ex>[r]^{b_\Fg^{\rm coinv}}& \ar@<.6 ex>[l]^{~~B_\Fg^{\rm coinv}} (\Cb\ot \Fc^{\ot 3})^\Fc \ar[u]^{\p_\Fg^{\rm coinv}} \ar@<.6 ex>[r]^{~~~~~b_\Fg^{\rm coinv}} &\ar@<.6 ex>[l]^{~~B_\Fg^{\rm coinv}} \hdots&,  }
\end{xy}
\end{align}

\begin{proposition} \label{tocoinv}
     The map $\Ic:  C^{\bullet, \bullet}_{\rm coinv}(\Fg^\ast, \Fc)\ra C^{\bullet, \bullet}(\Fg^\ast, \Fc)$
is an isomorphism of bicomplexes.
\end{proposition}
\begin{proof}
Using \eqref{switch}, one checks that $\p_\Fg^{\rm coinv}=\Ic\circ\p^\ast_\Fg\circ\Ic^{-1}$.

 A similar but easier calculation, using that $\D(1)=1\ot 1$, verifies
the relation $b_\Fc^{\rm coinv}=\Ic\circ b_\Fc^\ast\circ \Ic^{-1}$.

One shows that $\Ic\tau=\tau_{\rm coinv}\Ic$ as follows:
\begin{align*}
&\Ic\tau(\a\ot f^1\odots f^q)=\\
&\Ic(\a\ns{0}\ot S(f^1\ps{q})f^2\odots S(f^1\ps{2})f^q\ot S(f^1\ps{1})S(\a\ns{1}))=\\
&\a\ns{0}\ot (S(f^1\ps{q})f^2)\ps{1}\ot S((S(f^1\ps{q})f^2)\ps{2})(S(f^1\ns{q})f^2)\ps{1} \ot\dots\\
   &\dots\ot S((S(f^1\ps{2})f^q)\ps{2}) (S(f^1\ps{2})f^q)\ps{1}\ot \\
   &S((S(f^1\ps{2})f^q)\ps{2})(S(f^1\ps{1})S(\a\ns{2}))\ps{1}\ot S((S(f^1\ps{1})S(\a\ns{2}))\ps{2}) S(\a\ns{1})=\\
   &\a\ns{0}\ot  S(f^1\ps{2})f^2\ps{1}\odots S(f^{q-1}\ps{2})f^q\ps{1}  \ot S(\a\ns{1} f^q\ps{2})\ot f^1\ps{1}=\\
   &\tau_{\rm coinv}\Ic(\a\ot f^1\odots f^q).
\end{align*}
Leaving to the reader the verification of
 the similar relation $\Ic\s=\s_{\rm coinv}\Ic$, we conclude
 that $\p_\Fg^{\rm coinv}$,  $b_\Fc^{\rm coinv}$ and $B_\Fc^{\rm coinv}$ are well-defined
 and form a bicomplex. By the very construction, $\Ic$ is automatically a map of bicomplexes.
\end{proof}

The appropriate group cohomology complex for our purposes is that for
 the cohomology with polynomial cochains of the inverse limit group
\begin{align}
\FN : = \limproj_{k \ra \infty} \FN_k
\end{align}
with $\FN_k$ denoting the group of invertible
 $k$-jets at $0$ of diffeomorphisms in $N$. Recall, \cf~\cite[\S 1.2-1.3]{mr09}, that
the algebra
$\Fc$ is precisely the polynomial algebra generated by the components of
such jets.
 When there is no danger of confusion, we shall abuse
the notation by simply
writing $\psi \in \FN$ instead of $j_0^\infty(\psi) \in \FN \,$ with $\psi \in N$.
We denote by
$C^q_{\rm pol}(\FN,\wedge^p \Fg^\ast)$  the set of polynomial functions
$$c: \underset{q+1\,\text{times}}{\underbrace{\FN \times \ldots \times \FN}} \ra \wedge^p\Fg^\ast $$
 satisfying the covariance condition
 \begin{align} \label{covcond}
   c (\psi_0 \psi, \dots ,\psi_q \psi)= \,\psi^{-1} \rt c (\psi_0,\dots,\psi_q) , \qquad \fl \, \psi \in \FN ,
\end{align}
where $\rt$ stands for the left action defined in Remark \ref{ractg*}. The coboundary
for group cohomology with homogeneous cochains is
\begin{align} \label{homb}
b_{\rm pol} c (\psi_0, \dots, \psi_{q+1}) =\sum_{i=0}^{q+1}(-1)^i
c (\psi_0,\dots,\hat{\psi_i},\dots, \psi_{q+1}) .
\end{align}
In addition, we equip $C^\bullet_{\rm pol}(\FN,\wedge^\bullet \Fg^\ast)$ with a
$B$-boundary $B_{\rm pol}$
made out  as usual (\cf \eg  \eqref{B-bd}) of the cyclic operator $\tau_{\rm pol}$
and the codegeneracy $\s_{\rm pol}$, which in turn  are given by
\begin{align}\label{t-s-pol}
\begin{split}
&\tau_{\rm pol}(c)(\psi_0, \dots, \psi_q)= c(\psi_1,  \dots, \psi_q,\psi_0), \\
&\s_{\rm pol}(c)(\psi_0,\dots, \psi_{q-1})=c(\psi_0,\dots, \psi_{q-1}, \psi_{q-1}).
\end{split}
\end{align}
The vertical operators are given by the  coboundaries
 $\p_{\rm pol}: C^q_{\rm pol}(\FN,\wedge^p \Fg^\ast)\ra C^q_{\rm pol}(\FN,\wedge^{p+1} \Fg^\ast) $,
\begin{align}\label{p-pol}
(\p_{\rm pol} c) (\psi_0, \dots, \psi_{q}) =\p (c(\psi_0, \dots, \psi_{q})) -
 \sum_i \t^i\wdg (X_i\rt c)(\psi_0, \dots, \psi_{q}) .
\end{align}
Here  $\p:\wedge^p \Fg^\ast\ra \wedge^{p+1} \Fg^\ast$ stands for
the  Lie algebra cohomology
coboundary of $\Fg$ with trivial coefficients,
 $\{X_i\}_{1\leq i \leq m}$ is a basis of $\Fg$ with dual basis  $\{\t_i\}_{1\leq i \leq m}$,
and the
action of $\Fg $ on a group cochain $c\in C^q_{\rm pol}(\FN,\wedge^p \Fg^\ast)$ is given by
\begin{align} \label{Xactchain}
&(X\rt c)(\psi_0, \dots, \psi_{q})= \sum_i \dt c(\psi_0, \dots, \psi_{i}\lt\exp(tX), \dots, \psi_{q}) .
\end{align}

We thus arrive at a bicomplex which mixes group cohomology of $\FN$ with
coefficients in $\wg^\bullet\Fg^\ast$ and Lie algebra cohomology of $\Fg$ with
coefficients in $\Fc$, described by the diagram
\begin{align} \label{g*N}
\begin{xy} \xymatrix{ \vdots & \vdots
 &\vdots  & \\
C_{\rm pol}^0 \big(\FN , \wdg^2\Fg^\ast \big)
  \ar@<.6 ex>[r]^{b_{\rm pol}} \ar[u]^{\p_{\rm pol}}&\ar@<.6 ex>[l]^{B_{\rm pol}}
C_{\rm pol}^1 \big(\FN , \wdg^2\Fg^{\ast}  \big)
  \ar@<.6 ex>[r]^{b_{\rm pol}} \ar[u]^{\p_{\rm pol}}
 & \ar@<.6 ex>[l]^{B_{\rm pol}}C_{\rm pol}^2 \big(\FN ,  \Fg^{\ast}  \big)
 \ar[u]^{\p_{\rm pol}}  \ar@<.6 ex>[r]^{b_{\rm pol}}& \ar@<.6 ex>[l]^{B_{\rm pol}}\hdots\\
C_{\rm pol}^0 \big(\FN , \Fg \big)
 \ar@<.6 ex>[r]^{b_{\rm pol}} \ar[u]^{\p_{\rm pol}}& \ar@<.6 ex>[l]^{B_{\rm pol}}C_{\rm pol}^1 \big(\FN , \Fg^{\ast}  \big)
\ar@<.6 ex>[r]^{b_{\rm pol}} \ar[u]^{\p_{\rm pol}}& \ar@<.6 ex>[l]^{B_{\rm pol}}C_{\rm pol}^2 \big(\FN ,  \Fg^{\ast}\big)
\ar[u]^{\p_{\rm pol}} \ar@<.6 ex>[r]^{b_{\rm pol}}& \ar@<.6 ex>[l]^{B_{\rm pol}}\hdots\\
C_{\rm pol}^0 \big(\FN , \Cb \big) \ar@<.6 ex>[r]^{b_{\rm pol}}\ar[u]^{\p_{\rm pol}}& \ar@<.6 ex>[l]^{B_{\rm pol}} C_{\rm pol}^1
\big(\FN , \Cb  \big)
 \ar@<.6 ex>[r]^{b_{\rm pol}}\ar[u]^{\p_{\rm pol}}& \ar@<.6 ex>[l]^{B_{\rm pol}}C_{\rm pol}^2 \big(\FN , \Cb \big)\ar[u]^{\p_{\rm pol}}
\ar@<.6 ex>[r]^{b_{\rm pol}}& \ar@<.6 ex>[l]^{B_{\rm pol}} \hdots  }
\end{xy}
\end{align}
In turn, this can be related to the bicomplex \eqref{g*F} via the obvious map
$\Jc:~C^{\bullet,\bullet}_{\rm coinv}(\Fg^\ast, \Fc)\ra C^\bullet_{\rm pol}(\FN,\wg^\bullet\Fg^\ast)$
defined, with the self-explanatory notation, by the formula
\begin{align} \label{Jmap}
\Jc\big(\sum \a\ot \td f\big) \,(\psi_0,\dots, \psi_q) \,=\,\sum f^0(\psi_0)\dots f^q(\psi_q)\a .
\end{align}

\begin{proposition} \label{co-pol}
The map $\Jc: C^{\bullet,\bullet}_{\rm coinv}(\Fg^\ast, \Fc)\ra C^\bullet_{\rm pol}(\FN,\wg^\bullet\Fg^\ast)$
 is an isomorphism of bicomplexes.
\end{proposition}
\begin{proof}
Let us first show that $\Jc$ is well-defined. Suppressing the summation sign in the
 coinvariance condition \eqref{coinv-condition}, we
assume that $\a\ot \td f$  satisfies it and compute, using the identity \eqref{Naction},
\begin{align*}
&\Jc(\a\ot f^0\odots f^q)(\psi_0\psi, \dots \psi_q\psi)=\a f^0(\psi_0\psi)\dots f^q(\psi_q\psi)=\\
&\a f^0\ps{1}(\psi_0)f^0\ps{2}(\psi)\dots f^q\ps{1}(\psi_q)f^q\ps{2}(\psi)=\a f^0\ps{1}(\psi_0)\dots f^q\ps{1}(\psi_q)(f^0\ps{2}\dots f^q\ps{2})(\psi)=\\
&\a\ns{0} f^0(\psi_0)\dots f^q(\psi_q) S(\a\ns{1})(\psi)=\a\ns{0} f^0(\psi_0)\dots f^q(\psi_q) \a\ns{1}(\psi^{-1})=\\
&(\psi^{-1}\rt\a) f^0(\psi_0)\dots f^q(\psi_q)=\psi^{-1}\rt\big(\Jc(\a \ot f^0\odots f^q)(\psi_0,\dots, \psi_q)\big).
\end{align*}
By the very definitions of $\Fc$ and $C^\bullet(\FN, \wg^\bullet\Fg^\ast)$, it is obvious that $\Jc$ is isomorphism of vector spaces. We just need to show that
\begin{equation*}
b_{\rm pol}\Jc=\Jc b_{\rm coinv}, \quad \p_{\rm pol}\Jc=\Jc\p_{\rm coinv},
\quad \tau_{\rm pol}\Jc=\Jc\tau_{\rm coinv}, \quad \s_{\rm pol}\Jc=\Jc\s_{\rm coinv}.
\end{equation*}
The verification of the first commutation relation is straightforward. Indeed,
\begin{align*}
&(b_{\rm pol}\Jc)(\a\ot f^0\odots f^q)(\psi_0,\dots, \psi_{q+1})=\\
&\sum_i (-1)^i\Jc(\a\ot f^0\odots f^q)(\psi_0,\dots,\hat{\psi_i},\dots, \psi_{q+1})=\\
&\sum_i (-1)^i\a f^0(\psi_0)\dots f^{i-1}(\psi_{i-1})f^{i}(\psi_{i+1})\dots  f^q(\psi_{q+1})=\\
&\sum_i(-1)^i \Jc(\a\ot f^0\odots f^{i-1}\ot 1\ot f^{i}\odots f^q)(\psi_0,\dots, \psi_{q+1})=\\
&\Jc \big(b_{\rm coinv}(\a\ot f^0\odots f^q)\big)(\psi_0,\dots,\psi_{q+1}).
\end{align*}
To check the second relation, recalling \eqref{Xactchain}, we compute

\begin{align*}
&(\p_{\rm pol}\Jc)(\a\ot f^0\odots f^q)(\psi_0,\dots, \psi_{q})=
\p\big(\Jc(\a\ot f^0\odots f^q)(\psi_0,\dots, \psi_{q})\big)-\\
&\sum_i \t^i\wdg (X_i\rt \Jc(\a\ot f^0\odots f^q))(\psi_0, \dots, \psi_{q})=\\
&\p\big(\a f^0(\psi_0)\dots f^q(\psi_{q})\big)-\sum_{i,j} \t^i\wdg \a f^0(\psi_0)\dots  \dt f^j(\psi_j\lt \exp(tX_i))\dots f^q(\psi_{q})=\\
&\p(\a)f^0(\psi_0)\dots f^q(\psi_{q})-\sum_{i,j} \t^i\wdg \a f^0(\psi_0)\dots  (X_i\rt f^j)(\psi_j)\dots f^q(\psi_{q})=\\
&\Jc(\p_{\rm coinv}(\a\ot f^0\odots f^q))(\psi_0,\dots,\psi_q).
\end{align*}
The remaining two commutation relations are obvious.
\end{proof}

Since $\Fc$ is commutative, the coaction of  $\Fc$ on $\Fc^{\ot q+1}$
restricts to $ \wedge^{q+1}\Fc$, giving the coaction
\begin{align}
\Db_{\rm wedge}(f^0\wdots f^q)= f^0\ps{1}\wdots f^q\ps{1}\ot f^0\ps{2}\dots f^q\ps{2}.
\end{align}
With this understood, we consider the
subcomplex  $\, C^{p,q}_{\rm c-w}(\Fg^\ast, \Fc) \sbs
C^{\bullet,\bullet}_{\rm coinv}(\Fg^\ast, \Fc)$ formed of the
$\Fc$-coinvariant cochains $(\wedge ^{p}\Fg^\ast\ot \wedge ^{q+1}\Fc)^{\Fc}$.
Diagrammatically,
\begin{align}\label{wedge-coinv}
\begin{xy} \xymatrix{  \vdots & \vdots
 &\vdots &&\\
 \wdg^2\Fg^\ast  \ar[u]^{\p_{\rm c-w}}\ar[r]^{b_{\rm c-w}~~~~~~~}&  (\wdg^2\Fg^\ast\ot\wdg^2 \Fc)^\Fc \ar[u]^{\p_{\rm  c-w}} \ar[r]^{b_{\rm c-w}}& (\wdg^2\Fg^\ast\ot\wdg^3\Fc)^\Fc \ar[u]^{\p_{\rm c-w}} \ar[r]^{~~~~~~~~~b_{\rm c-w}} & \hdots&  \\
 \Fg^\ast  \ar[u]^{\p_{\rm c-w}}\ar[r]^{b_{\rm c-w}~~~~~}&  (\Fg^\ast\ot\wdg^2 \Fc)^\Fc \ar[u]^{\p_{\rm c-w}} \ar[r]^{b_{\rm c-w}}& (\Fg^\ast\ot\wdg^3\Fc)^\Fc \ar[u]^{\p_{\rm c-w}} \ar[r]^{~~~~~b_{\rm c-w} }& \hdots&  \\
 \Cb  \ar[u]^{\p_{\rm c-w}}\ar[r]^{b_{\rm c-w}~~~~~~~}&  (\Cb\ot\wdg^2 \Fc)^\Fc \ar[u]^{\p_{\rm c-w}} \ar[r]^{b_{\rm c-w}}& (\Cb\ot\wdg^3\Fc)^\Fc \ar[u]^{\p_{\rm c-w}} \ar[r]^{~~~~~b_{\rm c-w}} & \hdots&,  }
\end{xy}
\end{align}
with the coboundaries
\begin{align}\label{b-F-wedge}
b_{\rm c-w}(\a\ot f^0\wdots f^q)= \a\ot 1\wdg f^0\wdots f^q,
\end{align}
and

\begin{align} \label{p-F-wedge}
&\p_{\rm c-w}(\a\ot f^0\wdots f^q)(X^0,X^2,\dots,X^{p})=\\\notag
&\sum_{0\le i<j\le p} (-1)^{i+j}\a([X^i,X^j], X^0,\dots \check{X^i},\dots \check{X^j},\dots, X^p) f^0\wdots f^q-\\\notag
&\sum_{0\le i\le p, \;\;0\le j\le q} (-1)^i\a(X^0,\dots \check{X^i},\dots X^p) f^0\wdots X^i\rt f^j\wdots f^q ,
\end{align}
or equivalently,
\begin{align}
\begin{split}
&\p_{\rm c-w}(\a\ot f^0\wdots f^q)=\\
& \p\a\ot f^0\wdots f^q \, -\, \sum_i \t^i\wdg \a\ot \wedge\ot X_i\rt(f^0\wdots f^q).
\end{split}
\end{align}
One remarks that
\begin{multline*}
\tau_{\rm coinv}(\a\ot f^0\wdots f^q)= \a\ot f^1\wdots f^q\wdg f^0=\\
 (-1)^q  \a\ot f^0\wdots f^q,
 \end{multline*}
 which implies that $B^{\rm coinv}_\Fc\equiv 0$ in $C^{\bullet,\bullet}_{\rm c-w}(\Fg^\ast, \Fc)$.

We denote by $\td\a_\Fc: C^{p,q}_{\rm c-w}(\Fg^\ast, \Fc)\ra C^{p,q}_{\rm coinv}(\Fg^\ast, \Fc)$
 the $\Fc$-antisymmetrization map
\begin{align}\label{anti-sym}
\td\a_\Fc(\a\ot f^0\wdots f^q)=\frac{1}{(q+1)!} \sum_{\s\in S_{q+1}}(-1)^\s \a\ot f^{\s(0)}\odots f^{\s(q)}.
\end{align}
and note that it commutes with coboundaries. One also observes that  the projection
$\pi_\Fc: C^{p,q}_{\rm coinv}(\Fg^\ast, \Fc)\ra  C^{p,q}_{\rm c-w}(\Fg^\ast, \Fc)$,
\begin{align}\label{pi}
\pi_\Fc(\a\ot f^0\odots f^q)= \a\ot f^0\wdots f^q,
\end{align}
is a map of bicomplexes which is a left inverse for $A$.

\begin{proposition}\label{iso-coinv-c-w}
The map $\td\a_\Fc: C^{p,q}_{\rm c-w}(\Fg^\ast, \Fc)\ra C^{p,q}_{\rm coinv}(\Fg^\ast, \Fc)$
 is a quasi-isomorphism.
\end{proposition}
\begin{proof}
Denote by $C^\bullet_{\rm pol-w}(\FN,\wg^\bullet\Fg)$
the subcomplex of $C^\bullet_{\rm pol}(\FN,\wg^\bullet\Fg)$ consisting of
the totally antisymmetric group cochains, equipped
with the induced   coboundaries, and define
\begin{align*}
\Jc^{\rm wedge}(\a\ot f^0\wdots f^q)(\psi_0,\dots,\psi_q)=\frac{1}{(q+1)!}\sum_{\s\in S_{q+1}}(-1)^{\s}\a f^{\s(0)}(\psi_0)\cdots f^{\s(q)}(\psi_q).
\end{align*}
 It is clear that the following diagram is commutative.
\begin{equation}
\xymatrix{
 C^{p,q}_{\rm c-w}(\Fg^\ast, \Fc)\ar@{^{(}->}[r]^{\td\a_\Fc}_\simeq \ar[d]^{\Jc^{\rm wedge}}& C^{p,q}_{\rm coinv}(\Fg^\ast, \Fc)\ar[d]^{\Jc}\\
C^{q}_{\rm pol-w}(\FN,\wdg^p\Fg^\ast) \ar@{^{(}->}[r]& C^{q}_{\rm pol}(\FN,\wdg^p\Fg^\ast) .
}
\end{equation}
The lower row of the above square is known to be
a quasi-isomorphism, and the vertical maps
$\Jc$ and $\Jc^{\rm wedge}$ are clearly isomorphisms.
\end{proof}

\bigskip

%%%%%%%%%%%%%%%%%%%%%%

\section{Lie algebra, equivariant and Hopf cyclic cohomology} \label{S: van est}

In the first part of this section we shall
recall from~\cite{cm2} the construction of a canonical homomorphism
from the Gelfand-Fuks cohomology to the
$\Diff$-equivariant cohomology of the frame bundle. This homomorphism implements
a key component of the classical construction realizing the Gelfand-Fuks classes as
characteristic classes of foliations. The novel feature, brought out in the
second part of this section, consists in identifying its image as a bicomplex computing the
Hopf cyclic cohomology. The last part of the section establishes the isomorphism
of the latter bicomplex with the Gelfand-Fuks Lie algebra cohomology.

\subsection{From Gelfand-Fuks to equivariant cohomology} \label{Ss: GF to EC}

 With $\Gb$ designating as above
the group of global diffeomorphisms of an infinite primitive pseudogroup,
we let $\Fa = \Fa (\Gb)$ denote the corresponding Lie algebra of formal vector fields. A
formal vector field $v \in \Fa$ can be identified with the infinite jet of a vector field at $0$.
Letting $\Jc_0^\infty = \Jc_0^\infty(\Gb) $ stand for the manifold of
(invertible) jets of infinite order at $0$ of
local diffeomorphisms in $\Gb$,
and denoting by $j_0^\infty (\phi)$ the jet of $\phi \in \Gb$,
any formal vector field $v \in \Fa$ can be represented as the jet at $0$ of a
$1$-parameter group of local diffeomorphisms $\{\phi_t\}_{t \in \Rb}$,
$\displaystyle \quad v = j_0^\infty \left(\frac{d}{dt} \mid_{t=0} \phi_t\right)$,
and gives rise to a left invariant vector field
$\tilde v$ on $\Jc_0^\infty$, whose value at $ j_0^\infty (\phi) \in \Jc_0^\infty$ is given by

\begin{align*}
 \tilde{v} \mid_{j_0^\infty (\phi)}  = j_0^\infty \left(\frac{d}{dt} \mid_{t=0} (\phi \circ \phi_t ) \right) .
\end{align*}
We can thus associate to each $\, \om  \in C_{\rm top}^{m}(\Fa)$
the left invariant form $\tilde\om$ on $\Jc_0^\infty$ defined by
 \begin{align} \label{form}
 \tilde\om (\tilde{v}_1\mid_{j_0^\infty (\phi)} , \ldots , \tilde{v}_m \mid_{j_0^\infty (\phi)}) \,
 = \, \om (v_1, \ldots , v_m) .
 \end{align}
 We  let $J\pi^\infty_1 : \Jc_0^\infty  \rightarrow \Jc_0^1 \simeq G$ stand for
 the projection onto the first jet bundle, and denote
 by $\pi_G: \Jc_0^\infty \rightarrow G$ and $\pi_N : \Jc_0^\infty \rightarrow \FN$ the projections
 relative to the Kac decomposition $\Jc_0^\infty = G \cdot \FN$.
 \medskip

 To construct the desired homomorphism, we associate to any
 $\, \om  \in C^r (\Fa)$ and any pair of integers $(p, q)$, $p+ m = q +r$,
 $0 \leq q \leq m=\dim G$,
 a $p$-cochain $\Cc_{p,q}(\om)(\psi_0,\dots, \psi_p)$ on the group $\FN$ with
values in $q$-currents on $G$ as follows.
For each $q$-form with compact support
$\z \in \Om^q_{c} (G)$, one sets
\begin{align}\label{vanest}
\langle \Cc_{p,q}(\om)(\psi_0,\dots, \psi_p), \, \z \rangle \, = \,
 \int _{\Sigma(\psi_0,\dots, \psi_p)}
 J\pi_1^{\infty \, \ast} (\z) \wedge {\tilde\om} \, ;
\end{align}
the integration is taken over the finite dimensional
 cycle in $\Jc_0^\infty$
 \begin{align} \label{intcycle}
  \Sigma = \Sigma(\psi_0,\dots, \psi_p) = \{ j_0^\infty(\rho) \in \Jc_0^\infty \, ; \quad
 \pi_N(\rho^{-1})  \in \D(\psi_0,\dots, \psi_p) \} ,
 \end{align}
where
  \begin{align} \label{simplex}
   \D(\psi_0,\dots,\psi_p)\, =\,
\left\{\sum_{i=0}^p t_i \,  j_0^\infty(\psi_i) \, \mid t_i\geq 0, \, \sum_{i=0}^pt_i=1\right\} .
\end{align}
We note that $\, \D(\psi_0,\dots,\psi_p)\, \sbs \, \FN$,
because at the first jet level
  \begin{align} \label{Nsimplex}
\sum_{i=0}^p t_i \,  j_0^1(\psi_i)  \, =\, (0, \id) .
\end{align}
Also,
the integration makes sense because the form $ {\tilde\om}$ lives on a finite-order jet bundle
and has an algebraic expression, with polynomial coefficients, in the canonical coordinates.

As mentioned
in the preceding section, we shall often
abbreviate the notation by not marking the distinction between
$\psi \in \FN$ and its jet $j_0^\infty(\psi) \in \FN$. Manifestly, the simplex definition is
covariant with respect to right translations by elements of $\FN$,
\begin{align} \label{simplex1}
   \D(\psi_0 \psi,\dots,\psi_p \psi)\, =\,
 \D(\psi_0 ,\dots,\psi_p)\, \psi ,  \qquad \fl \, \psi \in \FN .
\end{align}
A similar relation holds with respect to the right action of $G$ on $\FN$. For
$\vp \in G$ this action amounts to passing from  $\psi \in \FN$ to
 $\psi \lt \vp \in \FN$.
Recall from~\cite[Ch.1]{mr09} that,  with the shorthand notation
$\, \eta = \eta^i_{j  \ell_1 \ldots \ell_r}$  and $\, \g = \g^i_{j \ell_1 \ldots \ell_r}$,
one has
\begin{equation} \label{Gact1}
 \g (\psi \lt \vp) \, = \, \g (\psi \circ \vp) \, = \, \g (\psi) \circ \tilde{\vp}
  , \qquad \forall \, \vp \in G ,
\end{equation}
in particular
\begin{equation} \label{Gact2}
\eta (\psi \lt \vp) \, = \, \g (\psi)(\tilde{\vp}(e)) ,
\end{equation}
Thus, tautologically, one has
  \begin{align} \label{simplex2}
   \D(\psi_0 \lt \vp,\dots,\psi_p \lt \vp)\, =\,
 \D(\psi_0 ,\dots,\psi_p)  \lt \vp  , \qquad \forall \, \vp \in G ,
 \end{align}
and the identities \eqref{simplex1} and \eqref{simplex2} taken together express the
covariance of the simplex with respect to the right action of $\Gb$ on $\FN$:
  \begin{align} \label{simplex3}
   \D(\psi_0 \lt \phi,\dots,\psi_p \lt \phi)\, =\,
 \D(\psi_0 ,\dots,\psi_p)  \lt \phi , \qquad \forall \, \phi \in \Gb .
 \end{align}

\begin{lemma} \label{lem:vanestinv}
The group cochain $\Cc_{p,\bullet} (\om)$ satisfies the covariance
property
\begin{align}\label{eq:vanestinv}
 \Cc_{p,\bullet}(\om)(\psi_0 \lt \phi,\dots,\psi_p \lt \phi) =\,
^tL^\ast_{\phi^{-1} \rt} \, \Cc_{p,\bullet}(\om)(\psi_0,\dots, \psi_p) , \quad \forall \, \phi \in \Gb ,
\end{align}
where the left translation operator in the right hand side refers to the left action of
$\Gb$ on G.
\end{lemma}
\begin{proof} Indeed, by \eqref{simplex3},
 when computing $\langle \Cc_{p,\bullet}(\om)(\psi_0 \lt \phi,\dots,\psi_p \lt \phi), \, \z \rangle$ one
integrates on the cycle
$\quad \Sigma_\phi= \{ \rho  \in {\Gb} \, ; \, \pi_N(\rho^{-1}) \lt \phi^{-1} \in
\D(\psi_0,\dots, \psi_p) \}$. Now remark that
\begin{align} \label{pin}
\pi_N(\rho^{-1}) \lt \phi^{-1}  = \pi_N(\rho^{-1} \cdot \phi^{-1}) ;
\end{align}
for  $\phi \in \FN$
this absolutely obvious, while for $\phi \in G$ it can be seen by writing
\begin{equation*}
\begin{split}
\rho^{-1}\cdot  \phi^{-1} &= \pi_G (\rho^{-1})\cdot \pi_N (\rho^{-1}) \cdot \phi^{-1}  \\
&=\pi_G (\rho^{-1})\cdot (\pi_N (\rho^{-1}) \rt \phi^{-1}) \cdot (\pi_N (\rho^{-1}) \lt \phi^{-1})  .
\end{split}
\end{equation*}
By \eqref{pin},
$\, \rho \in \Sigma_\phi $ iff $\,\phi \cdot \rho \in \Sigma$, hence
$\,\Sigma_\phi  \, = \, L_{\phi^{-1}} \, \Sigma$. Thus,
\begin{equation*}
\begin{split}
\langle \Cc_{p,\bullet}(\om)&(\psi_0 \lt \phi,\dots,\psi_p \lt \phi) , \,  \z \rangle \, = \,
\int _{L_\phi^{-1}\Sigma}  J\pi_1^{\infty \, \ast} (\z) \wedge {\tilde\om} \, =\\
 &=\int _\Sigma  (J\pi_1^{\infty} \circ L_{\phi^{-1}})^\ast (\z) \wedge L_{\phi^{-1}}^\ast ({\tilde\om})
 = \int _\Sigma  (J\pi_1^{\infty} \circ L_{\phi^{-1}})^\ast (\z) \wedge {\tilde\om} ,
 \end{split}
\end{equation*}
since $\tilde\om$ is left invariant. Furthermore, noting that
 $\, J\pi^\infty_1 \circ L_\phi = L_{\phi \rt} \circ  J\pi^\infty_1$, which is
 the jet counterpart of the relation
 \begin{align*}
  \pi_G (\phi \cdot \rho) \, = \, \pi_G (\phi)\cdot \pi_N (\phi) \rt \rho \, = : \, \phi \rt \pi_G (\rho) .
 \end{align*}
one sees that $\, (J\pi_1^{\infty} \circ L_{\phi^{-1}})^\ast (\z) =
(L_{\phi^{-1} \rt} \circ  J\pi^\infty_1)^\ast (\z)
\, =\,  J\pi_1^{\infty \, \ast}  \big(L_{\phi^{-1} \rt}^\ast (\z)\big)$.
Thus
\begin{equation*}
\begin{split}
\langle \Cc_{p,\bullet}(\om)&(\psi_0 \lt \phi,\dots,\psi_p \lt \phi) , \,  \z \rangle \,= \,
  \int _\Sigma   J\pi_1^{\infty \, \ast}  \big(L_{\phi^{-1} \rt}^\ast (\z)\big)\wedge {\tilde\om} \, = \\
   &=\, \langle \Cc_{p,\bullet}(\om)(\psi_0,\dots,\psi_p) , \, L_{\phi^{-1}  \rt}^\ast (\z) \rangle ,
 \end{split}
\end{equation*}
and the identity \eqref{eq:vanestinv} follows by
transposing the action from forms to currents.
 \end{proof}

   \begin{proposition} \label{vEcomplex}
 The map $\Cc : C_{\rm top}^\bullet (\Fa) \ra
 C_{\rm top}^\bullet (\FN,   \Om_\bullet (G))$  is a map of complexes.
\end{proposition}

\begin{proof} One has
\begin{align*}
&\langle \Cc_{p, \bullet}(d\om)(\psi_0,\dots, \psi_{p}), \, \z \rangle =  \int _{\Sigma(\psi_0,\dots, \psi_{p})}
 J\pi_1^{\infty \, \ast} (\z) \wedge {d\tilde\om} \,  =\\
&= \int _{\Sigma(\psi_0,\dots, \psi_{p})}
d\big( J\pi_1^{\infty \, \ast} (\z) \wedge {\tilde\om}\big)
\, \pm \, \int _{\Sigma(\psi_0,\dots, \psi_{p})}
 J\pi_1^{\infty \, \ast} (d\z) \wedge {\tilde\om} .
\end{align*}
The second integral corresponds to the de Rham boundary
$\, \p^{\rm t}  \Cc(\om)(\psi_0,\dots, \psi_{p})$
while the first becomes,
on applying Stokes, the group coboundary:
\begin{align*}
&=\int _{\p\Sigma(\psi_0,\dots, \psi_{p})}
J\pi_1^{\infty \, \ast} (\z) \wedge {\tilde\om} \,=\sum_{i=0}^{p+1}(-1)^{i}\int _{\Sigma(\psi_0 ,
\dots,\widehat{\psi_i},\dots, \psi_{p})}
J \pi_1^{\infty \, \ast} (\z) \wedge {\tilde\om} \\
 &= \sum_{i=0}^{p}(-1)^{i}
 \langle \Cc_{p-1, \bullet}(\om) (\psi_0 , \dots,\widehat{\psi_i},\dots, \psi_{p}) , \z\rangle .
 \end{align*}
 \end{proof}

 The current $\, \Cc_{p, \bullet} (\om)$ in \eqref{vanest} is defined
 as the integration along the fibre,
 in the fibration $\, J\pi^\infty_1 : \Jc_0^\infty \ra \Jc_0^1 \simeq G$, of
 the current $\tilde{\om} \cap \D $. In order to better understand it,
 we shall revisit its definition.

 Let  $\, \rho = \vp \cdot \psi \in \Gb$, with $\vp \in G$ and $\psi \in \FN$. Then
 $$ \rho^{-1} \, =\,  \psi^{-1} \cdot \vp^{-1}  \, =\,
 (\psi^{-1} \rt \vp^{-1})\cdot  (\psi^{-1} \lt \vp^{-1})  ,
 $$
and so
\begin{equation} \label{recycle}
\begin{split}
j_0^\infty (\rho) \in \Sigma (\psi_0,\dots, \psi_p) & \iff \psi^{-1}  \lt \vp^{-1} \in
\D (\psi_0,\dots, \psi_p) \\
& \iff \psi \in \big(\D (\psi_0,\dots, \psi_p) \lt \vp\big)^{-1}
\end{split}
\end{equation}

In view of the equation \eqref{simplex2}, this
shows that the cycle \eqref{intcycle}
can be described as the image of the map
$\, \s_{(\psi_0,\dots,\psi_p)} : G \times \D_p \rightarrow  \Jc_0^\infty$
given by
  \begin{align} \label{sigmap}
   \s_{(\psi_0,\dots,\psi_p)}  (\vp , \tb) \, = \,
   \vp \cdot \big(\sb_{(\psi_0,\dots,\psi_p)}  (\tb) \lt \vp \big)^{-1},
\end{align}
where $\, {\sb}_{(\psi_0,\dots,\psi_p)} : \D_p \rightarrow  \FN$ is the map of the standard
$p$-simplex
\begin{align*}
\begin{split}
\D_p \, &=\,\{ \tb = (t_0 , \ldots t_p) \, \mid t_i\geq 0, \, \sum_{i=0}^pt_i=1\} , \\
\sb_{(\psi_0,\dots,\psi_p)}  (\tb) \, &= \,
    \sum_{i=0}^p t_i \, j_0^\infty(\psi_i) .
 \end{split}
\end{align*}
 Therefore the definition of the current \eqref{vanest} can be restated as follows:
\begin{align}\label{revanest}
\langle \Cc_{p,\bullet}(\om)(\psi_0,\dots, \psi_p), \, \z \rangle \, = \,
 \int _{G \times \D_p} \s_{(\psi_0,\dots,\psi_p)}^\ast
 \left(J\pi_1^{\infty \, \ast} (\z) \wedge {\tilde\om}\right)\, .
\end{align}

Since $\,\, J\pi_1^{\infty} \circ \s_{(\psi_0,\dots,\psi_p)} :G \times \D_p \rightarrow G$
coincides with the projection on the first factor $\pi_1: G \times \D_p \ra G$, it follows that
\begin{align} \label{vanest2}
\begin{split}
\langle \Cc_{p,\bullet}(\om)(\psi_0,\dots, \psi_p), \, \z \rangle &= \int _{G \times \D_p}
\pi_1^\ast (\z) \wedge \s_{(\psi_0,\dots,\psi_p)}^\ast ({\tilde\om}) \\
&=  \int _G \z \wedge  \int_{\pi_\D} \s_{(\psi_0,\dots,\psi_p)}^\ast ({\tilde\om}) ,
\end{split}
\end{align}
where $\, \displaystyle  \int_{\pi_\D} $ stands for the integration
along the fibre of  $\pi_\D: G \times \D_p \ra G$.
Thus,
the current $\, \Cc_{p, \bullet} (\om)(\psi_0,\dots, \psi_p)$ is actually
the Poincar\'e dual of  the form
\begin{align} \label{veform}
 \Dc_{p, \, \dim G-\bullet} (\om)(\psi_0,\dots, \psi_p) \, = \,
   \int_{\pi_\D}  \s_{(\psi_0,\dots,\psi_p)}^\ast ({\tilde\om}) \,  \in \, \Om^{\dim G-\bullet} (G) .
\end{align}

%%%%%%%%%%%%%%%%%%%
Transposing the covariance property in Lemma \ref{lem:vanestinv},
by Poincar\'e duality yields the identity
\begin{align} \label{veinv}
\begin{split}
\Dc_{p, \bullet}  ({\tilde\om}) (\psi_0 \lt \phi,\dots,\psi_p \lt \phi) \, = \,
L^\ast_{\phi \rt} \, \Dc_{p, \bullet}  ({\tilde\om}) (\psi_0,\dots, \psi_p) , \quad \forall \, \phi \in \Gb.
\end{split}
\end{align}
In particular, for any $\,   \vp \in G$ and any $q$-tuple of left invariant tangent vector
fields $(X^1, \ldots , X^q)$ on $G$  one has
\begin{align} \label{Gveinv}
\begin{split}
\Dc_{p, q} & ({\tilde\om}) (\psi_0,\dots, \psi_p)(X_\vp^1, \ldots , X_\vp^q) \\
&\, = \Dc_{p, q}  ({\tilde\om}) (\psi_0 \lt \vp,\dots,\psi_p \lt \vp)
(L_{\vp^{-1} \ast} X_\vp^1, \ldots , L_{\vp^{-1} \ast}X_\vp^q) \\
&\qquad = \Dc_{p, q}  ({\tilde\om}) (\psi_0 \lt \vp,\dots,\psi_p \lt \vp)
(X_e^1, \ldots , X_e^q) .
\end{split}
\end{align}

\subsection{From equivariant to Hopf cyclic cohomology} \label{Ss: EC to HC}

 Equation \eqref{Gveinv} shows that, as a
group cochain with values in $ \Om^ \bullet (G)$,
$\, \Dc_{p, \bullet} (\om)$ is completely
determined by the values taken at  $e \in G$,  i.e. by
 the ``core'' cochain
  \begin{align} \label{restunit}
\Ec_{p,  \bullet} (\om)(\psi_0,\dots, \psi_p) : = \,
 \int_{\pi_\D}  \s_{(\psi_0,\dots,\psi_p)}^\ast ({\tilde\om})  \mid_{\vp=e} \, \,
  \in \, \wedge^ \bullet\Fg^\ast .
\end{align}
To give a precise meaning to the localization at the unit $e \in G$, let us fix
a basis of left $G$-invariant forms on $G$,
$$
\{ \a_I = \a_{i_1} \wdg \ldots \wdg \a_{i_q}  \, ; \, I = (i_1 < \ldots < i_q) , \quad q \geq 0 \}  ,
$$
 and note that the form \eqref{veform} can be uniquely written as a sum
  \begin{align*}
 \int_{\pi_\D}  \s_{(\psi_0,\dots,\psi_p)}^\ast ({\tilde\om})   \, = \,
 \sum_I h_{(\psi_0,\dots,\psi_p)}^I \, \a_I \, , \quad
 \text{with}\quad h_{(\psi_0,\dots,\psi_p)}^I  \in C^\infty (G) .
\end{align*}
The right hand side of \eqref{restunit} is then given by the equality
  \begin{align} \label{restunit2}
 \int_{\pi_\D}  \s_{(\psi_0,\dots,\psi_p)}^\ast ({\tilde\om})  \mid_{\vp=e} \, \,= \,
\sum_I h_{(\psi_0,\dots,\psi_p)}^I (e) \, \a_I \mid_e \, \quad \in \, \wedge^ \bullet\Fg^\ast .
\end{align}
Denoting by $\, \{ X^I \}$ the basis of left invariant vector fields on $G$ dual to $\{ \a_I \}$, one
can express the above coefficients as genuine
integrals of forms over simplices,
  \begin{align} \label{coeffs}
 h_{(\psi_0,\dots,\psi_p)}^I (e)\,= \,
  \int_{\D_p}  \i_{X_e^I} \s_{(\psi_0,\dots,\psi_p)}^\ast ({\tilde\om})   \, ,
\end{align}
which leads to the following unambiguous definition of $\, \Ec_{p, \bullet} (\om)$ :
  \begin{align} \label{intrestunit}
\Ec_{p,  q} (\om)(\psi_0,\dots, \psi_p) \, = \,  \sum_{|I| = q}
\left( \int_{\D_p}  \i_{X_e^I} \s_{(\psi_0,\dots,\psi_p)}^\ast ({\tilde\om}) \right)
 \,\a_I \mid_e \,
  \in \, \wedge^q \Fg^\ast .
\end{align}

Observe now that one can factor the map defined in \eqref{sigmap} as follows:
\begin{align} \label{nu}
\begin{split}
\s_{(\psi_0,\dots,\psi_p)} &= \nu \circ (\Id_G \times \, \sb_{(\psi_0,\dots,\psi_p)} ) , \\
 \text{where} \qquad \nu : G \times \FN \rightarrow  \Jc_0^\infty &, \qquad
   \nu(\vp, \psi) \, = \, \vp \, \imath(\psi \lt \vp) , \\
  \text{and} \qquad \imath : \Gb \rightarrow  \Gb &, \qquad
   \imath (\psi)\, = \, \psi^{-1}  .
   \end{split}
\end{align}
Substituting this in  \eqref{intrestunit} yields the equivalent definition
  \begin{align} \label{intrestunit2}
\Ec_{p,  q} (\om)(\psi_0,\dots, \psi_p) \, = \,  \sum_{|I| = q}
\left(  \int_{\D(\psi_0,\dots, \psi_p)}
\imath^\ast \big( \i_{X_e^I} \nu^\ast ({\tilde\om})\big) \right)
 \,\a_I \mid_e .
\end{align}

\begin{remark} \label{rem:etod}
With the above notation,
 the relation between the cochains $\, \Dc_{p, \bullet} (\om)$ and
$\, \Ec_{p, \bullet} (\om)$ is encapsulated in the following identity:
\begin{align} \label{etod}
\begin{split}
 \Dc_{p,  q} (\om)(\psi_0,\dots, \psi_p) \mid_\vp \, = \,
  \sum_{|I| = q}
  \left(  \int_{\D(\psi_0 \lt \, \vp,\dots, \psi_p  \lt \, \vp)}
\imath^\ast \big( \i_{X_e^I} \nu^\ast ({\tilde\om})\big) \right)
 \,\a_I  \mid_\vp.
    \end{split}
\end{align}
\end{remark}

We next focus on further clarifying the significance of the cochain $\, \Ec_{p, \bullet} (\om)$.

\begin{lemma} \label{lem:redunit}
The form
   \begin{align} \label{muform}
\mu (\om) \, = \,  \sum_{|I| = q}
\imath^\ast \big( \i_{X_e^I} \nu^\ast ({\tilde\om})\big) \ot \a_I \mid_e \quad  \in
\Om^p (\FN, \wedge^q \Fg^\ast)
\end{align}
is invariant under the action of $\FN$, operating on $\FN$ by right translations and on the
coefficients via the differential of the action $\rt$. In addition,
  \begin{align} \label{eq:redunit}
\Ec_{p, q} (\om)(\psi_0,\dots, \psi_p)  \, = \,
 \int_{\D(\psi_0,\dots, \psi_p)}  \mu (\om)  .
 \end{align}
\end{lemma}

\begin{proof}
The identity \eqref{eq:redunit} is merely the reformulation of \eqref{intrestunit2} in view of
the notation introduced in \eqref{muform}.

By specializing the covariance property \eqref{veinv} to elements $\, \psi \in \FN$, one
obtains
\begin{align} \label{Nveinv}
\begin{split}
\Dc_{p, \bullet}  ({\tilde\om}) (\psi_0 \psi ,\dots,\psi_p  \psi) \, = \,
L^\ast_{\psi \rt} \, \Dc_{p, \bullet}  ({\tilde\om}) (\psi_0,\dots, \psi_p)  .
\end{split}
\end{align}
Now restricting the right hand side to $e \in G$, and noting that
$\, L_{\psi \rt} (e) = e$, while by the very definitions
\begin{align*}
\left( L^\ast_{\psi \rt} \a_I \right) \mid_e  \, = \, \psi^{-1} \rt  \left(\a_I \mid_e \right) ,
\end{align*}
one obtains
\begin{align*}
\begin{split}
L^\ast_{\psi \rt} \, \Dc_{p, \bullet}  ({\tilde\om}) (\psi_0,\dots, \psi_p) \mid_e \, = \,
\psi^{-1} \rt  \int_{\D(\psi_0,\dots, \psi_p)}  \mu (\om) \, = \,
 \int_{\D(\psi_0,\dots, \psi_p)}\, \psi^{-1} \rt  \mu (\om) .
\end{split}
\end{align*}
On the other hand, the left hand side restricted to $e \in G$ yields
\begin{align*}
\begin{split}
\Dc_{p, \bullet}  ({\tilde\om}) (\psi_0 \psi ,\dots,\psi_p  \psi) \mid_e \, = \,
 \int_{\D(\psi_0,\dots, \psi_p)  \psi}  \mu (\om) \, = \,
  \int_{\D(\psi_0,\dots, \psi_p)}  R^\ast_\psi \mu (\om) .
\end{split}
\end{align*}
The equality of the two sides gives the identity
\begin{align*}
 \int_{\D(\psi_0,\dots, \psi_p)}  R^\ast_\psi \mu (\om) \, = \,
  \int_{\D(\psi_0,\dots, \psi_p)}\, \psi^{-1} \rt  \mu (\om) , \qquad \fl \, \psi \in \FN .
\end{align*}
Since this is valid for any simplex, the two integrands must coincide.
\end{proof}

\bigskip
%%%%%%%%%%%%%%%%%%%%%%

  \begin{lemma} \label{vEcomplex2}
 The map $\Ec : C_{\rm top}^\bullet (\Fa) \ra
 C_{\rm pol}^\bullet (\FN,   \wedge^\bullet \Fg^\ast)$  is a map of complexes.
\end{lemma}

\begin{proof} The cochain $\, \Ec_{p, q} (\om)(\psi_0,\dots, \psi_p) $
depends {\em polynomially} on the coordinates of the jet of some finite order
of $\psi_0,\dots, \psi_p$. This follows from the fact that the form
$\, \mu(\om)$, {\em cf.} \eqref{muform}, can be expressed as a linear
combination with polynomial coefficients in the canonical jet-coordinates
of the standard basis for invariant forms on $\FN $ relative to the same coordinates.
Furthermore, the covariance relation \eqref{Nveinv}
shows that  $\, \Ec_{p, q} (\om) \in C_{\rm pol}^p (\FN, \wedge^q\Fg^\ast)$.

By Proposition \ref{vEcomplex}, one has
\begin{align*}
\Ec_{p, \bullet}(d\om)(\psi_0, \ldots, \psi_{p}) = d \big(\Ec_{p, \bullet}(\om)(\psi_0, \ldots, \psi_{p})\big)+
\big(\d\Ec_{p-1, \bullet}(\om)\big)(\psi_0, \ldots, \psi_{p}) .
\end{align*}
  The de Rham coboundary term in the right hand side is obtained by evaluating
at $\, \vp = e$ the expressions
\begin{align*} & d \big(\Dc_{p, q}  ({\tilde\om}) (\psi_0,\dots, \psi_p)\big)
({\tilde X}_0, \ldots , {\tilde X}_q) \, = \\
&\quad = \sum_i (-1)^i {\tilde X}_i \cdot \Dc_{p, q}  ({\tilde\om}) (\psi_0,\dots, \psi_p)
({\tilde X}_0, \ldots , \check{\tilde X}_i, \ldots , {\tilde X}_q) \\
&\quad + \sum_{i <j} (-1)^{i+j} \,
 \Dc_{p, q}  ({\tilde\om}) (\psi_0,\dots, \psi_p)
  ([{\tilde X}_i, {\tilde X}_j], {\tilde X}_0, \ldots , \check{\tilde X}_i,
  \ldots , \check{\tilde X}_j , \ldots , {\tilde X}_q) ,
\end{align*}
where ${\tilde X}$ stands for the left invariant vector field on $G$ associated
to $X \in \Fg$.
The second sum evaluated at $\, e \in G$ gives
\begin{align*}
 \sum_{i <j} (-1)^{i+j}  \Ec_{p, q}  ({\tilde\om}) (\psi_0,\dots, \psi_p)
  ([{X}_i, {X}_j], {X}_0, \ldots , \check{X}_i,
  \ldots , \check{X}_j , \ldots , {X}_q) .
\end{align*}
On the other hand, expressing the action of ${\tilde X}_0$ at $\, e \in G$
as the derivative   $\displaystyle \frac{d}{ds}\mid_{s=0} L(\exp sX_0)$,
the first term in the first sum yields, on applying
the identity \eqref{Gveinv},
\begin{align*}
& \frac{d}{ds}\mid_{s=0} \Dc_{p, q}  ({\tilde\om}) (\psi_0,\dots, \psi_p)
 \mid_{\exp sX_0} ({\tilde X}_1, \ldots , {\tilde X}_q) = \\
&=\frac{d}{ds}\mid_{s=0} \Ec_{p, q}  ({\tilde\om})(\psi_0 \lt \exp sX_0,\dots,\psi_p \lt \exp sX_0)
({X}_1, \ldots , {X}_q) .
  \end{align*}
 Repeating this argument for each of the remaining terms
 shows that the operator $d$ in the right hand side coresponds
 precisely the coboundary operator for
  Lie algebra cohomology  with coefficients.
   \end{proof}

 In conjunction with the previous isomorphisms
    \begin{eqnarray*}
&C^{\bullet}_{\rm pol}(\FN,\wg^\bullet\Fg^\ast) \, \cong \,
C^{\bullet,\bullet}_{\rm coinv}(\Fg^\ast, \Fc)
 \qquad &\text{\cf Proposition \ref{co-pol} } , \\
&C^{\bullet,\bullet}_{\rm coinv}(\Fg^\ast, \Fc)  \, \cong \, C^{\bullet, \bullet}_{\rm c-w}(\Fc,\Fg^\ast)  \qquad &\text{\cf Proposition \ref{iso-coinv-c-w} } , \\
&C^{\bullet, \bullet}_{\rm coinv}(\Fc,\Fg^\ast) \, \cong \,  C^{\bullet, \bullet}(\Fg^\ast, \Fc) \qquad  &\text{\cf Proposition \ref{tocoinv} } ,\\
&C^{\bullet, \bullet}(\Fc,\Fg^\ast) \, \cong \,  C^{\bullet}_\Uc(\Fc, V^\bullet(\Fg^\ast))
\qquad  &\text{\cf Theorem \ref{thm:dgH} } ,
\end{eqnarray*}
 this allows us to conclude with the following statement.

 \begin{proposition} \label{3id}
 Under the above identifications, the resulting homomorphism
 $\Ec_{\rm LH}$ maps
the Gelfand-Fuks Lie algebra cohomology complex $C_{\rm top}^\bullet (\Fa)$
to the $\Fg$-equivariant Hopf cyclic complex of $\Fc$ with coefficients in the
Koszul resolution $V^\bullet(\Fg^\ast)$.
 \end{proposition}

%%%%%%%%%%%%%%%%%%%%%%%%%%%%%%%%%%%%%%%%%
 \subsection{Explicit van Est isomorphism} \label{Ss: thm.vE}

Note that while the original assignment maps
$$ \om \in C_{\rm top}^s (\Fa) \mapsto \bigoplus_{p+ \dim G - r = s} \Cc_{p, r} (\om) \in
C_{\rm pol}^p (\FN,   \Om_r (G))  ,
$$
after the passage from $\Cc_{p, r}(\om)$ to $\Ec_{p, q} (\om)$, $ r + q = \dim G$,
by Poincar\'e duality it takes the form
$$ \om \in C_{\rm top}^s (\Fa) \mapsto \bigoplus_{p+ q= s} \Ec_{p, q} (\om) \in
C_{\rm pol}^p (\FN,    \wedge^q \Fg^\ast) .
$$
Furthermore, this assignment can be promoted to a map of bicomplexes,
if one identifies the Lie algebra $\Fa$ with the double crossed sum Lie algebra
 $\Fg\bowtie \Fn$.
%%%%%%%%%%%%%%%%%%%%%%%%%%%%%%%%%%%%%%%%%%%
Let us recall that a pair of Lie algebras  $(\Fg, \Fn)$ forms a matched pair  if there
 is a right action of $\Fg$ on $\Fn$ and a left action of $\Fn$ on $\Fg$,
\begin{equation}
\alpha: \Fn\ot \Fg\ra \Fn, \quad \alpha_X(\z)=\z\lt X, \quad \beta:\Fn\ot \Fg\ra \Fg, \quad \beta_\z(X)=\z\rt X,
\end{equation}
satisfying the following conditions:
\begin{align*}
&[\z,\x]\rt X=\z\rt(\x\rt X)-\x\rt(\z\rt X), \quad \z\lt[X, Z]=(\z\lt X)\lt Z-(\z\lt Z)\lt X, \\
&\z\rt[X, Z]=[\z\rt X, Z]+[X,\z\rt Z] + (\z\lt X)\rt Z-(\z\lt Z)\rt X\\
&[\z,\x]X=[\z\lt X,\x]+[\z,\x\lt X]+ \z\lt(\x\rt X)-\x\lt(\z\rt X).
\end{align*}
Given such a matched pair, one defines a Lie algebra with
underlying vector space $\Fg\oplus\Fn$ by setting:
\begin{equation*}
[X\oplus\z,  Z\oplus\x]=([X, Z]+\z\rt Z-\x\rt X)\oplus ([\z,\x]+\z\lt Z-\x\lt X).
\end{equation*}
 Conversely, given a Lie algebra $\Fa$ and
 two Lie subalgebras $\Fg$ and $\Fn$ so that $\Fa=\Fg\oplus\Fn$ as vector  spaces,  then
 $(\Fg, \Fn)$  forms a matched pair of Lie algebras and $\Fa\cong \Fg\bowtie \Fn$ as Lie algebras.
  In this case the actions of $\Fg$ on $\Fn$ and $\Fn$ on $\Fg$  for $\z\in \Fn$ and $X\in\Fg$ are uniquely determined  by:
\begin{equation*}
[\z,X]=\z\rt X+\z\lt X
\end{equation*}
This is precisely our case, and we shall identify from now on $\Fa$ with $ \Fg\bowtie \Fn$.

\begin{lemma}\label{biLie}
The complex $\{C_{\rm top}^\bullet(\Fa), d\}$ is isomorphic to the total complex of  the bicomplex
$\{C_{\rm top}^\bullet(\Fg)\ot C^\bullet(\Fn), \p_\Fg, \p_\Fn\}$, where $\p_\Fg$ and $ \p_\Fn$
denote the respective coboundaries for Lie algebra cohomology with coefficients.
\end{lemma}
\begin{proof} The isomorphism is implemented by the map \\
$\natural:  C^{s}(\Fg\bowtie\Fn)\ra \bigoplus_{p+q=s}C^p(\Fg)\ot C^q(\Fn)$,
\begin{align*}
\natural(\omega)(Z_1,\dots, Z_p\mid \z_1, \dots, \z_q)=\omega(Z_1\oplus 0,\dots, Z_p\oplus 0, 0\oplus\z_1, \dots, 0\oplus \z_q) ,
\end{align*}
whose inverse is given by
 \begin{align*}
&\natural^{-1}(\mu\ot\nu)(Z_1\oplus\z_1, \dots,Z_{p+q}\oplus\z_{p+q})= \\
&\sum_{\s\in Sh(p,q)}(-1)^{\s}\mu(Z_{\s(1)}, \dots,Z_{\s(p)})\nu(\z_{\s(p+1)}, \dots, \z_{\s(p+q)}) .
\end{align*}
Using  the obvious identities
\begin{align*}
& [Z_i\oplus 0, Z_j\oplus 0]=[Z_i,Z_j]\oplus 0, \quad [0\oplus \z_k, 0\oplus \z_l]=0\oplus[\z_k, \z_l], \\
&[Z_i\oplus 0, 0\oplus \z_k]= -\z_k\rt Z_i-\z_k\lt Z_i ,
\end{align*}
it is straightforward to check that $\natural$ takes the coboundary $d$ of
 $C_{\rm top}^\bullet(\Fa)$ into the total coboundary of $C_{\rm top}^\bullet(\Fg)\ot C^\bullet(\Fn)$.
 \end{proof}

 \bigskip

 \begin{lemma}
Let   $X \in \Fg = T_e G$. Then
\begin{align} \label{nustar}
\nu_{ \ast (e, \psi)}(X) \, = \,  - \imath_{\ast  \psi} ({\tilde X}_\psi ) \, + \,
  \imath_{\ast  \psi} (X_\psi^\lt ), \qquad
\fl \,  \psi \in \FN ,
\end{align}
where ${\tilde X}$ denotes the corresponding left invariant vector field on $\Jc_0^\infty$
and and $X^\lt$ the vector field on $\Jc_0^\infty$ associated via the right action $\lt$ of $G$.
\end{lemma}

\begin{proof}
Let $\, f  \in C^\infty (\Jc^\infty_0)$, which we assume of the form
$\, f = f_1 \ot f_2 $ with $\, f_1 \in C^\infty (G)$ and
 $\, f_2 \in C^\infty (\FN)$. One has
\begin{align*}
 \begin{split}
& \nu_{\ast (e, \psi)}  (X) f = \,  \frac{d}{dt} \mid_{t=0} (f \circ \nu)( \exp t X , \psi)  \, = \\
 & = \frac{d}{dt} \mid_{t=0} f (\exp t X \cdot \imath (\psi \lt \exp tX)) \\
  &= \frac{d}{dt} \mid_{t=0} \big(f_1 ( \exp t X )\cdot f_2 ( \imath (\psi \lt \exp tX)) \big) \\
  &= \frac{d}{dt} \mid_{t=0} f_1 ( \exp t X )\cdot f_2 (\imath (\psi)) \, + \, f_1(e) \cdot
  \frac{d}{dt} \mid_{t=0}  f_2 ( \imath (\psi \lt \exp tX)) \\
  &=  \frac{d}{dt} \mid_{t=0} f ( \exp t X \cdot \psi^{-1}) \, + \,
  \frac{d}{dt} \mid_{t=0}  f (\imath(\psi \lt \exp tX)) \\
  &= -\frac{d}{dt} \mid_{t=0} f ( \imath( \psi \exp t X)) \, + \, \imath_{\ast  \psi} (X_\psi^\lt ) f  \\
&=\, - \imath_{\ast  \psi} ({\tilde X}_\psi )f  \, + \, \imath_{\ast  \psi} (X_\psi^\lt ) f  .
      \end{split}
\end{align*}
 \end{proof}

 \begin{corollary} \label{biLie2}
 The map $\,  \mu_e : C^\bullet_{\rm top}(\Fa) \ra
 \wdg^\bullet \Fg^\ast \ot   \wdg_{\rm top}^\bullet \Fn^\ast$, $\, \mu_e (\om) := \mu (\om) \mid_e$
 coincides with the isomorphism
 $\, \natural:  C_{\rm top}^\bullet (\Fg\bowtie\Fn)\ra C^\bullet(\Fg)\ot C_{\rm top}^\bullet(\Fn)$.
 \end{corollary}

 \begin{proof} The evaluation of  $\, \mu (\om)$ at the identity element $e \in \Gb$
 only involves the map $\nu_{ \ast (e, e)} : \Fg \ra \Fa$. Applying \eqref{nustar},
 for $\psi =e$, one simply obtains
 \begin{align*}
 \nu_{ \ast (e, e)} (X) \, = \, X , \qquad \fl \, X \in \Fg ,
 \end{align*}
 because obviously $\, X_e^\lt \, = \, 0$.
  \end{proof}

\bigskip

%%%%%%%%%%%%%%%%%%%%%

Identifying the two complexes as in Corollary \ref{biLie2} and relying
on Lemma \ref{vEcomplex2}, we shall view the map $\Ec$
as a map of bicomplexes $\Ec_{\bullet, \bullet} : C_{\rm top}^\bullet (\Fn, \wedge^\bullet \Fg^\ast) \ra
 C_{\rm pol}^\bullet (\FN, \wedge^\bullet \Fg^\ast)$.

\begin{proposition}\label{nn*}
The map $\Ec_{\bullet, \bullet} : C_{\rm top}^\bullet (\Fn, \wedge^\bullet \Fg^\ast) \ra
 C_{\rm pol}^\bullet (\FN, \wedge^\bullet \Fg^\ast)$
 induces a quasi-isomorphism  of total complexes.
\end{proposition}

\proof
In order to show that  the map of bicomplexes
$\Ec_{\bullet, \bullet} $ implements a quasi-isomorphism
 of total complexes
it suffices to check that for  each row $q \in \Nb$,
 $\Ec_{\bullet, q}:C_{\rm top}^\bullet (\Fn, \wg^q \Fg^\ast  )
 \ra \Cc_{\rm pol}^\bullet (\FN, \wg^q \Fg^\ast )$
is a quasi-isomorphism. We recall that, by Lemma
 \ref{lem:redunit}, this map associates to a right invariant
 form  ${\tilde \om} \in \Om^p (\FN, \wg^q \Fg^\ast )^\FN$ the cochain
   \begin{align*}
\Ec_{p, \bullet} (\om)(\psi_0,\dots, \psi_p) \, = \,
  \int_{\D(\psi_0,\dots, \psi_p)} \,\mu (\om)\,
  \in \, \wedge^\bullet \Fg^\ast .
\end{align*}
 This shows that $\, \Ec_{\bullet, q}$ coincides with the standard
explicit chain map implementing the van Est homomorphism
for Lie groups ({\rm cf. e.g.} \cite[Prop. 5.1]{Dupont}). However, since $\FN$ is
only a projective limit of Lie groups, the usual argument based
on continuously injective resolutions ({\it cf.} \cite{HochMost})
does not properly apply. Instead, we shall first give a direct argument for
the isomorphism
\begin{align} \label{abs-iso}
  H_{\rm top}^\bullet (\Fn, \wg^\bullet \Fg^\ast ) \cong H_{\rm pol}^\bullet (\FN, \wg^\bullet \Fg^\ast ) ,
   \end{align}
  along the lines of the original proof by van Est \cite{vEst}, and only then
  conclude that it is implemented by the chain map $\, \Ec_{\bullet, q}$.

   To this end, let us consider the bicomplex
$\Cc_{\rm pol}^\bullet \big(\FN , \Om_{\rm pol}^\bullet (\FN,\wg^\bullet \Fg^\ast  )\big)$
consisting of {\it inhomogeneous
cochains} on $\FN$ with values in $\wg^\bullet \Fg^\ast$-valued forms
on $\FN$ with polynomial coefficients. A $(p, q)$-cochain in this bicomplex
is a function $\om :  \underset{p\,\text{times}}{\underbrace{\FN \times \ldots \times \FN}} \ra
 \Om_{\rm pol}^q (\FN,\wg^\bullet \Fg^\ast)$
which depends polynomially on the jet coordinates. The group $\FN$ acts on the coefficients
$\a \in \Om_{\rm pol}^\bullet (\FN,\wg^\bullet \Fg^\ast )$ via
\begin{align} \label{actonform}
  \psi \star \a \, = \, \psi \rt  R_\psi^\ast \a \, , \qquad  \psi  \in \FN .
   \end{align}
The horizontal coboundary is the group cohomology
 operator
\begin{align} \label{horb}
&d_{\rm h}\om(\psi_1, \dots, \psi_{p+1})\, = \, \psi_1 \star \om(\psi_2,\dots, \psi_{p+1}) \, + \\
\notag &\sum_{i=1}^{p}(-1)^{i} \om(\psi_1,\dots, \psi_i\psi_{i+1} \dots , \psi_{p+1})+
(-1)^{p+1} \om(\psi_1,\dots, \psi_p) .
\end{align}
The vertical coboundary is obtained by applying the exterior derivative $d$,
\begin{align} \label{verb}
&(d_{\rm v}\om) (\psi_1, \dots, \psi_{p})\, = \, d \big(\om (\psi_1, \dots, \psi_{p})\big) ,
\end{align}
operation which commutes with the action \eqref{actonform}, and therefore with
the horizontal coboundary.

 One augments the resulting bicomplex
\begin{align}\label{ve++}
\begin{xy} \xymatrix{ \vdots & \vdots
 &\vdots  & \\
\Cc_{\rm pol}^0 \big(\FN , \Om_{\rm pol}^2 \big)
  \ar[r]^{d_{\rm h}} \ar[u]^{d_{\rm v}}&
\Cc_{\rm pol}^1 \big(\FN , \Om_{\rm pol}^2 \big)
  \ar[r]^{d_{\rm h}} \ar[u]^{d_{\rm v}}
 &\Cc_{\rm pol}^2 \big(\FN , \Om_{\rm pol}^2 \big)
 \ar[u]^{d_{\rm v}}  \ar[r]^{d_{\rm h}}&\hdots\\
\Cc_{\rm pol}^0 \big(\FN , \Om_{\rm pol}^1 \big)
 \ar[r]^{d_{\rm h}} \ar[u]^{d_{\rm v}}&\Cc_{\rm pol}^1 \big(\FN , \Om_{\rm pol}^1 \big)
\ar[r]^{d_{\rm h}} \ar[u]^{d_{\rm v}}&\Cc_{\rm pol}^2 \big(\FN , \Om_{\rm pol}^1 \big)
\ar[u]^{d_{\rm v}} \ar[r]^{d_{\rm h}}&\hdots\\
\Cc_{\rm pol}^0 \big(\FN , \Om_{\rm pol}^0 \big) \ar[r]^{d_{\rm h}}\ar[u]^{d_{\rm v}}& \Cc_{\rm pol}^1
\big(\FN , \Om_{\rm pol}^0 \big)
 \ar[r]^{d_{\rm h}}\ar[u]^{d_{\rm v}}&\Cc_{\rm pol}^2 \big(\FN , \Om_{\rm pol}^0 \big)\ar[u]^{d_{\rm v}}
\ar[r]^{d_{\rm h}}& \hdots  }
\end{xy}
\end{align}
 with the extra column
\begin{align} \label{quasi1}
\iota_{\rm h}: \Ker \left[ d_{\rm h} : \Cc_{\rm pol}^0 \big(\FN , \Om_{\rm pol}^\bullet \big)
 \ra  \Cc_{\rm pol}^1 \big(\FN , \Om_{\rm pol}^\bullet \big) \right] \hookrightarrow
 \Cc_{\rm pol}^0 \big(\FN , \Om_{\rm pol}^\bullet \big) ,
 \end{align}
 and with the bottom row
\begin{align}   \label{quasi2}
\iota_{\rm v}: \Ker \left[ d_{\rm v} : \Cc_{\rm pol}^\bullet \big(\FN , \Om_{\rm pol}^0 \big)
 \ra \Cc_{\rm pol}^\bullet \big(\FN , \Om_{\rm pol}^1 \big) \right] \hookrightarrow
\Cc_{\rm pol}^\bullet \big(\FN , \Om_{\rm pol}^0 \big) .
\end{align}
The added column gives the de Rham complex of right
$\FN$-invariant $\wg^\bullet \Fg^\ast $-valued
forms $\{ \Om_{\rm pol}^\bullet (\FN,\wg^\bullet \Fg^\ast )^\FN , d \}$. The latter is
isomorphic to $C_{\rm top}^\bullet (\Fn, \wg^\bullet \Fg^\ast )$, and so
the cohomology of this column is $H_{\rm top}^\bullet (\Fn, \wg^\bullet \Fg^\ast )$.
On the other hand the added bottom row is
the complex of inhomogeneous cochains $\Cc_{\rm pol}^\bullet(\FN , \wg^\bullet \Fg^\ast )$,
which computes  the group cohomology
$H_{\rm pol}^\bullet (\FN, \wg^\bullet \Fg^\ast )$.

Furthermore, each of the augmented rows and columns are exact. In order
to check this, it will be convenient
to describe the forms on $\FN$ in terms of the canonical trivialization of the
cotangent bundle. Specifically, we shall associate to each
$\s \in \Om_{\rm pol}^q (\FN,\wg^\bullet \Fg^\ast )$ a function
$\hat{\s}: \FN \ra \wg_{\rm top}^q \Fn^\ast \ot \wg^\bullet \Fg^\ast$ as follows.
Expressing $\s$ as a finite sum
 $\displaystyle \,  \s \, = \, \sum f \, \td{\b}$, with  $f \in C^\infty (\FN, \wg^\bullet \Fg^\ast)$ and
with $\td{\b}$'s right invariant forms
on $\FN$ associated to a basis $\{\b \}$
of $\wg_{\rm top}^q \Fn^\ast $,
we define
\begin{align} \label{trivialize}
 \hat{\s} (\rho) \, := \,  \sum_I \,  \rho \rt f (\rho) \, \b  .
 \end{align}
Note that if $\, \psi \in \FN$, then
\begin{align*}
 \psi \star \s \,=  \, \sum_I \, \psi \rt R_\psi^\ast f \cdot \td{\b}
    \end{align*}
hence
\begin{align} \label{traction}
 \widehat{\psi \star \s} (\rho) \, = \,  \sum_I \, \rho\, \psi \rt
 f (\rho \psi) \cdot \b  \, = \,  \hat{\s} (\rho \psi) .
   \end{align}
Thus, under the identification of $ \Om_{\rm pol}^q (\FN,\wg^\bullet \Fg^\ast )$ with
$C^\infty (\FN, \wg_{\rm top}^q \Fn^\ast \ot \wg^\bullet \Fg^\ast)$ via \eqref{trivialize},
 the action of $\FN$ simply becomes the right translation. Working in this
 picture, the exactness of the rows can now be
 exhibited by verifying that the linear operators
$\kappa_h^p: \Cc_{\rm pol}^{p+1} \big(\FN , \Om_{\rm pol}^\bullet \big)
 \ra \Cc_{\rm pol}^{p} \big(\FN , \Om_{\rm pol}^\bullet \big)$, defined by
 \begin{align} \label{horhomot}
(\kappa_h^p \om)(\psi_1, \dots,\psi_p) (\psi) \, = \,
\om( \psi, \psi_1, \dots,\psi_p)(e)  ,  \qquad \psi \in \FN ,
\end{align}
form a contracting homotopy. Indeed,

 \begin{align*}
 &d_{\rm h}(\kappa_h^p \om) (\psi_1, \dots,\psi_{p+1}) (\psi)
 \, = \, \big(\psi_1 \star (\kappa_h^p \om) (\psi_2, \dots,\psi_{p+1}§)\big)(\psi) \, + \\
&+
\sum_{i=1}^{p}(-1)^{i}   (\kappa_h^p \om)(\psi_1,\dots, \psi_i\psi_{i+1} \dots ,
\psi_{p+1}) (\psi) \, + \\
&+(-1)^{p+1}   (\kappa_h^p \om)(\psi_1,\dots, \psi_p)(\psi)   \,
\,= \, \om(\psi \psi_1, \psi_2, \dots,\psi_p)(e)  +\\
 &+ \sum_{i=1}^{p}(-1)^{i}  \om(\psi, \psi_1,\dots, \psi_i\psi_{i+1} \dots ,
\psi_{p+1})(e) + \\
&  +(-1)^{p+1} \om(\psi, \psi_1,\dots, \psi_p)(e)  \, ,
  \end{align*}
 while on the other hand
 \begin{align*}
&(\kappa_h^{p+1} d_{\rm h} \om)(\psi_1, \dots, \psi_{p+1})(\psi)\, =\,
d_{\rm h}  \om ( \psi, \psi_1, \dots,\psi_{p+1})(e) \, = \\
& = \,  \om (\psi_1, \dots, \psi_{p+1})(\psi)-
 \om (\psi\psi_1, \dots, \psi_{p+1})(e)\, +\\
&+ \sum_{i=1}^{p}(-1)^{i-1}\om (\psi,\psi_1,\dots,
\psi_i\psi_{i+1} , \dots , \psi_{p+1})(e) \\
& + (-1)^{p} \om  (\psi, \psi_1,\dots \psi_p)(e).
  \end{align*}
Summing up one obtains the homotopy identity
 \begin{align*}
d_{\rm h}(\kappa_h^p \om) (\psi_1, \dots,\psi_{p+1}) \, +\, (\kappa_h^{p+1} d_{\rm h} \om)(\psi_1, \dots, \psi_{p+1})
\, = \, \om (\psi_1, \dots, \psi_{p+1}) .
   \end{align*}

To prove that the columns are exact one employs, as in the standard proof of the Poincar\'e Lemma, the contraction along the radial vector field
 \begin{align*}
\Xi \, = \, \sum \a \frac{\p}{\p \a} ,
 \end{align*}
in the affine coordinates $\,  \a^\bullet_{\bullet \ldots \bullet} (\psi)$ of $j_0^\infty(\psi)$,  $\psi \in \FN$, to construct the linear operators $\quad \chi^q : \Om_{\rm
pol}^q(\FN,\wg^\bullet \Fg^\ast  ) \ra \Om_{\rm pol}^{q -1} (\FN,\wg^\bullet \Fg^\ast  )$,
 \begin{align*}
\chi^q (\om) \mid_{\psi}  \, = \, \int_0^1 \i_\Xi (\om)\mid_{t  \psi} \frac{dt}{t} .
\end{align*}
One has
 \begin{align*}
\chi^{q} \circ d_v = d_v \circ \chi^{q-1}  \qquad \text{and} \qquad \chi^q \circ \Lc_\Xi  \, = \, \Id ,
\end{align*}
where $\quad \Lc_\Xi \,= \, d_v \circ \i_\Xi\,  +\,  \i_\Xi\circ d_v \,$ denotes the Lie derivative along $\Xi$.  Therefore
 \begin{align*}
 d_v\circ (\chi^{q-1} \circ \i_\Xi) \,  +\, (\chi^q \circ \i_\Xi) \circ d_v \, = \, \Id ,
 \end{align*}
 showing that the operators
  \begin{align*}
(\kappa_v^{q} \om) (\psi_1,  \ldots , \psi_p) \, = \, \chi^{q-1} \circ \i_\Xi
\big( \om (\psi_1,  \ldots , \psi_p)\big) ,
\end{align*}
give contracting  homotopies for the columns.

Since both the rows and the columns are exact, the edge homomorphisms
$\iota_{\rm h}$ of \eqref{quasi1}
 and $\iota_{\rm v}$ of \eqref{quasi2} induce
isomorphisms with the cohomology of the total complex, which in turn gives
the isomorphism \eqref{abs-iso}.
Actually, for each $q \in \FN$,
 one can display an explicit quasi-isomorphism
\begin{align*}
\digamma_q = j_{\rm v} \circ \iota_{\rm h} :
 C_{\rm top}^\bullet (\Fn, \wg^q\Fg^\ast ) \ra \Cc_{\rm pol}^\bullet (\FN, \wg^q\Fg^\ast) ,
   \end{align*}
  by composing the inclusion $\iota_{\rm h}$ with a chain homotopic inverse for $\iota_{\rm v}$,
\begin{align*}
j_{\rm v}:  \Cc_{\rm pol}^\bullet (\FN, \Om^\bullet_{\rm pol}(\FN, \wg^q\Fg^\ast )) \ra
\Cc_{\rm pol}^\bullet (\FN, \wg^q\Fg^\ast) ,
   \end{align*}
manufactured out the vertical homotopy operators and the horizontal coboundary
 by means of the ``collating formula'' of  \cite[Prop. 9.5]{BottTu}. For a Lie algebra cocycle
 $\, \b \in Z_{\rm top}^p (\Fn, \wg^q\Fg^\ast )$, it gives
\begin{align} \label{exve2}
\digamma_q (\b) \, = \, (-1)^p (d_{\rm h} \circ \kappa_v)^p (\td{\b}) .
   \end{align}
On the other hand, the conversion of $ \Ec_{\bullet, q} $  into a inhomogeneous
group cochain,
   \begin{align}  \label{nhvest}
  E_{p, q} (\om)(\psi_1,\dots, \psi_p) \,= \,
 \Ec_{p, q} (\om)(1, \psi_1,\psi_1 \psi_2 , \dots, \psi_1 \psi_2 \cdots \psi_p)  ,
\end{align}
gives also a chain map
\begin{align*}
E_{\bullet, q} :
 C_{\rm top}^\bullet (\Fn, \wg^q\Fg^\ast ) \ra \Cc_{\rm pol}^\bullet (\FN, \wg^q\Fg^\ast) .
   \end{align*}
Since for any each of the quotient nilpotent Lie groups in the projective limit the
two chain maps $E_{\bullet, q}$ and $\digamma_q$ are automatically chain homotopic,
we can finally conclude that $\Ec_{\bullet, q}$ is a quasi-isomorphism.
\endproof

\bigskip

We are now in a position to strengthen Proposition \ref{3id} and conclude this section
with the following result.

 \begin{theorem} \label{thm:KvE}
There is a canonical  quasi-isomorphism
 between the Lie algebra cohomology and the
 Hopf cyclic cohomology complexes associated to
 an infinite primitive Cartan-Lie pseudogroup
 $$
 \Ec_{\rm LH} : C_{\rm top}^\bullet (\Fa)\ra C^{\bullet, \bullet}_{\rm c-w}(\Fg^\ast ,\Fc)
 \cong C^\bullet_\Uc(\Fc, V^\bullet(\Fg^\ast)) .
 $$
\end{theorem}

Remark \ref{rem:etod} describes the precise transition between
the maps $\Cc : C_{\rm top}^\bullet (\Fa) \ra
 C_{\rm top}^\bullet (\FN,   \Omega_\bullet (G))$ of Proposition
\ref{vEcomplex} and $ \Ec_{\rm LH}$, which in essence encodes the relationship between
the classical and the Hopf cyclic constructions of transverse characteristic classes.

%%%%%%%%%%%%%%%%%%

\section{Hopf cyclic characteristic maps}\label{S: HCmaps}

This section is devoted to the construction of several characteristic maps that deliver
the universal Hopf cyclic classes to the cyclic cohomology of the \'etale holonomy groupoids.
For the clarity of the exposition, we shall explicitly treat only the case of {\em action
holonomy  groupoids}, \ie of the form $\, G \al \G$, with $\G$ a discrete subgroup of $\Gb$.
The convolution algebra of such a groupoid is the crossed product
$\Ac_{c, \G}(G):=C_c^{\infty}(G)\rtimes \G$. Each characteristic map will land into the
suitable model of cyclic cohomology of $\Ac_{c, \G}(G)$, that matches the source model
of Hopf cyclic cohomology of $\Hc = \Hc (\Gb)$. We shall show however, and this is
the main result of this section, that all these characteristic maps are equivalent.

\subsection{Characteristic maps induced by Hopf actions} \label{Ss: Hact}
The first characteristic homomorphism connects the bicocyclic module in
diagram   \eqref{UF},  $C (\Uc, \Fc, \Cb_\d)$,
 to the Getzler-Jones bicocyclic module~ \cite{gj}, describing the
 cyclic cohomology of $\Ac_{c, \G}(G)=C_c^{\infty}(G)\rtimes \G $.

Along with $\Ac_{c, \G}(G)$, we shall consider its opposite algebra $\Ac_{c, \G}(G)^o$.
The latter is the crossed product algebra $\Cb\G\ltimes C^{\infty}_c(G)$, where
 $\G$ acts on  $C^{\infty}_c(G)$ by the  natural  right action,
 with multiplication rule
 \begin{equation}
U_{\phi_1} g \, U_{\phi_2} f=U_{\phi_1\phi_2}(g\circ\td \phi_2)f.
 \end{equation}
 As before, if $ \phi \in \Gb$, $\td \phi$ stands for its lift
to a diffeomorphism of $G$. To avoid excessive notation, from now on
we shall suppress the notational distinction between the two, and reinstate it
only when there is danger of confusion.

 We recall below the bicomplex
  $\{C^{(\bullet,\bullet)}(C_c^{\infty}(G),\G),\vta,\vd_i,\hta,\hd_j \}$ of~\cite{gj}, with
 cochains
\begin{align} \label{GJ-bicocyclic}
C^{(p,q)}(C_c^{\infty}(G),\G):=\Hom(C^{\infty}_c(G)^{\ot p+1}\ot \Cb\G^{\ot q+1}, \Cb).
\end{align}
With the abbreviated notation
\begin{align*}
\td{g}:=  \,g_0\ot g_1 \odots g_{p} , \qquad
\td{U}_\phi := U_{\phi_0}\ot \cdots \ot U_{\phi_q} ,
\end{align*}
the vertical cyclic structure is given by
 \begin{align*}
&\vta(\mu)( \td{g}\mid \td{U}_\phi)= \mu(g_p\circ (\phi_0\dots \phi_q)\ot g_0 \odots g_{p-1}\mid  {U_\phi} ) ,\\
&\vd_i(\mu)(\td{g}\mid {\td{U}_\phi})=
\mu(g_0\odots g_{i}g_{i+1}\odots g_{p+1}\mid {U_\phi} ),\qquad 0 \le i \le p , \\
&\vd_{p+1}(\mu)(\td{g}\mid \td{U}_\phi)= \mu(  g_{p+1}\circ (\phi_0\dots \phi_q) \,g_0\ot g_1 \odots g_{p})\mid U_\phi) ,
\end{align*}
and horizontal cyclic structure is determined by
\begin{align*}
&\hta(\mu)( \td{g}\mid \td{U}_\phi)= \mu\left(  g_0\circ {\phi_{q}}^{-1}\odots  g_p\circ {\phi_{q}}^{-1} \mid U_{\phi_q}\ot U_{\phi_0}\odots U_{\phi_{q-1}}\right) \\
&\hd_i(\mu)(\td{g}\mid \td{U}_\phi)= \mu\left(\td{g}\mid U_{\phi_0}\odots U_{\phi_{i}\phi_{i+1}}\odots U_{\phi_{q+1}} \right) , \;\;\; 0 \le i \le q  \\
&\hd_{q+1}(\mu)(\td{g}\mid {\td{U}_\phi})= \mu\left(   g_0\circ \phi_{q+1}^{-1}\odots  g_p\circ
\phi_{q+1}^{-1} \mid  U_{\phi_{q+1}\phi_{0}}\ot U_{\phi_1}\odots U_{\phi_{q}}\right).
\end{align*}
It is shown in~\cite{gj}, its total complex is quasi-isomorphic with the
 total complex of $C^\bullet(\Ac_{c, \G}(G)^{o},b,B)$.
We also recall from~\cite{gj} the isomorphism between the diagonal subcomplex
$\Cc_{\rm diag}^\bullet (C_c^{\infty}(G),\G)$ and $C^\bullet(\Ac_{c, \G}(G)^{o},b,B)$. Specifically,
$\Psi_{\G}:  C^{n,n}(C_c^{\infty}(G),\G) \ra C^n(\Ac_{c, \G}(G)^o)$, is given by
\begin{align}
&\Psi_{\G}(\mu)(U_{\phi_0}g_0\odots U_{\phi_n}g_n)=\\\notag
&~~~~~~~~~~~~~~~~~~~~~~\mu(  g_0\circ (\phi_1\cdots \phi_n)\odots g_{n-1}\circ
\phi_n\ot g_n \mid U_{\phi_0}\odots U_{\phi_n}) ;
\end{align}
 $\Psi_\G$ is a cyclic  isomorphism, whose inverse $\Phi_{\G}$ is  given by
\begin{align*}
&\Psi^{-1}_{\G}(\mu)( g_0\odots g_n\mid U_{\phi_0}\odots U_{\phi_n})= \\
&~~~~~~~~~~~~~~~~~\mu(\;U_{\phi_0} g_0\circ
(\phi_1\cdots\phi_n)^{-1}\odots  U_{\phi_{n-1}} g_{n-1}\circ{\phi_n^{-1}}\ot U_{\phi_n}\;g_n).
\end{align*}

According to \cite[Proposition 2.18]{mr09},
there exists an isomorphism of Hopf algebras $\;\i:  \Hc^{\cop}(\Gb)_{\rm ab} \ra \Fc(\Gb)$
that sends any operator to the corresponding function evaluated at identity.
We denote by $\dbar = \i^{-1}: \Fc(\Gb)\ra \Hc^{\cop}_{\rm ab}(\Gb)$ its  inverse, and define
$\gbar:\Fc(\Gb)\ra {\rm Fun}(\G,C^{\infty}(G))$ by
\begin{align}\label{betamap}
\gbar(f)(\phi)=\dbar(S(f))(U_\phi)U^\ast_\phi .
\end{align}

\begin{lemma} \label{beta}For any $f\in\Fc(\Gb)$, $\gbar(f)$   satisfies the  cocycle
property
\begin{equation}\label{cocycle-condition}
\gbar(f)(\phi_1\phi_2)=\gbar(f\ps{1})(\phi_1)\gbar(f\ps{2})(\phi_2)\circ\td\phi_1^{-1} ,
\qquad \phi_1,\phi_2 \in \G.
\end{equation}
\end{lemma}

\begin{proof}
  Since $\dbar$ is anti coalgebra map and  $\Ac_{c, \G}(G)$ is a  $\Hc(\Gb)$ module algebra we  have
\begin{align*}
&\gbar(f)(\phi_1\phi_2)=\dbar(S(f))(U_{\phi_1}U_{\phi_2})U^\ast_{\phi_1\phi_2}=\dbar(S(f\ps{1}))(U_{\phi_1})\dbar(S(f\ps{2}))
(U_{\phi_2})U^\ast_{\phi_1\phi_2}=\\
&\gbar(f\ps{1})(\phi_1)U_{\phi_1}\gbar(f\ps{2})(\phi_2) U_{\phi_2} U^\ast_{\phi_1\phi_2}=\gbar(f\ps{1})(\phi_1)\gbar(f\ps{2})(\phi_2)\circ{{\td\phi_1^{-1}}}.
\end{align*}
\end{proof}

\begin{lemma} For any $ \, u \in \Uc (\Fg)$ and $\phi \in \Gb$ one has
\begin{equation}\label{g-u}
U_\phi uU^\ast_\phi \, = \, \gbar(u\ns{1})(\phi)u\ns{0}.
\end{equation}
\end{lemma}
\begin{proof}
To begin with, recall \cf \eqref{gij} the coaction
$\, \Db : \Uc (\Fg) \ra \Uc (\Fg) \ot \Fc$
is defined as follows:
\begin{align*}
u\ns{0} \ot u\ns{1} = \sum_I X_I \ot \g_u^I (\psi) (e) \quad \iff \quad
 U_\psi uU^\ast_\psi = \sum_I \g_u^I (\psi) X_I  ,
 \end{align*}
where $\{X_I\}$ stands for a PBW basis of $\Uc (\Fg)$.

Let us first check the claimed identity for the case when $u = X_i \in \Fg$, and
\begin{align} \label{x1}
U_\phi X_i U^\ast_\phi = \sum_j \g_i^j (\phi) X_j.
 \end{align}
We need to show that
if $\, \eta_i^j(\psi) = \g_i^j (\psi) (e)$ , $\fl \, \psi \in \FN$,  then
\begin{align} \label{u=X}
 \dbar (S(\eta_i^j))(U_\phi) \, = \, \g_i^j (\phi) \, U_\phi, \quad \fl \, \phi \in \Gb.
 \end{align}
Recall now from \cite[\S 1.2]{mr09}
the operators $\D_i^j (U_\phi^\ast) \, = \, \g_i^j (\phi) U_\phi^\ast $,
 which arise from the action formula
\begin{align} \label{x2}
X_i (U_\phi^\ast b) =  \sum_j \D_i^j(U_\phi^\ast) \,X_j (b) , \quad b \in C^\infty (G)  ,
 \end{align}
 and satisfy, by the very definition of $\dbar$,
 \begin{align} \label{qd}
 \D_i^j \, = \,  \dbar (\eta_i^j) .
 \end{align}
Applying $\, U_\phi$ on both sides of \eqref{x2} and comparing with \eqref{x1}, one sees that
\begin{align} \label{qd2}
 \g_i^j(\phi) \, = \,  U_\phi \, \g_i^j (\phi) \, U_\phi^\ast
  \, = \,\g_i^j (\phi) \circ \td{\phi}^{-1} \, = \, - \g_i^j (\phi^{-1})  .
 \end{align}
Then
\begin{align*}
S(\eta_i^j)(\psi)\, = \, \eta_i^j(\psi^{-1}) \, = \, \g_i^j (\psi^{-1}) (e) \, = \, - \g_i^j (\psi) (e)
\, = \, - \eta_i^j (\psi) ,
  \end{align*}
and so by \eqref{qd} and \eqref{qd2}
\begin{align*}
\dbar(S(\eta_i^j))(U_\phi) \, = \, - \d_i^j (U_\phi)  \, = \, - \g_i^j (\phi^{-1})\, U_\phi
 \, = \,  \g_i^j(\phi)\, U_\phi ,
  \end{align*}
  thus verifying \eqref{u=X}.

 The compatibility of the coaction with the multiplication, \cf~\cite[Lemma 2.12]{mr09}
 allows to extend in a straightforward manner the validity of this identity to $\Uc (\Fg)$.
 For the convenience of the reader, we give the detailed verification for $\,  u= XY$.
 Recall that in view of \eqref{mp2} one has
\begin{align*}
&(XY)\ns{0}\ot (XY)\ns{1}= X\ps{1}\ns{0}Y\ns{0}\ot X\ps{1}\ns{1}(X\ps{2}\rt Y\ns{1})=\\
&X\ns{0}Y\ns{0}\ot X\ns{1}Y\ns{1}+ Y\ns{0}\ot X\rt Y\ns{1}.
\end{align*}
 We observe that
  \begin{equation*}
  X\ns{0}\ps{1}\ot X\ns{0}\ps{2}\ot X\ns{1}= X\ns{0}\ot 1\ot X\ns{1}+ 1\ot X\ns{0}\ot X\ns{1}.
  \end{equation*}
Since
$\i$ is $U(\Fg)$-module map, it follows that
 $\gbar(Z\rt f)(\phi)= Z(\gbar(f)(\phi))$  for any $Z\in\Fg$ and any  $f\in \Fc$.
Acting on an arbitrary function $g\in C^{\infty}(G)$, we write
 \begin{align*}
 &(U_\phi XYU^\ast_\phi)(g)= (U_\phi XU^\ast_\phi\, U_\phi YU^\ast_\phi)(g)= \\
 &\gbar(X\ns{1})(\phi)X\ns{0} \gbar(Y\ns{1})(\phi) Y\ns{0})(g)=  \gbar(X\ns{1})(\phi)X\ns{0} (\gbar(Y\ns{1})(\phi) Y\ns{0}(g))=\\
 &\gbar(X\ns{1})(\phi)X\ns{0}\ps{1} (\gbar(Y\ns{1})(\phi) )X\ns{0}\ps{2}(Y\ns{0}(g))= \\
  &\gbar(X\ns{1})(\phi)\gbar(Y\ns{1})(\phi) X\ns{0}Y\ns{0}(g)+ \gbar(X\ns{1})(\phi) X\ns{0}(\gbar(Y\ns{1})(\phi)) Y\ns{0}(g)=\\
  &\gbar((XY)\ns{1})(\phi) (XY)\ns{0}(g).
 \end{align*}
\end{proof}

\begin{lemma}
For any $\phi \in \Gb$,
\begin{equation}\label{S(f)(psi)}
\gbar(S(f))(\phi)=\gbar(f)(\phi^{-1})\circ{\td\phi^{-1}}.
\end{equation}
\end{lemma}
\begin{proof}
 We define $M_\phi, R_\phi, L_\phi\in \Hom(\Fc, C^{\infty}(G))$ by setting
\begin{align*}
R_\phi(f)=\gbar(f)(\phi^{-1})\circ{\td\phi^{-1}} ,\quad
L_\phi(f)=\gbar(f)(\phi) , \quad
M_\phi(f)= \gbar(S(f))(\phi).
\end{align*}
Using the usual convolution product and  \eqref{cocycle-condition} one computes
\begin{align} \label{inverse}
\begin{split}
&L_\phi\ast R_\phi(f)= L_{\phi}(f\ps{1}) R_\phi(f\ps{2})= \\
& ~~~~~~~~~~~~~~~~~\gbar(f\ps{1})(\phi)\gbar(f\ps{2})(\phi^{-1})\circ{\td\phi^{-1}}
=\gbar(f)(\phi\phi^{-1})=\ve(f).
\end{split}
\end{align}
On the other hand, because $\gbar$ is algebra map, one easily sees
\begin{align*}
L_\phi\ast M_\phi\, =\, \ve \, =\, M_\phi \ast L_\phi .
\end{align*}
Together with \eqref{inverse}, this implies  $\, R_\phi \, =\, M_\phi $.
\end{proof}
\begin{lemma} For any $\phi\in \Gb$, $f\in \Fc$, and   $u\in U(\Fg)$ we have
\begin{equation}\label{g-u-f}
\gbar(u\rt f)(\phi)= u\ns{0}(\gbar(f)(\phi))\gbar(u\ns{1})(\phi).
\end{equation}
\end{lemma}
\begin{proof}
Using the fact that $\dbar$ is $U(\Fg)$-module algebra  map  and Lemma \ref{S-linear} we see that
\begin{align*}
&\gbar(u\rt f)(\phi)= \dbar(S(u\rt f))(U_\phi)\,U^\ast_\phi=\left(\dbar(S(u\ns{1}))\dbar(u\ns{0}\rt S(f))\right)(U_\phi) \,U^\ast_\phi=\\
&\left(\dbar(S(u\ns{1}))\dbar(u\ns{0}\rt S(f))\right)(U_\phi) \,U^\ast_\phi= \dbar(S(u\ns{1}))\left(u\ns{0}(\gbar(f)(\phi)\,U_\phi)\right)\,U^\ast_\phi=\\
&=\gbar(u\ns{1})(\phi)u\ns{0}(\gbar(f)(\phi)).
\end{align*}
\end{proof}

We want to define a characteristic
 map $\chi:C^{p,q}_\Hc(\Uc,\Fc,\Cb_\d)\ra C^{p,q}(C^{\infty}_c(G),\G)$, by the formula
\begin{align}\label{char}
&\chi\left(\one\ot \td u\ot \td f\right)(\td g\mid \td U )=\\ \notag
&\tau\left(u^0(g_0)u^1(g_1)\dots u^p(g_p)\dbar(S(f^0))(U_{\phi_0})  \dots \dbar(S(f^{q}))(U_{\phi_q})\right).
\end{align}
Here $C_\Hc(\Uc,\Fc,\Cb_\d)$ is the bicocyclic module  defined in \eqref{FX},
and $\tau : \Ac_{c, \G}(G) \ra \Cb$ is the invariant trace  defined by
\begin{align} \label{inv-tr}
 &\displaystyle\tau(gU^\ast_\psi)= \left\{ \begin{matrix}  \displaystyle \int_{G} g\, \varpi,  &\text{if}\quad\psi= \Id\\&\\
0,&\text{otherwise.}\end{matrix}\right.
\end{align}
Equivalently, $\chi:C^{p,q}_\Hc(\Uc,\Fc,\Cb_\d)\ra C^{p,q}(C^{\infty}_c(G),\G)$ is given by
\begin{align*}
&\chi\left(\one\ot \td u\ot \td f\right)(\td g\mid \td U )=\\
&\displaystyle \int_G u^0(g_0)u^1(g_1)\dots u^p(g_p)\gbar(f^0)(\phi_0)\gbar(f^1)(\phi_1)\circ{\td\phi_0^{-1}}  \dots
\gbar(f^q)(\phi_q)\circ{\td\phi_{q-1}^{-1}\dots \td\phi_{0}^{-1}}\varpi,
\end{align*}
if $\psi_0\dots \psi_q=\Id$, and $0$ otherwise.

\begin{lemma}
The characteristic map $\chi$ of \eqref{char} is well-defined.
\end{lemma}

\begin{proof}
First we  recall that
\begin{align*}
(f\acl u)\cdot (\td u\ot \td f)= u\ps{1}u^0\odots u\ps{p+1}u^p \ot (f\acl u\ps{p+2})\cdot \td f,
\end{align*}
where
\begin{align*}
&(f\acl u)\cdot(f^0\odots f^q)=\\
& f\ps{1}(u\ps{1}\ns{0}\rt f^0)\ot u\ps{1}\ns{1} (u\ps{2}\ns{0}\rt f^1)\odots f\ps{q+1} u\ps{1}\ns{q}\dots u\ps{q}\ns{1} (u\ps{q+1}\rt f^q).
\end{align*}
 In the following we use the notation  $(u\rt \td u)(\td g)= u\ps{1}u^0(g^0)\dots u\ps{p+1}u^p(g^p)$. We observe that
\begin{align}\notag
& u\ps{1}\rt \td u(\td g)\gbar(u\ps{2}\ns{0}\rt f^0)(\phi_0)\gbar(u\ps{2}\ns{1} (u\ps{3}\ns{0}\rt f^1))(\phi_1)\circ{\td\phi_0^{-1}}  \cdots\\\label{chi-well-1}
&~~~~~~~~\cdots\gbar(u\ps{2}\ns{q}\dots u\ps{q+1}\ns{1} (u\ps{q+2}\rt f^q))(\phi_q)\circ{\td\phi_{q-1}^{-1}\dots \td\phi_{0}^{-1}}=
\end{align}
By using  \eqref{g-u-f} we see that
\begin{align}\notag
&\eqref{chi-well-1} =\\\notag
&u\ps{1}\rt \td u(\td g)\gbar(u\ps{2}\ns{0}\rt f^0)(\phi_0)\gbar(u\ps{2}\ns{1} (u\ps{3}\ns{0}\rt f^1))(\phi_1)\circ{\td\phi_0^{-1}}  \cdots\\\notag
&~~~~~~~~\cdots\gbar(u\ps{2}\ns{q}\dots u\ps{q+1}\ns{1} (u\ps{q+2}\rt f^q))(\phi_q)\circ{\td\phi_{q-1}^{-1}\dots \td\phi_{q-1}^{-1}}=\\\notag
&u\ps{1}\rt \td u(\td g)u\ps{2}\ns{0}(\gbar( f^0)(\phi_0))[\gbar(u\ps{2}\ns{1})(\phi_0)\cdots \gbar(u\ps{2}\ns{q+1})(\phi_q)\circ{\td\phi_{q-1}\cdots\td\phi_{0}^{-1}} ]  \cdots\\\notag
&u\ps{2+i}\ns{0}(\gbar( f^i)(\phi_i)) \circ{(\phi_0\cdots \phi_{i-1})^{-1}}[\gbar(u\ps{2+i}\ns{1})(\phi_i)\circ{(\phi_0\cdots \phi_{i-1})^{-1}}\cdots \\\notag
&\gbar(u\ps{2+i}\ns{q+1-i})(\phi_q)\circ{\td\phi_{q-1}^{-1}\cdots\td\phi_{0} ^{-1}}]   \cdots\\\label{chi-well-2}
&u\ps{q+2}\ns{0}(\gbar( f^q)(\phi_q)) \circ{(\phi_0\cdots \phi_{q-1})^{-1}}\gbar(u\ps{q+2}\ns{1})(\phi_q)\circ{(\phi_0\cdots \phi_{q-1})^{-1}}=
\end{align}
Using \eqref{cocycle-condition}  we continue as follows:
\begin{align}\notag
&\eqref{chi-well-2}=\\\notag
&u\ps{1}\rt \td u(\td g)u\ps{2}(\gbar( f^0)(\phi_0)) u\ps{3}\ns{0}(\gbar( f^1)(\phi_1)) \circ{\phi_0^{-1}} \gbar(u\ps{3}\ns{1})(\phi_1\cdots \phi_q)\circ{\phi_0^{-1}}  \cdots \\\notag
&u\ps{2+i}\ns{0}(\gbar( f^i)(\phi_i)) \circ{(\phi_0\cdots \phi_{i-1})^{-1}}\gbar(u\ps{2+i}\ns{1})(\phi_i\cdots \phi_q)\circ{(\phi_0\cdots \phi_{i-1})^{-1}}  \cdots\\\label{chi-well-3}
&u\ps{q+2}\ns{0}(\gbar( f^q)(\phi_q)) \circ{(\phi_0\cdots \phi_{q-1})^{-1}}\gbar(u\ps{q+2}\ns{1})(\phi_q)\circ{(\phi_0\cdots \phi_{q-1})^{-1}}.
\end{align}
On the other hand one uses  \eqref{g-u} together with the fact that  $(\phi_0\cdots \phi_i)^{-1}=\phi_{i+1}\cdots \phi_q$ to observe that
\begin{align}\notag
&u\ps{1}\rt \td u(\td g) u\ps{2}(\gbar(f^0)(\phi_0))\, u\ps{3}(\gbar(f^1)(\phi_1)\circ{\td\phi_0^{-1}} ) \cdots\\
&\cdots u\ps{q+2}(\gbar(f^q)(\phi_q)\circ{\td\phi_{q-1}\dots \td\phi_{0}^{-1}})=\\\notag
&u\ps{1}\rt \td u(\td g) u\ps{2}(\gbar(f^0)(\phi_0))\, \gbar(u\ps{3}\ns{1})(\phi_1\cdots\phi_q)\circ{\td\phi_{0}^{-1}}u\ps{3}\ns{0}
(\gbar(f^1)(\phi_1))\circ{\td\phi_0^{-1}}  \cdots\\\notag
& \gbar(u\ps{i+2}\ns{1})(\phi_{i}\cdots\phi_q)\circ{\td\phi_{i-1}^{-1}\cdots \td\phi_{0}^{-1}}u\ps{i+2}\ns{0}(\gbar(f^{i})(\phi_{i}))\circ{\td\phi_{i-1}^{-1}\cdots \td\phi_{0})^{-1}} \cdots \\\label{chi-well-4}
 &\gbar(u\ps{q+2}\ns{1})(\phi_{q})\circ{\td\phi_{q-1}^{-1}\cdots \td\phi_{0}^{-1}}u\ps{q+2}\ns{0}(\gbar(f^{q})(\phi_{q}))\circ{\td\phi_{q-1}^{-1}\cdots \td\phi_{0}^{-1}}.
 \end{align}
 One uses the equality  \eqref{chi-well-3}=\eqref{chi-well-4} and the fact that  $\tau$ is $\d$-invariant
 to see that
 \begin{equation}\label{chi-well-5}
 \chi\left(\one\ot (1\acl u)\cdot (\td u\ot \td f)\right)= \chi\left(\d(1\acl u)\ot \td u\ot \td f\right).
  \end{equation}
  On the other  hand  we see that
  \begin{equation}
  (f\acl 1)\cdot (\td u\ot \td f)= u^0\odots u^p\ot f\ps{1}f^0\odots f\ps{q+1}f^q,
    \end{equation}
which together with the facts  that $\gbar$ is an algebra map  and  $\d(\dbar(S(f)))= \ve(S(f))= \ve(f)=\d(f\acl 1)$, yields that
\begin{equation}\label{chi-well-6}
 \chi\left(\one\ot (f\acl 1)\cdot (\td u\ot \td f)\right)= \chi\left(\d(f\acl 1)\ot \td u\ot \td f\right).
  \end{equation}
 So \eqref{chi-well-5} and \eqref{chi-well-6} prove that
 \begin{equation*}
 \chi\left(\one\ot (f\acl u)\cdot(\td u\ot \td f)\right)= \chi\left(\d(f\acl u)\ot \td u\ot \td f\right).
 \end{equation*}
\end{proof}

\begin{proposition}\label{char-bicyclic-map}
The characteristic  map $\chi$ defined in \eqref{char}  is a  map of  bicocyclic modules.
\end{proposition}
\begin{proof}
As $\Fg$ acts by derivation on $C^{\infty}(G)$ it  is obvious that $\chi$ commutes with all vertical  cofaces except possibly very last one. Since $\dbar$ is an algebra map and anticoalgebra map and $\Hc(\Gb)$ acts on $\Ac_{c, \G}(G)$ and makes it module algebra it is also obvious that $\chi$  commutes with all horizontal cofaces  except possibly the last one. So it suffices to show that $\chi$ commutes with horizontal and vertical cyclic operators. Indeed,  let $\phi_0\dots \phi_q=id$.
\begin{align*}
&\vta\chi(\one\ot\td u\ot\td f)(g_0\odots g_p\mid U_{\phi_0}\odots U_{\phi_q})=\\
&\chi(\one\ot\td u\ot\td f)(g_p\ot g_0\odots g_{p-1}\mid U_{\phi_0}\odots U_{\phi_q})=\\
&\int_G\left[u^0(g_p)u^1(g_0)\dots u^{p-1}(g_p)\gbar(f^0)(\phi_0)\gbar(f^{1})(\phi_{1})\circ{\td\phi_0^{-1}}\dots \right.\\ &\left.~~~~~~~~~~~~~~~~~~~~~\dots  \gbar(f^q)(\phi_q)\circ{ \td\phi_{q-1}^{-1}\cdots\td\phi_{0}^{-1})}\right]\varpi.
\end{align*}
On the other hand  by using Lemma \ref{beta} and \eqref{g-u} and the fact that
$\phi_0\dots \phi_q=\Id$ we see that
\begin{align*}
&\chi\vta(\one\ot\td u\ot\td f)(g_0\odots g_p\mid U_{\phi_0}\odots U_{\phi_q})=\\
&\chi(\one\ot  u^1\odots u^{p-1}\ot u^0\ns{0} \mid u^0\ns{1}f^0\ot\dots\\
&~~~~~~~~~~~~~~~~~~~~~~~~~~~~~~~~~\dots u^0\ns{q+1}f^q)(g_0\odots g_p\mid U_{\phi_0}\odots U_{\phi_q})=\\
&\int_G\left[u^1(g_0)\dots u^{p}(g_{p-1}) u^0\ns{0}(g_p)\gbar(u^0\ns{1})(\phi_0)\gbar(f^0)(\phi_0)\dots\right.  \\
&\left.\gbar(u^0\ns{q+1})(\phi_{q})\circ{\td\phi_{q-1}^{-1}\cdots\td\phi_0^{-1}} \gbar(f^q)(\phi_{q})\circ{\td\phi_{q-1}^{-1}\cdots \td\phi_0^{-1}}\right]\varpi=\\
&\int_G\left[u^1(g_0)\dots u^{p}(g_{p-1}) u^0\ns{0}(g_p)\gbar(u^0\ns{1})(\phi_0) \cdots\gbar(u^0\ns{q+1})(\phi_{q})\circ{\td\phi_{q-1}^{-1}\cdots \td\phi_0^{-1}} \right.  \\
&\left. \gbar(f^0)(\phi_0)\cdots \gbar(f^q)(\phi_{q})\circ{\td\phi_{q-1}^{-1}\cdots \td\phi_0^{-1}}\right]\varpi=\\
&\int_G\left[u^1(g_0)\dots u^{p}(g_{p-1}) u^0\ns{0}(g_p)\gbar(u^0\ns{1})(\phi_0\dots \phi_q) \right.\\
&~~~~~~~~~~~~~~~~~~~~~~~~\left. \gbar(f^0)(\phi_0)\cdots \gbar(f^q)(\phi_{q})\circ{\td\phi_{q-1}^{-1}\cdots \td\phi_0^{-1}}\right]\varpi.
\end{align*}
Let us now show that $\chi$ commutes with the horizontal cyclic operators.

\begin{align}\notag
&\chi\hta(\one\ot\td u\ot\td f)(g_0\odots g_p\mid U_{\phi_0}\odots U_{\phi_q})=\\\notag
&\chi(u^0\ns{0}\odots u^p\ns{0}\mid f^1\odots f^q\ot S(u^0\ns{1}\dots u^p\ns{1})f^0)\\\notag
&~~~~~~~~~~~~~~~~~~~~~~~~~(g_0\odots g_p\mid U_{\phi_0}\odots U_{\phi_q})=\\\notag
&\int_G\left[ u^0\ns{0}(g_0)\dots u^p\ns{0}(g_p) \gbar(f^1)(\phi_0)\dots \gbar(f^q)(\phi_{q-1})\circ{\td\phi_{q-2}^{-1}\cdots \td\phi_0^{-1}}\right. \\\label{horizontalintegral}
&\left.\gbar(f^0S(u^0\ns{1}\dots u^p\ns{1} )(\phi_q)\circ{\td\phi_{q-1}^{-1}\cdots \td\phi_0^{-1}}\right]\varpi
\end{align}
Using \eqref{S(f)(psi)}, \eqref{g-u},  and the equality $\phi_q= (\phi_{q-1}^{-1}\cdots \phi_0^{-1})$ we see that
\begin{align} \notag
\eqref{horizontalintegral}=\notag
&\int_G\left[ U^\ast_{\phi_q}u^0U_{\phi_q}(g_0)\cdots  U^\ast_{\phi_q}u^pU_{\phi_q}(g_p)\gbar(f^1)(\phi_0)\dots \right. \\\notag
 &\left. \gbar(f^q)(\phi_{q-1})\circ{\td\phi_{q-2}^{-1}\cdots \td\phi_0^{-1}}\,\gbar(f^0)(\phi_q)\circ{\td\phi_{q-1}^{-1}\cdots \td\phi_0^{-1}}\right]\varpi=\\\notag
 &\int_G\left[ u^0(g_0\circ{\td\phi_{q}^{-1}})\circ{\td\phi_{q}}^{-1}\cdots  u^p(g_p\circ{\td\phi_{q}^{-1}})\circ{\td\phi_{q}}\,\gbar(f^1)(\phi_0)\dots \right. \\\label{horizontalintegral2}
 &\left. \gbar(f^q)(\phi_{q-1})\circ{\td\phi_{q-2}^{-1}\cdots \td\phi_0^{-1}}\,\gbar(f^0)(\phi_q)\circ{\td\phi_{q-1}^{-1}\cdots \td\phi_0^{-1}}\right]\varpi=
\end{align}

Using the fact that the volume form $\varpi$ is $\Gb$-invariant and also $\phi_0\dots\phi_q=\Id$ we observe that
\begin{align*}
&\eqref{horizontalintegral2}=\\
&\int_G\left[ u^0(g_0\circ{\td\phi_{q}^{-1}})\cdots  u^p(g_p\circ\td{\phi_{q}^{-1}}\,\gbar(f^1)(\phi_0)\circ{\td\phi_{q}^{-1}}\dots \right. \\
 &\left. \gbar(f^q)(\phi_{q-1})\circ{\td\phi_{q-2}^{-1}\cdots \td\phi_0^{-1}\td\phi_{q}^{-1}}\,\gbar(f^0)(\phi_q)\right]\varpi=\\\notag
&\chi(\one\ot\td u\ot\td f)( g_0\circ \td\phi_{q}^{-1}\odots  g_p\circ \td\phi_{q}^{-1} \mid U_{\phi_q}\ot U_{\phi_0}\odots U_{\phi_{q-1}})=\\
&\hta\chi(\one\ot\td u\ot\td f)(g_0\odots g_p\mid U_{\phi_0}\odots U_{\phi_q}).
\end{align*}
\end{proof}
So $\chi$ induces a cyclic  map  in  the level of diagonal complexes and a map of mixed complexes
at the level of  total complexes, as follows:
\begin{align}
&\chi_{\rm diag}: C_\Hc^{n,n}(\Uc,\Fc,\Cb_\d)\ra \Hom(C_c^\infty (G)^{\ot n}\ot \Cb\G^{\ot n},\Cb), \\
&\chi_{\rm tot} :\bigoplus_{p+q=n}C_\Hc^{p,q}(\Uc,\Fc,\Cb_\d)\ra
\bigoplus_{p+q=n} \Hom(C_c^\infty (G)^{\ot p}\ot \Cb\G^{\ot q},\Cb),
\end{align}

By   \cite[Theorem 2.15]{mr09}
$\Hc^{\rm cop}\cong \Fc\acl \Uc$. Let us recall the latter isomorphism.
One identifies $\Fc^{\rm cop}$ with the commutative part of  $\Hc$  which is denoted by $\Hc_{\rm ab}$, in a natural way namely sending the standard generators of $\Fc$ to those of $\Hc_{\rm ab}$. Let us call this isomorphism by $\i$ and refer the reader to its definition in \cite[Proposition 2.7]{mr09}.
Then the isomorphism $\Jc: \Hc^{\rm cop}\ra \Fc\acl \Uc$ is defined by on the basis elements by
\begin{align*}
\Jc(\d_KZ_I)=\i(\d_K)\acl Z_I  ,
\end{align*}
where  $\{\d_K\}$ is a basis of $\Hc_{\rm ab}$   and $\{Z_I \}$ is the usual PBW basis of $\Uc$.
 It is proved in {\em loc.~cit.}
 that $\Hc$ acts on $\Ac$ from the  left and makes $\Ac$ a module algebra.
 This action induces an $\Hc^{\rm cop}$-module algebra  on $\Ac^o$ via
 \begin{equation}
 \Hc^{\rm cop}\ot \Ac^{o}\ra  \Ac^{o}, \quad h^o(\bar{a})=\overline{h(a)}
 \end{equation}
One observes that $\bar~:\Ac_{c, \G}(G) \ra \Ac_{c, \G}(G)^o$ defined by
$\overline{gU^\ast_\phi}= U_\phi g$ is an anti algebra isomorphism between
$\Ac_{c, \G}(G)$ and $\Ac_{c, \G}(G)^{o}$.

We next recall from \cite[Proposition 3.1]{mr09} the  isomorphism of $(b,B)$ Hopf cyclic complexes
\begin{align}
&\Ic_\Hc: C^n(\Hc^{\rm cop},\Cb_\d)\ra  C^n(\Hc,\Cb_\d),\\\notag
&\Ic_\Hc(\one\ot {h_1}\odots {h_n})= \one \ot h_n\ot h_{n-1}\odots h_1 ,
\end{align}
as well as its algebra counterpart
\begin{align}
 &\Ic_\G: C^n(\Ac_{\rm c, \G}^{o})\ra  C^n(\Ac_{\rm c,\G}),\\\notag
&\Ic_\G(\t)(a_0\odots a_n)= \t(a_0\ot a_n\odots a_1).
\end{align}
Finally, let us recall that the original characteristic map
$\, \chi_{\act}:\Hc(\Gb)^{\ot n}\ra \Hom(\Ac_{c, \G}(G)^{\ot n+1},\Cb)$,
which actually guided the definition
of Hopf cyclic cohomology in~\cite{cm2}, has the expression
\begin{align}\label{char-Hopf}
 \chi_{\act}(h^1\odots h^n)(a_0\odots,a_n)=\tau(a_0h^1(a_1)\dots h^n(a_n)).
\end{align}
Let us prove that all the above characteristic maps are equivalent.

\begin{theorem} \label{thm:actmap}
The diagram below is commutative and all its vertical arrows are quasi-isomorphisms.
\begin{equation}\label{chidiagram}
\xymatrix{
\ar[d]^{\Ic^{-1}_\Hc}C^n(\Hc,\Cb_\d)\ar[rr]^ {\chi_{\act}}&&
C^n(\Ac_{c, \G}(G))\ar[d]^{\Ic^{-1}_{\G}}\\
\ar[d]_{\Psi^{-1}_{\acl}}C^n(\Fc\acl\Uc, \Cb_\d)\ar[rr]^ {\chi_{\act}}&&
 C^n(\Ac_{c, \G}(G)^{\rm o})\ar[d]^{\Psi^{-1}_{\G}}\\
\ar[d]^{\overline{SH}} C^{n,n}_\Hc(\Uc,\Fc,\Cb_\d)\ar[rr]^{\chi_{\rm diag}}&&
 \ar[d]^{\overline{SH}} C^{n,n}(C_c^{\infty}(G),\G)\\
\underset{p+q=n}{\bigoplus}C^{p,q}_\Hc(\Uc,\Fc,\Cb_\d)\ar[rr]^{\chi_{\rm tot}}&&
 \underset{p+q=n}{\bigoplus}C^{p,q}(C_c^{\infty}(G),\G).}
\end{equation}
\end{theorem}

\begin{proof}
The commutativity of the  lower square  is a direct corollary of
Proposition \ref{char-bicyclic-map} as the shuffle map  $\overline{SH}$ is made of cofaces, codegeneracies and cyclic operators.
The commutativity of the top square is obvious due to the similarity of $\Ic_{\Ac}$ and $\Ic_{\Hc}$.

Let us check that the middle  square is commutative too.
Using the abbreviated  notation
$$
\td f\ot\td u = \one \ot_{\Hc^{\rm cop}} 1\ot u^1\odots u^n\ot 1\ot f^1\odots f^n ,
$$
after some obvious simplification one observes that
\begin{align*}
&\Psi_{\acl}(\td f\ot \td u)=\one \ot f^1\acl u^1\ns{0}\ot f^2  u^1\ns{1}\acl u^2\ns{0}\ot\dots\\
&~~~~~~~~~~~~~~~~~~~~~~~~~~~~~~~~~~~~~~~~~~~~\dots\ot f^n u^1\ns{n-1} \dots u^{n-1}\ns{1}\acl u^n.
\end{align*}
Assuming  that $\phi_0\dots \phi_n=\id$, one has
\begin{align}\notag
 &\chi^{\rm act}\circ\Psi_{\acl}(\td f\ot \td u)(U_{\phi_0}g_0\odots U_{\phi_n}g_n)=\\\notag
 & \tau\left( u^n(g_n)\dbar(f^n u^1\ns{n-1} \dots u^{n-1}\ns{1})(U^\ast_{\phi_n})\dots
 u^1\ns{0}(g_1)\dbar( f^1)U^\ast_{\phi_1}g_0U^\ast_{\phi_0}  \right)=\\\notag
 &\displaystyle\int_G  \left[ u^n(g_n)\gbar(S(f^n))(\phi_n^{-1}) \gbar(S(u^1\ns{n-1} \dots
 u^{n-1}\ns{1}))(\phi_n^{-1})\dots \right.\\\label{square2-a}
 &~~~~~~~~~~~~~~~~~~~~~~~~~~~~\left.  \dots (u^1\ns{0}(g_1)\gbar( S(f^1))(\phi_1^{-1}))\circ {\td\phi_2\cdots\td\phi_n} g_0\circ {\td\phi_1\cdots\td\phi_n}\right]\varpi.
   \end{align}
On the other hand,
 \begin{align}\notag
 &(\Psi_{\G}\circ\chi^\Dc)(\td f\ot \td u)(U_{\phi_0}g_0\odots U_{\phi_n}g_n)=\\\notag
& \chi^\Dc(\td f\ot \td u)(  g_0\circ{\td\phi_1\cdots \td\phi_n}\odots g_{n-1}\circ{\td\phi_n}
\ot g_n \mid U_{\phi_0}\odots U_{\phi_n})=\\\notag
&\tau( g_0\circ{\td\phi_1\cdots \td \phi_n}u^1(g_1\circ\td\phi_2\cdots \td\phi_n)\dots \\\notag
&\dots u^{n-1}(g_{n-1}\circ{\td\phi_n}) u^n(g_n) \dbar(f^n)(U^\ast_{\phi_n})\dots \dbar(f^1)(U^\ast_{\phi_1})U^\ast_{\phi_0})=\\\notag
&=\int_G\left[ g_0\circ{\td\phi_1\cdots \td \phi_n}u^1(g_1\circ{\td\phi_2\cdots\td\phi_n})\dots \right.\\\label{square2-b}
&\left.\dots u^n(g_n)\gbar(S(f^n))(\phi_n^{-1})\dots \gbar(S(f^1))(\phi_1^{-1})\circ {\td\phi_2\dots \td\phi_n}\right]\varpi_\Gb.
  \end{align}
 Using \eqref{g-u}  sufficiently  many times and recalling that $\phi_0\dots\phi_n=\id$,
 and that $\varpi$ is $\Gb$-invariant,
one obtains the equality of \eqref{square2-a} and \eqref{square2-b}.
\end{proof}

\subsection{Differential graded characteristic map} \label{Ss: dgchar}
%%%%%%%%%%%%%%%%%%%%%
We proceed now to construct  characteristic maps from the
Hopf version
of the Chevalley-Eilenberg complexes to the cyclic cohomology
bicomplex of the crossed product by differential forms
$\Om_{c, \G}(G) := \Om_c^\bullet (G) \rtimes \Gamma$.

To begin with, let us recall the definition of the
 Bott bicomplex  $C_{\rm Bott}^{\bullet,\bullet}(G, \Gb)$ computing the $\Gb$-equivariant
 cohomology of $G$.
The cochain space $C_{\rm Bott}^{q, p}(G, \Gb)$
$ =C_{\rm Bott}^{q}(\Om_p (G), \Gb)$
consists of all antisymmetrized $\Gb$-equivariant  functions on $\Gb^{\times (p+1)}$
with values in de Rham currents of degree $q$. As in \eqref{eq:vanestinv}, the
equivariance condition is
\begin{equation}\label{equivariant}
\g(\phi_0  \phi,\dots \phi_q \phi)=  \, ^tL^\ast_{\phi^{-1} \rt} \,
\g(\phi_0,\dots \phi_q), \quad \g\in C^{p,q}, \phi_i \in \G, \phi\in \Gb.
\end{equation}
When the current is given by a form, this condition becomes
\begin{equation}\label{equivariantform}
\g(\phi_0 \phi,\dots \phi_q \phi)=  \, L^\ast_{\phi \rt} \,
\g(\phi_0,\dots \phi_q) \, \equiv \, {\td \phi}^\ast \big(\g(\phi_0,\dots \phi_q)\big) ,
\end{equation}
by duality and taking into account that $\Gb$ acts on $G$ via the natural action on frames.
The horizontal and vertical coboundaries of  $C_{\rm Bott}^{\bullet,\bullet}(G, \Gb)$
are defined as follows:
\begin{align}\label{Bott}
&d_1(\g)(\phi_0,\dots,\phi_{q+1})=\sum_{i=0}^{q+1}(-1)^i\g(\phi_0,\dots, \check{\phi_i},\dots, \phi_{q+1}),\\
&d_2(\g)(\phi_0,\dots,\phi_q)=d^t(\g(\phi_0,\dots,\phi_q)),
\end{align}
We now define the map $\, \Theta: \left( \wedge^p \Fg^\ast\ot \wdg^{q+1}\Fc \right)^\Fc \ra
C_{\rm Bott}^{p,q}(G, \Gb)$ by the formula
\begin{align}\label{Theta-formula}
\begin{split}
&\Theta (\sum_I \a_I\ot \lu{I}f^0 \wdg \cdots \wdg \lu{I}f^q )(\phi_0,\dots,\phi_q)
\, = \, \\
&= \sum_{I, \s} (-1)^\s   \gbar(S( \lu{I}f^{\s(0)}))(\phi_0^{-1})\dots \gbar(S(\lu{I}f^{\s(q)}))(\phi_q^{-1}){{\td\a_I}} \,.
\end{split}
\end{align}
 Here ${{\td\a_I}}$ is the left invariant form on $G$ associated to $\a_I \in \wdg^q \Fg^\ast$.
 \medskip

Recall that the bicomplex $(\wdg^p\Fg^\ast\ot \wdg ^{q+1}\Fc)^\Fc$ has two cobundary $b_{\rm c-w}$ and $\p_{\rm c-w}$  defined in \eqref{b-F-wedge} and \eqref{p-F-wedge}  respectively.

\begin{proposition}
The map $\Theta$
 is a well-defined homomorphism of bicomplexes.
\end{proposition}
\begin{proof}
In order to prove that $\Theta$ is well-defined, it suffices to check that
$\Theta(\a\ot f^0\wdots f^q)$ is equivariant.

In view of the identity \eqref{coact1}, one has
 \begin{align} \label{coactgf}
\tilde{\phi}^\ast \tilde{\a}\,=\, \gbar(\a\ns{1})(\phi^{-1}) \, \tilde{\a}\ns{0} .
\end{align}
Also from Lemma \ref{beta},
 \begin{align*}
&\gbar(f)(\phi_1\phi_2)=\gbar(f\ps{1})(\phi_1)\gbar(f\ps{2})(\phi_2)
\circ{{\td\phi_1^{-1}}}.
\end{align*}
Using the fact that $\sum_I\a_I\ot \lu{I}f^0\wdots \lu{I}f^q$ belongs to the coinvariant space,
we can check the
condition  \eqref{equivariantform} for $\phi\in \G$, as follows:
 \begin{align*}
&  \sum_I\gbar(S(\lu{I}f^0))((\phi_0\phi)^{-1})\cdots \gbar(S(\lu{I}f^q))((\phi_n\phi)^{-1})\, {\td\a_I}\, =\\
&  \sum_I \gbar(S(\lu{I}f^0\ps{2}))(\phi^{-1})\;\gbar(S(\lu{I}f^0\ps{1}))(\phi_0^{-1})
\circ{\td\phi}\cdots   \\
&~~~~~~~~~~~~~~~~~~~~~~~~~\cdots\gbar(S(\lu{I}f^q\ps{2}))(\phi^{-1})\;\gbar(S(\lu{I}f^q\ps{1}))(\phi_q^{-1})
\circ{\td\phi}\,{ \td\a_I} \, =\\
& \sum_I\left(\gbar(S(\lu{I}f^0\ps{1}))(\phi_0^{-1})\dots\gbar(S(\lu{I}f^q\ps{1}))(\phi_q^{-1})\right)
\circ{\td\phi}\;
 \gbar(S(\lu{I}f^0\ps{2}\cdots \lu{I}f^q\ps{2}))(\phi^{-1})\,
  {\td\a_I}  =\\
&  \sum_I \tilde{\phi}^\ast\left(\gbar(S(\lu{I}f^0\ps{1}))(\phi_0^{-1})\dots\gbar(S(\lu{I}f^q\ps{1}))(\phi_q^{-1})\right)\;    \gbar(S(\lu{I}f^0\ps{2}\cdots \lu{I}f^q\ps{2}))(\phi^{-1})\,{\td\a_I} =\\
\end{align*}
 \begin{align*}
&=  \sum_I \tilde{\phi}^\ast\left(\gbar(S(\lu{I}f^0))(\phi_0^{-1})\dots\gbar(S(\lu{I}f^q))(\phi_q^{-1})\right)\;\gbar(\a_I\ns{1}) (\phi^{-1}) \,{\td\a_I\ns{0}} \, =\\
& \sum_I \tilde{\phi}^\ast\left(\gbar(S(\lu{I}f^0))(\phi_0^{-1})\dots\gbar(S(\lu{I}f^q))(\phi_q^{-1})\right)\;\tilde{\phi}^\ast
({\td\a_I})\,=\\
& \sum_I\tilde{\phi}^\ast\left(\gbar(S(\lu{I}f^0))(\phi_0^{-1})\dots\gbar(S(\lu{I}f^q))(\phi_q^{-1})\; { \td\a_I}\right).
\end{align*}

Next we prove that $\Theta b_\Fc=d_1\Theta$, we set $\lu{I}g^0=1$, $\lu{I}g^j=\lu{I}f^{j-1}$, for $1\le j\le q+1$.

\begin{align*}
&\Theta b_{\rm c-w}(\sum_I\a_I\ot \lu{I}f^0\wdots \lu{I}f^q)(\phi_0,\dots,\phi_{q+1})=\\
&\Theta(\sum_I\a_I\ot 1\wdg \lu{I}f^0\wdots \lu{I}f^q)(\phi_0,\dots,\phi_{q+1})=\\
&\sum_{I,\s\in S_{q+2} }(-1)^\s   \gbar( S(\lu{I}g^{\s(0)}))(\phi_0^{-1})\dots \gbar(S(\lu{I}g^{\s(q)}))(\phi_{q+1}^{-1}) {\td\a_I}  =\\
&\sum_{I,\s\in S_{q+2}} (-1)^\s  \gbar(S( \lu{I}g^0))(\phi_{\s(0)}^{-1})\dots \gbar(S(\lu{I}g^{q+1}))(\phi_{\s(q+1)}^{-1}){\td\a_I} =\\
&\sum_{I,\s\in S_{q+2}} \sum_{\s(0)=j}(-1)^\s   \gbar( S(\lu{I}g^0))(\phi_{\s(0)}^{-1})\dots \gbar(S(\lu{I}g^{q+1}))(\phi_{\s(q+1)}^{-1}){\td\a_I}=\\
&\sum_{I,0\le j\le q+1}(-1)^j \sum_{\pi\in S_{q+1}} (-1)^\pi  \gbar(S(\lu{I}f^{\pi(0)}))(\phi_0^{-1})\dots\\
&\gbar(S(\lu{I}f^{\pi(j-1)}))(\phi_{j-1}^{-1}) \gbar(S(\lu{I}f^{\pi(j)}))(\phi_{j+1}^{-1})\dots  \gbar(S(\lu{I}f^{\pi(q)}))(\phi_{q+1}^{-1}){ \td\a_I}=\\
& \sum_{I, 0\le j\le q+1}(-1)^j \Theta(\a_I\ot \lu{I}f^0\wdots \lu{I}f^q)(\phi_0,\dots,\check{\phi_j},\dots, \phi_{q+1}),.
\end{align*}

 To prove that $\Theta\p_{\rm c-w}=d_2\Theta$ we fix a basis $X_i$ for $\Fg$ and its dual basis $\t^i$ for $\Fg^\ast$. Let   $g\in C^{\infty}(G)$ then its differential is   $dg=\sum X_i(g)\t^i$. On the other hand we know  $\dbar$ is a $U(\Fg)$ module map which means that  $X(\dbar(f)(U^\ast_\phi))=\dbar(X\rt f)(U^\ast_\phi)$ this implies that $\gbar\circ S$ is $U(\Fg)$-module map. Finally  see that
  \begin{align*}
   &d(\sum_I \lu{I}g\;{\td\a_I})= \sum_I d(\lu{I}{g})\;{\td\a_I} +\sum_I \lu{I}g\;d({\td\a_I})=\\
&\sum_{I,i} X_i(\lu{I}g)\t^i\wedge {\td\a_I} + \sum_I \lu{I}g\;d({\td\a_I}).
   \end{align*}
  Here   $\lu{I}g=\sum_\s (-1)^\s  \gbar(S(\lu{I} f^{\s(0)}))(\phi_0^{-1})\dots \gbar(S(\lu{I}f^{\s(q)}))(\phi_q^{-1})$.
\end{proof}
\bigskip

 \begin{remark} \label{rem:novp} The map $\, \Theta$
 does not distinguish between the $\Gb$-equivariant and the $\FN$-equivariant
 cohomology of $G$.
  \end{remark}
Indeed, let $\vp_0,\dots,\vp_q \in G$ and
 $\psi_0,\dots,\psi_q \in \FN$.
Using the cocycle property \eqref{cocycle-condition} and the
fact that $\, \gbar(f) (\vp) = \ve (f)$, $\, \fl \, \vp \in G$, one can write
\begin{align*}
\begin{split}
\gbar(S(f))(\psi^{-1}\vp^{-1})& \, =\, \gbar(S(f)\ps{1})(\psi^{-1})
\gbar(S(f)\ps{2})(\vp^{-1})\circ{{\td\psi}}\, =\\
&= \, \gbar(S(f))(\psi^{-1}) .
\end{split}
\end{align*}
It follows that
 \begin{align}
\begin{split}
&\Theta (\sum_I\a_I\ot \lu{I}f_0 \wdg \cdots \wdg \lu{I}f_q )(\vp_0\psi_0,\dots, \vp_q\psi_q) \, = \, \\
&\qquad \qquad \Theta (\sum_I\a_I\ot \lu{I}f_0 \wdg \cdots \wdg \lu{I}f_q )(\psi_0,\dots,\psi_q) .
\end{split}
\end{align}

 \bigskip

In~\cite{Sasha02} Gorokhovsky introduced
the chain map $\Psi: C_{\rm Bott}^{\bullet,\ast}(G, \Gb)\ra C^\ast(\Om_{c,\G}(G))$,
which (adapted to a left action) is given by the formula
\begin{align} \label{Sasha-map-formula}
&\Psi (\g)( \om_0U^\ast_{\phi_0}, \om_1U^\ast_{\phi_1},\dots, \om_qU^\ast_{\phi_q})=\\\notag
&(-1)^{q+\deg{\g}}
\langle \g(1,\phi_0,\phi_1\phi_0,\dots,\phi_{q-1}\dots \phi_0),
\om_0\;{\phi_0}^\ast\om_1\,\dots {(\phi_{q-1}\dots \phi_{0})}^\ast\om_q\rangle,
\end{align}
if $\phi_q\dots \phi_0=1$ and $0$ otherwise.

Since $\Psi$ is a map of complexes.
 the composition
\begin{equation}\label{chi-wdg-def}
\quad \chi^{\wdg} =\Psi \circ \Theta: \wg^\ast\Fg^\ast\ot \wdg^{q+1}\Fc\ra C^{q}(\Om_{c,\G}(G))
\end{equation}
is also a chain map. It is given by the formula
\begin{align}\label{chi-wdg-formula-1}
\begin{split}
& \chi^{\wdg} (\a \ot \td f )( \om_0U^\ast_{\phi_0}, \om_1U^\ast_{\phi_1},\dots, \om_qU^\ast_{\phi_q})= \\
& (-1)^{\deg\a+q} \cdot \\
  & \sum_{\s\in S_{q+1}} (-1)^\s \int_G
 \ve(f^{\s(0)}) \gbar(S(f^{\s(1))})(\phi_0^{-1})\dots \gbar(S(f^{\s(q)}))(\phi_{0}^{-1}\dots \phi_{q-1}^{-1})\\
 &\qquad \qquad \cdot
 {\td \a}\, \om_0\;{\phi_0}^\ast\om_1\,\dots {(\phi_{q-1}\dots \phi_{0})}^\ast\om_q .
 \end{split}
\end{align}
if $\phi_q\dots \phi_0=1$ and $0$ otherwise.

\bigskip
%%%%%%%%%%%%%%%%%%%%%%

We now define the map
$\Upsilon:C^{\bullet,\bullet}_{\rm Bott}(\G, G)\ra C^{\bullet, \bullet}(C^{\infty}_c(G),\G)$ as follows
\begin{align}\label{Upsilon}
 &\Upsilon(\g)(f^0, \ldots, f^p\mid U_{\phi_0},\ldots,U_{\ph_q})=\\\notag
&\left\{\begin{matrix}\displaystyle \int_G f^0df^1\dots df^p\g(\phi_0^{-1}, \phi_1^{-1}\phi_{0}^{-1}, \dots \phi_{q-1}^{-1}\cdots \phi_{0}^{-1}, 1),&\;\;\text{if}\;\; \phi_0\cdots\phi_q=id,\\
&\\
0, & \text{otherwise}.\end{matrix}\right.
\end{align}

\begin{proposition}
The map $\Upsilon$ is a map of bicomplexes, more precisely
\begin{equation}
\Upsilon\circ d_2= \vB\circ\Upsilon,\quad \Upsilon\circ d_1= \hb\circ\Upsilon, \quad \hB\circ\Upsilon=\vb\circ \Upsilon=0.
\end{equation}
\end{proposition}

\begin{proof} Without loss of generality we can assume that
$\phi_0\dots \phi_q=\id$.
We first use that fact that $\g \in C^{q}(\G, \Om_p(G))$ is equivariant to see that
\begin{align*}
\begin{split}
&\hd_{q+1}(\Upsilon(\g))(f^0,\dots f^p\mid U_{\phi_0}, \dots, U_{\phi_{q+1}})=\\
& \int_G {\phi_{q+1}^{-1}}^\ast(f^0df^1\dots df^p) \g(\phi_0^{-1}\phi_{q+1}^{-1}, \phi_1^{-1}\phi_{0}^{-1}\phi_{q+1}^{-1}, \dots,  \phi_{q-1}^{-1}\cdots \phi_{0}^{-1}\phi_{q+1}^{-1},1)=\\
&\int_G {\phi_{q+1}^{-1}}^\ast(f^0df^1\dots df^p) {\phi_{q+1}^{-1}}^\ast\g(\phi_0^{-1}, \phi_1^{-1}\phi_{0}^{-1}, \dots,\phi_{q}^{-1}\cdots \phi_{0}^{-1})=\\
&\int_G f^0df^1\dots df^p \g(\phi_0^{-1}, \phi_1^{-1}\phi_{0}^{-1}, \dots,\phi_{q}^{-1}\cdots \phi_{0}^{-1}).
\end{split}
\end{align*}
We use this to check that  $\Upsilon\circ d_1= \hb\circ\Upsilon$, as follows:
\begin{align*}
&\Upsilon( d_1\g)(f^0, \ldots, f^p\mid U_{ \phi_0},\ldots,U_{\phi_{q+1}})= \\
&\sum_{i=0}^{q} (-1)^i\int_G f^0df^1\dots df^p \g(\phi_0^{-1}, \phi_1^{-1}\phi_{0}^{-1}, \dots,\widehat{\phi_{i-1}^{-1}\cdots \phi_{0}^{-1}},\dots, \phi_{q}^{-1}\cdots \phi_{0}^{-1}, 1) \\
&~~~~~~~~~~~~~~+
 (-1)^{q+1}\int_G f^0df^1\dots df^p \g(\phi_0^{-1}, \phi_1^{-1}\phi_{0}^{-1}, \dots,  \phi_{q}^{-1}\cdots \phi_{0}^{-1})= \\
&\sum_{i=0}^{q} (-1)^{i}\hd_{i}(\Upsilon(\g))(f^0,\dots f^p\mid U_{\phi_0}, \dots, U_{\phi_{q+1}})+\\
&(-1)^{q+1}\hd_{q+1}(\Upsilon(\g))(f^0,\dots f^p\mid U_{\phi_0}, \dots, U_{\phi_{q+1}})=\,
\hb(\Upsilon(\g))(f^0, \ldots,f^p\mid U_{ \phi_0},\ldots,U_{\phi_{q+1}}).
\end{align*}

Next, we show that $\Upsilon\circ d_2= \vB\circ\Upsilon$. Using the Stokes formula we see that

\begin{align*}
\begin{split}
&\Upsilon(d_2\g)(f^0,\dots, f^{p-1}\mid U_{\phi_0}, \dots, U_{\phi_q})=\\
&\int_G f^0df^1\dots df^{p-1} (d_2\g)(\phi_0^{-1}, \phi_1^{-1}\phi_{0}^{-1}, \dots, \phi_{q-1}^{-1}\cdots \phi_{0}^{-1}, 1)=\\
&\int_G f^0df^1\dots df^{p-1} d\g(\phi_0^{-1}, \phi_1^{-1}\phi_{0}^{-1}, \dots, \phi_{q-1}^{-1}\cdots \phi_{0}^{-1}, 1)=\\
&(-1)^{p-1}\int_G d(f^0df^1\dots df^{p-1} )\g(\phi_0^{-1}, \phi_1^{-1}\phi_{0}^{-1}, \dots, \phi_{q-1}^{-1}\cdots \phi_{0}^{-1}, 1)=\\
& (\vB(\Upsilon(\g))(f^0,\dots, f^{p-1}\mid U_{\phi_0}, \dots, U_{\phi_q}).
\end{split}
\end{align*}
To show that $\hB\circ\Upsilon=0$ we check that
 the image of $\Upsilon$ is horizontally cyclic.
\begin{align*}
\begin{split}
&\hta(\Upsilon(\g))(f^0, \dots, f^p\mid U_{\phi_0}, \dots, U_{\phi_q})=\\
&(\Upsilon(\g))({\phi_q^{-1}}^\ast(f^0), \dots, {\phi_q^{-1}}^\ast(f^p)\mid U_{\phi_q},U_{\phi_0},  \dots, U_{\phi_{q-1}})=\\
&\int_G {\phi_q^{-1}}^\ast(f^0df^1\dots df^p) \g(\phi_q^{-1}, \phi_{0}^{-1}\phi_q^{-1}, \dots \phi_{q-2}^{-1}\cdots \phi_{0}^{-1}\phi_q^{-1}, 1)=\\
&\int_G {\phi_q^{-1}}^\ast(f^0df^1\dots df^p) {\phi_q^{-1}}^\ast\g(1, \phi_{0}^{-1}, \dots \phi_{q-2}^{-1}\cdots \phi_{0}^{-1},\phi_{q-1}^{-1}\cdots \phi_{0}^{-1} )=\\
&(-1)^q\int_G (f^0df^1\dots df^p) \g( \phi_{0}^{-1}, \dots \phi_{q-2}^{-1}\cdots \phi_{0}^{-1},\phi_{q-1}^{-1}\cdots \phi_{0}^{-1},1 )=\\
&(-1)^q\Upsilon(\g)(f^0, \dots, f^p\mid  U_{\phi_0}, \dots, U_{\phi_q}).
\end{split}
\end{align*}
The identity $\vb\circ\Upsilon=0$ follows from the fact that $d$ is a derivation.
\end{proof}
\bigskip

%%%%%%%%%%%%%%%%%%%%%%%%%%%%%%%%%

We now consider the diagram
\begin{equation}\label{midsquare}
\xymatrix{
C^{\bullet,\bullet}_\Hc(\Uc,\Fc,\Cb_\d)\ar[r]^{\chi_{\rm tot}} & C^{\bullet.\bullet}(C^{\infty}_c(G),\G) \\
C^{\bullet,\bullet}_{\rm c-w}(\Fg^\ast,\Fc)\ar[r]^\Theta  \ar[u]^{\Lambda}&
C^{\bullet,\bullet}_{\rm Bott}(\Om(G), \G)\ar[u]^{\Upsilon}},
\end{equation}
where  $\Lambda:= \td\a_\Fg\circ (\FD_\Fg\ot id)\circ \Ic^{-1}\circ \td\a_\Fc$.
The antisymmetrization map
\begin{equation*}
\td\a_\Fc: C^{\bullet, \bullet}_{\rm c-w}(\Fg^\ast, \Fc)\ra  C^{\bullet, \bullet}_{\rm coinv}(\Fg^\ast, \Fc),
 \end{equation*}
 is defined in \eqref{anti-sym}, the map
 \begin{equation*}
 \Ic^{-1}: C^{\bullet, \bullet}_{\rm coinv}(\Fg^\ast, \Fc)\ra C^{\bullet, \bullet}(\Fg^\ast, \Fc),
 \end{equation*}
 is given by \eqref{I-1-inverse}, and
the Poincar\'e isomorphism
\begin{equation*}
\FD_\Fg:C^{\bullet, \bullet}(\Fg^\ast, \Fc)\ra C^{\bullet, \bullet}(\Fg,\Cb_\d \ot \Fc),
 \end{equation*}
by \eqref{theta}; finally, the  map
\begin{equation*}
\td\a_\Fg: C^{\bullet,\bullet}(\Fg,\Fc, \Cb)\ra  C^{\bullet,\bullet}(\Uc,\Fc,\Cb_\d),
 \end{equation*}
 is the first antisymmetrization map, defined in \eqref{antsym1}.

\begin{proposition}
The diagram \eqref{midsquare} is commutative.
\end{proposition}
\begin{proof}
It suffices to work with
 elements of the form $\om\ot 1\wdg f^1\wdots f^q\in  C^{\bullet, \bullet}_{\rm coinv}(\Fg^\ast, \Fc)$. Indeed we see that
\begin{align*}
\begin{split}
&\Ic^{-1}(\td\a_\Fc(\om\ot 1\wdg f^1\wdots f^q)=\\
&\sum_{\s\in S_q} (-1)^\s \om\ot f^{\s(1)}\ps{1}\ot f^{\s(1)}\ps{1}f^{\s(2)}\ps{1}\odots f^{\s(1)}\ps{q}\dots f^{\s(q-1)}\ps{2}f^{\s(q)}.
\end{split}
\end{align*}
 After applying $\td\a_\Fg\circ  \FD_\Fg$ on $\Ic^{-1}(\td\a_\Fc(\om\ot 1\wdg f^1\wdots f^q)$  the resulted
  cochain belongs to $\Cb_\d  \ot \Uc^{\ot p}\ot \Fc^{\ot q}$. We  then  identify it
  to the corresponding cochain
  in $\Cb_\d\ot_{\Hc}\Uc^{\ot p+1}\ot \Fc^{\ot q+1}$. The result is seen as follows:
 \begin{align}
 \begin{split}
 &\Lambda(\om\ot 1\wdg f^1\wdots f^q))=\\
  &\sum_{\s\in S_q} (-1)^\s 1\ot\a_\Fg(\FD_\Fg(\om))\ot 1\ot f^{\s(1)}\ps{1}\ot f^{\s(1)}\ps{1}f^{\s(2)}\ps{1}\ot\dots\\
  &\dots\ot f^{\s(1)}\ps{q}\dots f^{\s(q-1)}\ps{2}f^{\s(q)}
 \end{split}
 \end{align}
 Transferring the image of $\Lambda$ via $\chi$
 we observe that  for $\phi_0, \dots\phi_q \in \G$ with $\phi_0, \dots\phi_q=\id$ we have
 \begin{align} \label{chi-Lambda-image}
 \begin{split}
 &\chi(\Lambda(\td\a_\Fc(\om\ot 1\wdg f^1\wdots f^q))( g_0, \dots,  g_p,\phi_0, \dots, \phi_q)=\\
 &\displaystyle \sum_{\s\in S_q} (-1)^\s \int_G (1\ot\a_\Fg\FD_\Fg(\om))(g_0, \dots, g_p) \gbar(f^{\s(1)}\ps{1})(\phi_1)\circ{\td\phi_0^{-1}}  \dots\\
 &\dots
\gbar(f^{\s(1)}\ps{q}\dots f^{\s(q-1)}\ps{2}f^{\s(q)})(\phi_q)\circ{\td\phi_{q-1}^{-1}\dots \td\phi_{0}^{-1}}\varpi_G
 \end{split}
 \end{align}

Let $X_1, \dots, X_m$ be a basis   for $\Fg$ and $\t^1, \dots, \t^m$ be its dual basis.
 Without loss of generality, we can assume
  $\om= \t^{p+1}\dots \t^m$. We  compute the term

\begin{align}  \label{Poincare-applicatio-proof}
\begin{split}
&g_0dg_1\dots dg_p\om= g_0X_{i_1}(g_1)\t^{i_1}\dots X_{i_p}(g_p)\t^{i_1}\om=\\
&\sum_{1\le i_1, \dots, i_p\le m}g_0 X_{i_1}(g_1)\dots X_{i_p}(g_p)\t^{i_1}\dots \t^{i_1}\om=\\
&\sum_{\mu\in S_p}(-1)^\mu g_0 X_{\mu(1)}(g)\dots X_{\mu(p)}(g_p)\varpi_G=(1\ot a_\Fg\circ\FD_\Fg(\om))(g_0, \dots, g_p)\varpi.
\end{split}
\end{align}
We next compute
\begin{align}\label{Upsilon-Theta-immage}
\begin{split}
&\Upsilon\circ\Theta  (\om \ot 1\wdg f^1\wdots f^q)(g_0,\dots, g_p, \phi_0, \dots, \phi_q)=\\
&\displaystyle \int_G g^0dg^1\dots dg^p \Theta(\om \ot 1\wdg f^1\wdots f^q)(\phi_0^{-1},\phi_1^{-1}\phi_0^{-1}, \dots,\phi_{q-1}^{-1}\dots\phi_0^{-1}, 1)=\\
&\sum_{\s\in S_p}(-1)^s\displaystyle \int_G g^0dg^1\dots dg^p\;\om \gbar(S(f^{\s(1)}))(\phi_0)\gbar(S(f^{\s(2)}))(\phi_0\phi_1) \dots\gbar(S(f^q))(\phi_0\dots\phi_{q-1})=\\
&\sum_{\s\in S_p}(-1)^s\displaystyle \int_G (1\ot a_\Fg\circ\FD_\Fg(\om))(g_0, \dots, g_p) \gbar(S(f^{\s(1)}))(\phi_0)\gbar(S(f^{\s(2)}))(\phi_0\phi_1)\dots \\ &\dots\gbar(S(f^{\s(q)}))(\phi_0\dots\phi_{q-1})\varpi.
\end{split}
\end{align}
 Using \eqref{cocycle-condition}, \eqref{S(f)(psi)},
 the assumption $\phi_0\dots\phi_q=\id$, and the fact that $\gbar$ is an algebra map,  one sees that
 \begin{align*}
 \begin{split}
 &\gbar(S(f^{\s(1)}))(\phi_0)\gbar(S(f^{\s(2)}))(\phi_0\phi_1)
 \dots\gbar(S(f^{\s(q)}))(\phi_0\dots\phi_{q-1})=\\
 &\gbar(S(f^{\s(1)}))((\phi_1\dots \phi_q)^{-1})\gbar(S(f^{\s(2)}))((\phi_2\dots \phi_q)^{-1})
 \dots\gbar(S(f^{\s(q)}))(\phi_{q}^{-1})=\\
 &\gbar(f^{\s(1)})(\phi_1\dots \phi_q)\circ {\td\phi_0^{-1}}
 \gbar(f^{\s(2)})(\phi_2\dots \phi_q)\circ{\td\phi_1^{-1}\phi_0^{-1}}
 \dots\\
 &~~~~~~~~~~~~~~~~~~~~~~~~~~~~~~~~~~~~~~~~\dots \gbar(f^{\s(q)})(\phi_{q})\circ
 {\td\phi_{q-1}^{-1}\dots \phi_{0}^{-1}}=\\
 &\left(\gbar(f^{\s(1)}\ps{1})(\phi_1)
 \gbar(f^{\s(1)}\ps{2})(\phi_2)\circ{\td\phi_1^{-1}}\dots \gbar(f^{\s(1)}\ps{q})(\phi_q)\circ{\td\phi_q^{-1}\dots\td\phi_{1}^{-1}}\right)\circ {\td\phi_0^{-1}}\\
 &\left(\gbar(f^{\s(2)}\ps{1})(\phi_2)
 \gbar(f^{\s(2)}\ps{2})(\phi_3)\circ{\td\phi_2^{-1}}\dots
 \gbar(f^{\s(2)}\ps{q-1})(\phi_q)\circ{\td\phi_q^{-1}\dots\phi_2^{-1}}\right)\circ {\td\phi_1^{-1}\td\phi_0^{-1}}\dots\\
 & \vdots\\
 &\left(\gbar(f^{\s(q-1)}\ps{1})(\phi_{q-1})
 \gbar(f^{\s(q-1)}\ps{2})(\phi_q)\circ{\td\phi_{q}^{-1}\phi_{q-1}^{-1}}\right)
 \circ{\td\phi_{q-2}^{-1}\dots \td\phi_{0}^{-1}}\\
 &  \gbar(f^{\s(q)})(\phi_{q})\circ{\td\phi_{q-1}\dots \td\phi_{0}^{-1}}= \\
 &\gbar(f^{\s(1)}\ps{1})(\phi_1)\circ{\td\phi_0^{-1}} \dots \gbar(f^{\s(2)}\ps{q}\dots f^{\s(q-1)}\ps{2}f^{\s(q)})(\phi_q)\circ{\td\phi_{q-1}^{-1}\dots \td\phi_{0}^{-1}}.
  \end{split}
 \end{align*}
 In conjunction with
 \eqref{chi-Lambda-image}, \eqref{Poincare-applicatio-proof}, and \eqref{Upsilon-Theta-immage}
 this implies $\chi\circ\Lambda=\Upsilon\circ\Theta$.
 \end{proof}

%%%%%%%%%%%%%%%%%%%%%%%%%%%%%%
\bigskip

At this point we introduce a new type of characteristic  map, obtained by
cup product with a cocycle with coefficients (\cf~\cite{kr5}),
which lands in the differential graded algebra $\Om_{c,\G}(G)$.
We let $\Fc$  acts on $\wdg^\bullet \Fg^\ast$ trivially and coact via $\Db_{\Fg^\ast}$ defined in
\eqref{coact3}. We define $\tau\in\Hom_{U(\Fg)}(V^\bullet(\Fg^\ast)\ot \Om_{c,\G}, \Cb)= \Hom(\wdg^\bullet\Fg^\ast\ot \Om_{c,\G},\Cb)$ by
\begin{align}\label{tau}
 \tau(\a\ot \om\,U^\ast_\vp)=\left\{
                             \begin{array}{ll}
                               \displaystyle \int_G \td\a\,\om, & \text{if}\; \phi=\id; \\[.5cm]
                               0, & \text{otherwise}.
                             \end{array}
                           \right.
\end{align}

\begin{lemma}\label{invariant-trace-lemma}
The functional $\tau$ belongs to $Z^0_\Fc(\Om_{c,\G}(G);\wdg^\bullet\Fg^\ast)$, that is
for any $b,b_1,b_2\in \Om_{c,\G}(G)$, any $\a\in \wdg^\bullet\Fg^\ast$, and any $f\in\Fc$,
one has
\begin{align}\label{closed-trace}
&\displaystyle\tau(\a\ot b_2b_1)= (-1)^{\deg(b_1)\deg(b_2)}\tau(\a\ns{0}\ot \dbar(S(\a\ns{1})(b_1) b_2),\\[.2cm]\label{invariant-trace}
&\displaystyle\ve(f)\tau (\a\ot b)= \tau(\a f\ot b)= \tau(\a \ot \dbar(S(f))(b)).
\end{align}
\end{lemma}
\begin{proof}
As \eqref{invariant-trace} is obvious, we just prove  \eqref{closed-trace}.
Let $b_1=\om_1U^\ast_{\psi_1}$ and $b_2=\om_2U^\ast_{\phi_2}$, with $\phi_1\phi_2=\id$.
Then
\begin{align*}
&\tau(\a\ot b_2\,b_1) =\displaystyle \int_G\;\; \td\a\, \om_2\;\phi_2^\ast(\om_1)=\displaystyle \int_G {\phi_1}^\ast(\td \a)\,\phi_1^\ast(\om_2)\; \om_1=\\
&\displaystyle(-1)^{\deg(\om_1)\deg(\om_2)}\int_G\;\; {\td\a\ns{0}}\, \gbar(\a\ns{1})(\phi_1^{-1})\om_1 \phi_1^\ast\om_2=\\
&\displaystyle(-1)^{\deg(\om_1)\deg(\om_2)} \tau\left(\a\ns{0}\ot \dbar(S(\a\ns{1}))(b_1) b_2\right).
\end{align*}
\end{proof}

With these preparations we define
$\, \chi^\dag: \wdg^\bullet\Fg^\ast\ot \Fc^{\ot q}\ra C^q(\Om_{c,\G}(G))$
by the following formula:
\begin{align}\label{chi-dag-def-1}
&\chi^\dag(\a\ot\td f)(b_0,\ldots,b_q )=(-1)^{\deg(\a)}\tau(\a\ot b_0\dbar(S(f^1))(b_1)\dots \dbar(S(f^q))(b_q)).
\end{align}
Alternatively,
\begin{align}\label{chi-dag-def-2}
&\chi^\dag(\a\ot f^1\odots f^q)(\om_0\,U^\ast_{\phi_0},\dots, \om_q\, U^\ast_{\phi_q} )=\\\notag
&(-1)^{\deg(\a)}\displaystyle \int_G \td\a\,\om_0 \,\phi_0^\ast\om_1\dots (\phi_{q-1}\cdots\phi_{0})^\ast\om_q\, \gbar(f^1)(\phi_1^{-1})\circ \tilde{\phi_0}\cdots\\\notag
& ~~~~~~~~~~~~~~~~~~~~~~~~~~~~~~~~~~~~~~~~~~~~~~~~~~~\cdots \gbar(f^q)(\phi_q^{-1})\circ {\td\phi_{q-1}\dots \td\phi_0)} ,
\end{align}
if $\phi_q\dots\phi_0=\Id $ and $0$ otherwise.

\begin{proposition} \label{prop:dagger}
The map $\chi^\dag$ is a map of bicomplexes, satisfying
\begin{equation*}
\chi^\dag b_\Fc^\ast= b_{\Om_{c,\G}(G)}\chi^\dag, \quad \chi^\dag B_\Fc^\ast= B_{\Om_{c,\G}(G)}\chi^\dag, \quad \chi^\dag \p_{\Fg^\ast}= d_{\Om_{c,\G}(G)}\chi^\dag.
\end{equation*}
\end{proposition}
\begin{proof}
We begin by proving the last equality. Assuming $\phi_q\dots\phi_0=id$
and denoting
$g_\phi=  \gbar(f^1)(\phi_1^{-1})\circ \tilde{\phi_0}\cdots \gbar(f^q)(\phi_q^{-1})\circ {\td\phi_{q-1}\dots \td\phi_0}$, by the  Stokes formula one has
\begin{align*}
&d_{\Om_{c,\G}(G)}\chi^\dag(\a\ot f^1\odots f^q)(\om_0U^\ast_{\phi_0},\dots,\om_qU^\ast_{\phi_q})=\\
&\sum_{i=1}^q(-1)^{\deg({\a}\om_0\cdots\om_{i-1})}\displaystyle \int_G \td\a\,\om_0 \,\phi_0^\ast\om_1\dots (\phi_0\cdots\phi_{i-1})^\ast d\om_i\dots (\phi_0\cdots\phi_{i-q})^\ast \om_q\, g_\phi=\\
&-\displaystyle \int_G d(\td\a)\,\om_0 \,\phi_0^\ast\om_1\dots (\phi_0\cdots\phi_{q-1})^\ast\om_q\,  g_\phi- \\ &(-1)^{\deg(\a\om_0\cdots\om_{q})}\displaystyle \int_G \td\a\,\om_0 \,\phi_0^\ast\om_1\dots (\phi_{q-1}\cdots\phi_{0})^\ast\om_q\, dg_\phi.
\end{align*}
To compute $d(g_\phi)$, one writes
\begin{align*}
&d(g_\phi)= \sum_{i}\t^i\,X_i(g_\phi)=\\
& \sum_{i}\t^i\,X_i(\gbar(f^1)(\phi_1^{-1})\circ \tilde{\phi_0}\cdots \gbar(f^q)(\phi_q^{-1})\circ {\td\phi_{q-1}\dots \td\phi_0)}=\\
& \sum_{i,j}\t^i\,\gbar(f^1)(\phi_1^{-1})\circ \tilde{\phi_0}\cdots
X_i\left(\gbar(f^j)(\phi_j^{-1})\circ {\td\phi_{j-1}\dots \td\phi_0)}\right)\cdots\\
&\gbar(f^q)(\phi_q^{-1})\circ {\td\phi_{q-1}\dots \td\phi_0)}).
\end{align*}
Next we   use  \eqref{g-u}, and  \eqref{cocycle-condition} to see that
\begin{align*}
&X_i(\g(f^j)(\phi_j^{-1})\circ {\td\phi_{j-1}\dots \td\phi_0})= U^\ast_{\phi_{j-1}\dots \phi_0} U_{\phi_{j-1}\dots \phi_0} X_i U^\ast_{\phi_{j-1}\dots \phi_0}(\g(f^j)(\phi_j^{-1})=\\\notag
&U^\ast_{\phi_{j-1}\dots \phi_0} \g(X_i\ns{1}(\phi_{j-1}\dots \phi_0)) X\ns{0}(\g(f^j)(\phi_j^{-1}))=\\\notag
&U^\ast_{\phi_{j-1}\dots \phi_0} \g(X_i\ns{1}(\phi_{j}^{-1}\dots \phi_q^{-1})) X\ns{0}(\g(f^j)(\phi_j^{-1}))=\\\notag
&U^\ast_{\phi_{j-1}\dots \phi_0} \g(X_i\ns{1})(\phi_{j}^{-1}) \g(X_i\ns{2})(\phi_{j+1}^{-1})\circ{\td\phi_{j}}\cdots \\\label{chi-dag-1}
  &\cdots \g(X_i\ns{q-j+1}(\phi_q^{-1}) \circ{\td\phi_{q-1}\dots \td\phi_j)} X\ns{0}(\g(f^j)(\phi_j^{-1})).
\end{align*}
On the other hand, due to the fact that $X_i$ are  primitive one sees that
\begin{align*}
&\p_{\Fg^\ast}(\a\ot f^1\odots f^q)= d\a\ot f^1\odots f^q -\\
&\sum_{i}\t^i\wdg \a\ot  X_i\bullet(f^1\odots f^q)= d\a\ot f^1\odots f^q -\\
&\sum_{i}\t^i\wdg \a\ot  X_i\ps{1}\ns{0}\rt f^1\ot X_i\ps{1}\ns{1} (X\ps{1}\ns{0}\rt f^2)\ot\cdots\\
& \cdots\ot X_i\ps{q+1}\ns{1}\cdots  X_i\ps{q}\ns{1} (X_i\ps{q+1}\rt f^q)= \\
& d\a\ot f^1\odots f^q -\\
&\sum_{i}\t^i\wdg \a\ot  X_i\ps{1}\ns{0}\rt f^1\ot X_i\ps{1}\ns{1} (X_i\ps{2}\ns{0}\rt f^2)\ot\cdots\\
& \cdots\ot X_i\ps{q+1}\ns{1}\cdots  X_i\ps{q}\ns{1} (X_i\ps{q+1}\rt f^q)=\\
 &d\a\ot f^1\odots f^q -\\
&\sum_{i,j}\t^i\wdg \a\ot   f^1\odots   f^{j-1}\ot X_i\ns{0}\rt f^j\ot X_i\ns{1}f^{j+1} \odots X_i\ns{q-j}f^{q}).
\end{align*}
In view of  \eqref{g-u-f}, it follows
that   $\chi^\dag \p_{\Fg^\ast}= d_{\Om_{c,\G}(G)}\chi^\dag$.
Since Lemma \ref{invariant-trace-lemma} ensures that
$\wdg^\bullet \Fg^\ast$ is a SAYD over $\Fc$ and $\tau$ is $\wdg^\bullet\Fg^\ast$-trace,
by  \cite{hkrs2} $\chi^\dag$ is a horizontal cyclic map,
 and hence it commutes with $b_{\Fc}^\ast$ and $B_{\Fc}^\ast$.
\end{proof}
Recalling the map  $\Ic^{-1}:C^{p,q}_{\rm coinv}(\Fg^\ast,\Fc)\ra C^{p,q}(\Fg^\ast,\Fc)$
defined in \eqref{I-1-inverse} and the antisymmetrization map
$\td \a_\Fc:  C^{p,q}_{\rm c-w}(\Fg^\ast,\Fc)\ra C^{p,q}_{\rm coinv}(\Fg^\ast,\Fc) $ defined in \eqref{anti-sym}, we can now prove the following statement.
\begin{proposition} \label{thm:dgmap}
The following diagram is commutative
\begin{equation}
\begin{xy}
\xymatrix{  C^{\bullet,q}(\Fg^\ast,\Fc)\ar[rr]^{\chi^\dag}&& C^q(\Om_{c,\G}(G))\\
C^{\bullet,q}_{\rm c-w}(\Fg^\ast,\Fc)\ar[rr]^{\Theta}
\ar[u]^{\Ic^{-1}\circ \td \a_\Fc}& &C^{\bullet}_{\rm Bott}(\Om_q(G), \G)\ar[u]_{\Psi}
}
\end{xy}
\end{equation}
\end{proposition}
\begin{proof}
Decomposing each $f \in \Fc$ as $\, f = \ve(f)1 + f_{(0)}$, one sees that
\begin{align*}
f^0\wdots f^q \, = \, \sum_{k=0}^q (-1)^k \, \ve(f^k)1 \wdg f_{(0)}^0\wdots \check{f^k}\wdots f_{(0)}^q .
\end{align*}
Thus, it suffices to apply the map $\chi^{\wdg}:= \Psi\circ\Theta$ of \eqref{chi-wdg-formula-1}
for monomials of the form
$\, \a\ot 1 \wdg f^1\wdots f^q$, with $f^k \in \Fc_{(0)} := \Ker \ve$.
 In this case the $\ve(f)$ factor disappears from all the above expressions, and
 its formula  becomes
\begin{align}\label{chi-wdg-formula-2}
\begin{split}
& \chi^\wdg(  \a\ot 1\wdg f^1\wdots f^q)(\om_0U^\ast_{\phi_0}, \om_1U^\ast_{\phi_1},\dots, \om_qU^\ast_{\phi_q})
=\\
 & (-1)^{\deg(\a)+q} \sum_{\s\in S_{q}} (-1)^\s \int_G
 \gbar(S(f^{\s(1))})(\phi_0^{-1})\dots \gbar(S(f^{\s(q)}))(\phi_{0}^{-1}\dots \phi_{q-1}^{-1})\\
 &\qquad \qquad \cdot
 \td{\a} \om_0\;{\phi_0}^\ast\om_1\,\dots {(\phi_{q-1}\dots \phi_{0})}^\ast\om_q .
 \end{split}
\end{align}
if $\phi_q\dots \phi_0=1$ and $0$ otherwise.

On the other hand a direct computation shows that
\begin{align*}
&\Ic^{-1}\td \a_\Fc(\a\ot 1\wdg f^1\wdots f^q)= \\
&~~~~~~~~~~~~~(-1)^q\sum_{\s\in S_q}\a\ot f^{\s(1)}\ps{1}\ot f^{\s(1)}\ps{2} f^{\s(2)}\ps{1}\ot\cdots\\
 &~~~~~~~~~~~~~~~~~~~~~~~\cdots\ot f^{\s(1)}\ps{q-1}\cdots  f^{\s(q-1)}\ps{1}\ot f^{\s(1)}\ps{q}\cdots  f^{\s(q-1)}\ps{2}f^{\s(q)}.
\end{align*}

Composing with $\chi^\dag$, one obtains
\begin{align*}
&\chi^\dag(\Ic^{-1}\td \a_\Fc(\a\ot 1\wdg f^1\wdots f^q))(\om_0U^\ast_{\phi_0}, \om_1U^\ast_{\phi_1},\dots, \om_qU^\ast_{\phi_q})=\\
&(-1)^{\deg(\a)+q}\sum_{\s}(-1)^\s\displaystyle \int_G \td\a\,\om_0 \,\phi_0^\ast\om_1\dots (\phi_{q-1}\cdots\phi_{0})^\ast\om_q\, \gbar(f^{\s(1)}\ps{1})(\phi_1^{-1})\circ \tilde{\phi_0}\cdots\\\notag
& ~\cdots \gbar(f^{\s(1)}\ps{q-1}\cdots  f^{\s(q-1)}\ps{1})(\phi_{q-1}^{-1})\circ {\td\phi_{q-2}\dots \td\phi_0)}\; \\
&~~~~~~~~~~~~~~~~~~~~~~~~~~~~~~~~~~\gbar(f^{\s(1)}\ps{q}\cdots  f^{\s(q-1)}\ps{2}f^{\s(q)})(\phi_q^{-1})\circ {\td\phi_{q-1}\dots \td\phi_0)}.
\end{align*}

Using the fact that $\gbar$ is an algebra map and  \eqref{cocycle-condition} one sees that
\begin{align*}
&\gbar(f^{\s(1)}\ps{1})(\phi_1^{-1})\circ \tilde{\phi_0}\cdots \gbar(f^{\s(1)}\ps{q-1}\cdots  f^{\s(q-1)}\ps{1})(\phi_{q-1}^{-1})\circ {\td\phi_{q-2}\dots \td\phi_0}\; \\
&\gbar(f^{\s(1)}\ps{q}\cdots  f^{\s(q-1)}\ps{2}f^{\s(q)})(\phi_q^{-1})\circ {\td\phi_{q-1}\dots \td\phi_0}=\\
&\gbar(f^{\s(1)})(\phi_1^{-1}\cdots \phi_q^{-1})\circ \tilde{\phi_0}\cdots\\\notag
& ~\cdots \gbar(f^{\s(q-1)})(\phi_{q-1}^{-1}\cdots \phi_q^{-1})\circ {\td\phi_{q-2}\dots \td\phi_0)}\;\gbar(f^{\s(q)})(\phi_q^{-1})\circ {\td\phi_{q-1}\dots \td\phi_0}.
\end{align*}
In view of this simplification and since $\phi_i^{-1}\cdots \phi_q^{-1}=\phi_{i-1}\cdots \phi_{0}$,
it follows that

\begin{align}\notag
&\chi^\dag(\Ic^{-1}\td \a_\Fc(\a\ot 1\wdg f^1\wdots f^q))(\om_0U^\ast_{\phi_0}, \om_1U^\ast_{\phi_1},\dots, \om_qU^\ast_{\phi_q})=\\\notag
&(-1)^{\deg(\a)+q}\sum_{\s}(-1)^\s\displaystyle \int_G \td\a\,\om_0 \,\phi_0^\ast\om_1\dots (\phi_{q-1}\cdots\phi_{0})^\ast\om_q\, \gbar(f^{\s(1)})(\phi_0)\circ \tilde{\phi_0}\cdots\\\label{chi-dag-proof}
& ~\cdots \gbar(f^{\s(q-1)})(\phi_{q-2}\cdots \phi_0)\circ {\td\phi_{q-2}\dots \td\phi_0}\;\gbar(f^{\s(q)})(\phi_{q-1}\dots \phi_0)\circ {\td\phi_{q-1}\dots \td\phi_0}.
\end{align}
Finally, repeated use of the identity \eqref{S(f)(psi)} shows that the outcomes of
  \eqref{chi-dag-proof} and \eqref{chi-wdg-formula-2} are identical.
\end{proof}

\begin{remark} \label{rem:dgactmap}
The transfer of characteristic classes from Gelfand-Fuks cohomology to the
cyclic cohomology of action groupoids can be summarized in the following
(quasi-)commutative diagram:
\begin{equation}
\begin{xy}
\xymatrix{    & & && C^{\bullet}(\Ac_{c,\G}(G)) \\
&&C^{\bullet,\bullet}(\Fg^\ast,\Fc)\ar[rr]^{\chi^\dag}&& C^{\bullet}(\Om_{c,\G}(G)) \ar[u]^\daleth\\
C^{\bullet}_{\rm top}(\Fa)\ar[rr]^{\Ec_{\rm LH}}&&C^{\bullet,\bullet}_{\rm c-w}(\Fg^\ast,\Fc)\ar[rr]^{\Theta}
\ar[u]^{\Ic^{-1}\circ \td \a_\Fc}& &C^{\bullet}_{\rm Bott}(\Om_\bullet(G), \G)\ar[u]_{\Psi}
\ar@/_3pc/[uu]_\Phi
}
\end{xy}
\end{equation}
\end{remark}
The two vertical arrows which have not appeared before are defined as follows.
The map $\daleth$ is similar to \eqref{exve2}, but with differential $d$ and with the
homotopy given by the cyclic analogue of the Cartan formula. By
Gorokhovsky's result~\cite[Theorem 19]{Sasha02}, $\daleth \circ \Psi$ induces the
same map in cohomology as Connes' map $\Phi$.
For completeness, we recall from~\cite[Chap. III, Theorem 14]{book} the definition of $\Phi$.

Let $\Bc (G)$  be the tensor product,
$\, {\Bc} (G)= \Om_c(G) \ot \Lb \, \Cb [\G']$, where $\G' = \G \setminus \{e\}$.
One labels the generators of $\Cb [\G']$ as
$\d_{\phi}$, $\phi \in \G$, with $\d_e = 0$, and then forms the crossed product
\begin{equation*}
\Cc = {\Bc}(G) \rtimes \G ,
\end{equation*}
with respect to the commutation rules
\begin{align*}
&U_{\phi}^\ast \, \om \, U_{\phi} = \phi^\ast \, \om = \om \circ \phi , &\qquad
\, \om \in \Om_c(G),\\
& U_{\phi_1}^\ast \, \d_{\phi_2} \, U_{\phi_1} =\d_{\phi_2 \circ \phi_1} -
\d_{\phi_1} , &\qquad  \phi_j \in \G \, .
\end{align*}
One endows the graded algebra  $\Cc$ with
the differential $d$ given by
\begin{equation*}
d (b \, U_{\phi}^\ast) = db \, U_{\phi}^\ast - (-1)^{\p b} \, b \, \d_{\phi} \,
U_{\phi}^\ast , \quad\qquad  d(\om\ot \td \d)=d\om\ot \td\d .
\end{equation*}
A cochain $\g \in C^{p,q}_{\rm Bott}(\Om(V), \G)$  determines a linear form
$\td\g$ on $\Cc$, by,
\begin{equation*}
\td\g(\om \ot \d_{g_1} \ldots \d_{g_q}) = \int_{V} \g (1, g_1 , \ldots ,g_q)\;\om ,\qquad  \td\g (b \, U_{\phi}^\ast) = 0 \quad \text{if} \;\;\; \phi \ne 1.
\end{equation*}
With this notation, the map  $\Phi: C^{p,q}_{\rm Bott}(\Om(G), \G)\ra C^{\dim G-p+q}(\Ac_{c.\G}(G))$
is given by
\begin{align} \label{Phi}
\begin{split}
&\Phi(\g)(x^0, \ldots, x^\ell)\;=\;
\lb_{p,q}\sum_{j=0}^l(-1)^{j(\ell-j)}\td\g(dx^{j+1}\cdots dx^\ell\; x^0\; dx^1\cdots dx^j) ,\\
&\text{where} \qquad \ell=\dim G-p+q \quad \text{and} \quad \lb_{p,q}=\frac{p!}{(\ell+1)!}.
\end{split}
\end{align}

\section{Relative case} \label{S:rel}
\subsection{Relative Hopf cyclic bicomplexes}
In this section we extend the previous results to the relative case.
Let $\Fg_0$ be the linear isotropy subalgebra of $\Fg$ and let
$\Fh\subseteq \Fg_0$ be a subalgebra.
One defines the relative Hopf cyclic cohomology of the Hopf algebra $\Hc$
with respect to the Hopf subalgebra $\Vc:=U(\Fh)$
and with coefficients in $\Cb_\d$  as the  Hopf cyclic cohomology of the $\Hc$-module coalgebra $\Cc:=\Hc\ot_\Vc\Cb$  with coefficients in $\Cb_\d$.
Here $\Vc$ acts on $\Hc$ by right multiplication and on $\Cb$ by the counit  and  $\Hc$ also  acts on $\Cc$ by left multiplication. The cochain spaces of the complex for relative Hopf cyclic
cohomology are
\begin{equation}\label{relative-complex}
C^n(\Hc,\Vc, \Cb)= \Cb_\d\ot_\Hc \Cc^{\ot n+1}\cong\Cb\ot_\Vc \Cc^{\ot n} ,
\end{equation}
and we refer the reader to \cite{big} or \cite[page 758]{mr09} for the definition of the
coboundaries. According to \cite[Theorem 3.16]{mr09},  \eqref{relative-complex} can be
identified via  maps similar to $\Psi_{\acl}$ defined in \eqref{PSI-1}
with the diagonal of the following bicocyclic modules:
\begin{align}
&C^{p,q}_\Hc(\Qc,\Fc, \Cb_\d):= \Cb_\d\ot_\Hc \Qc^{\ot p+1}\ot \Fc^{\ot q+1},\\
&C^{p,q}(\Qc,\Fc, \Cb_\d):= \Cb_\d\ot_\Vc \Qc^{\ot p}\ot \Fc^{\ot q}.
\end{align}
Here  $\Qc:=U(\Fg)\ot_\Vc\Cb$, and $\Vc$ acts on $\Qc^{\ot p}\ot \Fc^{\ot q}$ diagonally by left multiplication on $\Qc$, and by the restriction of the action of  $\Uc$ on $\Fc$. Precisely,
\begin{align}
\begin{split}
&k\cdot (\td v\ot \td f)= k\ps{1}\cdot \td v\ot k\ps{2}\cdot \td f, \\
&k\cdot (v^1\odots v^p)= k\ps{1}v^1\odots k\ps{p}v^p,\\
& k\cdot(f^1\odots f^q)= k\ps{1}\rt f^1\odots k\ps{q}\rt f^q.
\end{split}
\end{align}
$\Vc$ acts from the right on $\Cb_\d:=\Cb$ by the restriction  of $\d:\Hc\ra \Cb$.
Diagrammatically, the bicomplex $C^{\bullet,\bullet}(\Qc,\Fc,\Cb_\d)$ is
depicted as follows
\begin{align}\label{relative-UF}
\begin{xy} \xymatrix{  \vdots\ar@<.6 ex>[d]^{\uparrow B} & \vdots\ar@<.6 ex>[d]^{\uparrow B}
 &\vdots \ar@<.6 ex>[d]^{\uparrow B} & &\\
\Cb_\d \ot_\Vc \Qc^{\ot 2} \ar@<.6 ex>[r]^{\hb}\ar@<.6
ex>[u]^{  \uparrow b  } \ar@<.6 ex>[d]^{\uparrow B}&
  \Cb_\d\ot_\Vc  \Qc^{\ot 2}\ot \Fc   \ar@<.6 ex>[r]^{\hb}\ar@<.6 ex>[l]^{\hB}\ar@<.6 ex>[u]^{  \uparrow b  }
   \ar@<.6 ex>[d]^{\uparrow B}&\Cb_\d\ot_\Vc  \Qc^{\ot 2}\ot\Fc^{\ot 2}
   \ar@<.6 ex>[r]^{~~\hb}\ar@<.6 ex>[l]^{\hB}\ar@<.6 ex>[u]^{  \uparrow b  }
   \ar@<.6 ex>[d]^{\uparrow B}&\ar@<.6 ex>[l]^{~~\hB} \hdots&\\
\Cb_\d \ot_\Vc \Qc \ar@<.6 ex>[r]^{\hb}\ar@<.6 ex>[u]^{  \uparrow b  }
 \ar@<.6 ex>[d]^{\uparrow B}&  \Cb_\d \ot_\Vc \Qc \ot\Fc \ar@<.6 ex>[r]^{\hb}
 \ar@<.6 ex>[l]^{\hB}\ar@<.6 ex>[u]^{  \uparrow b  } \ar@<.6 ex>[d]^{\uparrow B}
 &\Cb_\d\ot_\Vc  \Qc \ot \Fc^{\ot 2}  \ar@<.6 ex>[r]^{~~\hb}\ar@<.6 ex>[l]^{\hB}\ar@<.6 ex>[u]^{  \uparrow b  }
  \ar@<.6 ex>[d]^{\uparrow B}&\ar@<.6 ex>[l]^{~~\hB} \hdots&\\
\Cb_\d \ot_\Vc \Cb \ar@<.6 ex>[r]^{\hb}\ar@<.6 ex>[u]^{  \uparrow b  }&
\Cb_\d\ot_\Vc \Fc \ar@<.6 ex>[r]^{\hb}\ar[l]^{\hB}\ar@<.6
ex>[u]^{  \uparrow b  }&\Cb_\d\ot_\Vc \Fc^{\ot 2}  \ar@<.6
ex>[r]^{~~\hb}\ar@<.6 ex>[l]^{\hB}\ar@<1 ex >[u]^{  \uparrow b  }
&\ar@<.6 ex>[l]^{~~\hB} \hdots& ,}
\end{xy}
\end{align}
 with the boundaries and coboundaries defined similarly to their counterpart in the
 bicomplex \eqref{UF}.

Next, one defines the relative version of the bicomplex $C^{\bullet, \bullet}(\Fg, \Fc, \Cb_\d)$,
\cf \eqref{UF+}, namely the
bicomplex by $C^{\bullet, \bullet}(\Fg/\Fh, \Fc, \Cb_\d)$ depicted in the diagram
\begin{align}\label{relative-UF+}
\begin{xy} \xymatrix{  \vdots\ar[d]^{{\p_\Fg^\Fh}} & \vdots\ar[d]^{{\p_\Fg^\Fh}}
 &\vdots \ar[d]^{{\p_\Fg^\Fh}} & &\\
\Cb_\d \ot_\Fh \wg^2{\Fq} \ar@<.6 ex>[r]^{b_\Fc^\Fh } \ar[d]^{{\p_\Fg^\Fh}}&\ar@<.6 ex>[l]^{~~B_\Fc^\Fh }
 \Cb_\d\ot_\Fh \wg^2{\Fq}\ot\Fc  \ar@<.6 ex>[r]^{b_\Fc^\Fh } \ar[d]^{{\p_\Fg^\Fh}}
 &\ar@<.6 ex>[l]^{~~B_\Fc^\Fh }\Cb_\d\ot_\Fh \wg^2{\Fq}\ot\Fc^{\ot 2} \ar[r]^{~~~~~~b_\Fc^\Fh }   \ar[d]^{{\p_\Fg^\Fh}}& \ar@<.6 ex>[l]^{~~B_\Fc^\Fh }\hdots&\\
\Cb_\d \ot_\Fh {\Fq} \ar@<.6 ex>[r]^{b_\Fc^\Fh } \ar[d]^{{\p_\Fg^\Fh}}&\ar@<.6 ex>[l]^{~~B_\Fc^\Fh }  \Cb_\d\ot_\Fh {\Fq}\ot\Fc \ar@<.6 ex>[r]^{b_\Fc^\Fh } \ar[d]^{{\p_\Fg^\Fh}}&\ar@<.6 ex>[l]^{~~B_\Fc^\Fh }\Cb_\d\ot_\Fh {\Fq}\ot\Fc^{\ot 2}
   \ar[d]^{{\p_\Fg^\Fh}} \ar[r]^{~~~~~~~b_\Fc^\Fh }&\ar@<.6 ex>[l]^{~~B_\Fc^\Fh } \hdots&\\
\Cb_\d \ot_\Fh \Cb \ar@<.6 ex>[r]^{b_\Fc^\Fh }& \ar@<.6 ex>[l]^{~~B_\Fc^\Fh }  \Cb_\d \ot_\Fh\Fc \ar@<.6 ex>[r]^{b_\Fc^\Fh }&\ar@<.6 ex>[l]^{~~B_\Fc^\Fh } \Cb_\d \ot_\Fh\Fc^{\ot 2} \ar[r]^{~~~~~b_\Fc^\Fh } & \ar@<.6 ex>[l]^{~~B_\Fc^\Fh }\hdots&  ,}
\end{xy}
\end{align}
as follows:
 \begin{equation}
 C^{\bullet, \bullet}(\Fg/\Fh, \Fc, \Cb_\d):= \Cb_\d \ot_\Fh \wg^p{\Fq}\ot \Fc^{\ot q},
 \end{equation}
 where $\Fq:= \Fg/\Fh$, and   $\Fh$ as before  acts on $\Fq$ by adjoint action and on $\Fc$ via the restriction of $\Vc$ as above.
The three (co)boundaries of $C^{\bullet, \bullet}(\Fg/\Fh, \Fc, \Cb_\d)$, induced from those
of  \eqref{UF+}, can be explained as follows. First, the Lie algebra homology
boundary $\p_\Fg$ with coefficients in $\Cb_\d\ot \Fc^{\ot q}$, with
the action of $\Fg$ defined in \eqref{specact},
induces the relative version for the pair $(\Fg,\Fh)$,
to be denoted by $\p^\Fh_\Fg$.
The expressions of the operators $\hB$ and $\hb$ involve
 the coaction of $\Fc$ on  $\Cb_\d\ot\wg^p{\Fq}$. In turn, this
 coaction is induced by the coaction of $\wg^p\Fg$ defined
 in \eqref{wgcoact}, and is given by
\begin{align} \label{relative-wgcoact}
\Db_\Fg(\one\ot\dot X^1\wdots \dot X^q)\,=\, \one\ot  \dot X^1\ns{0}\wdots \dot X^q\ns{0}\ot
 X^1\ns{1}\dots  X^q\ns{1} ,
\end{align}
where $\dot X$ stands for the class of $X$ in $\Fq$. Since $\Fh$ coacts trivially,
the above coaction is seen to be well-defined. Thus,
$\hb$ (resp. $\hB$) are the Hochschild coboundary (resp. Connes boundary)
operators of $\Fc$ with coefficients in $\Cb_\d\ot\wg^p{\Fq}$.
One notes that all horizontal operators are $\Fh$-linear and hence they induce  their own
counterparts on $\Cb_\d\ot_\Fh \wg^p\Fq\ot\Fc^{\ot q}$.
We denote them by $b^\Fh_{\Fc}$ and $B^\Fh_{\Fc}$.

One next observes  that  the antisymmetrization map defined in \eqref{anti-sym} descends to its version in  the relative case,
\begin{align} \label{relative-antsym1}
&\td \a_\Fg:\Cb_\d\ot_\Fh
\wg^q\Fq \ot \Fc^{\ot p}\ra \Cb_\d\ot_\Vc \Qc^{\ot q} \ot \Fc^{\ot p} , \\ \notag
&\td\a(\one \ot_\Fh\dot X^1\wdots\dot X^p\ot \td f)= \frac{1}{p!} \sum_{\s\in S_p}(-1)^\s  \one \ot_\Vc \dot X^{\s(1)}\odots \dot X^{\s(p)}\ot \td f,
\end{align}
Since everything is induced from the absolute  case $\td\a$ is a map of  bicomplexes.
\medskip

Like  in the absolute case, we resort to the relative Poincar\'e map to pass
from homology to
the relative Lie algebra cohomology. By definition,
\begin{align}\label{relative-theta}
\begin{split}
&\FD_\Fg : (\wedge^q\Fq^\ast\ot \Fc^{\ot p})^\Fh \ra
  \wedge^{m-k}\Fq^\ast \ot_\Fh \wedge^{m-k-q}\Fq \ot \Fc^{\ot p} \\
 &\FD_\Fg(\eta\ot \td f) \, = \,    \varpi_\Fq^\ast\ot_\Fh \iota(\eta) \varpi_\Fq\ot \td f  ,
 \end{split}
\end{align}
where $k=\dim \Fh$, and $\varpi_\Fq^\ast$ (resp. $\varpi_\Fq$) is  the volume (resp. covolume) form of $\Fq$.

 To ensure that the relative Poincar\'e map is an isomorphism, we make the assumption
 that $\Fh$ is reductive; in particular it acts on $\Fc^{\ot q}$ semisimply.
 The resulting bicomplex is denoted by $C^{\bullet,\bullet}((\Fg/\Fh)^\ast ,\Fc)$ and is defined
 as follows:
 \begin{equation}
 C^{p,q}((\Fg/\Fh)^\ast ,\Fc):= (\wg^{p}\Fq^\ast\ot \Fc^{\ot q})^\Fh ;
 \end{equation}
it consists of invariant cochain when $\Fh$ acts on $\Fq^\ast$ by coadjoint representation and
on $\Fc$ as usual which is the restriction of the original action of $\Fg$ on $\Fc$.
There  are again three coboundaries $B_{\Fc^\ast}^\Fh$, $b_{\Fc^\ast}^\Fh$ and $\p_{\Fq^\ast}^\Fh$. The first two are defined as the restriction of the Connes boundary and  Hochschild coboundary of $\Fc$ with coefficients in $\wg^p\Fq^\ast$, where the coaction is induced by \eqref{coact2}.
One notes that the induced coaction is well-defied, as the coaction of $\Fh$ is trivial.
The other coboundary $\p^\Fh_{\Fg^\ast}$ is the relative Lie algebra cohomology cobundary
of $\Fg,\Fh$ with coefficients in $\Fc^{\ot q}$ with the original action of $\Fg$.
One notes that all maps descend to the invariant subspace with respect to $\Fh$,
since $\Fh$ is a subalgebra of $\Fg_0$, and $\Hc$ is a $\Fg_0$-Hopf algebra,
which means $\D,\mu, \ve, \eta, S$ of $\Hc$ are   $\Fg_0$-equivariant.
 We depict the new bicomplex in the following diagram:
\begin{align}\label{relative-UF+*}
\begin{xy} \xymatrix{  \vdots & \vdots
 &\vdots &&\\
 (\wdg^2{\Fq^\ast})^\Fh  \ar[u]^{\p^\Fh_{{\Fg^\ast}}}\ar@<.6 ex>[r]^{b^\Fh_{\Fc^\ast}~~~~~~~}& \ar@<.6 ex>[l]^{~~B^\Fh_{\Fc^\ast}} (\wdg^2{\Fq^\ast}\ot\Fc)^\Fh \ar[u]^{\p^\Fh_{{\Fg^\ast}}} \ar@<.6 ex>[r]^{b^\Fh_{\Fc^\ast}}& \ar@<.6 ex>[l]^{~~B^\Fh_{\Fc^\ast}}(\wdg^2{\Fq^\ast}\ot\Fc^{\ot 2})^\Fh \ar[u]^{\p^\Fh_{{\Fg^\ast}}} \ar@<.6 ex>[r]^{~~~~~~~~~b^\Fh_{\Fc^\ast}} & \ar@<.6 ex>[l]^{~~B^\Fh_{\Fc^\ast}}\hdots&  \\
 ({\Fq^\ast})^\Fh  \ar[u]^{\p^\Fh_{{\Fg^\ast}}}\ar@<.6 ex>[r]^{b^\Fh_{\Fc^\ast}~~~~~}& \ar@<.6 ex>[l]^{~~B^\Fh_{\Fc^\ast}} ({\Fq^\ast}\ot\Fc)^\Fh \ar[u]^{\p^\Fh_{{\Fg^\ast}}} \ar@<.6 ex>[r]^{b^\Fh_{\Fc^\ast}}& \ar@<.6 ex>[l]^{~~B^\Fh_{\Fc^\ast}} ({\Fq^\ast}\ot \Fc^{\ot 2})^\Fh \ar[u]^{\p^\Fh_{{\Fg^\ast}}} \ar@<.6 ex>[r]^{~~~~~b^\Fh_{\Fc^\ast} }&\ar@<.6 ex>[l]^{~~B^\Fh_{\Fc^\ast}} \hdots&  \\
 \Cb  \ar[u]^{\p^\Fh_{{\Fg^\ast}}}\ar@<.6 ex>[r]^{b^\Fh_{\Fc^\ast}~~~~~~~}& \ar@<.6 ex>[l]^{~~B^\Fh_{\Fc^\ast}} (\Cb\ot \Fc)^\Fh \ar[u]^{\p^\Fh_{{\Fg^\ast}}}\ar@<.6 ex>[r]^{b^\Fh_{\Fc^\ast}}& \ar@<.6 ex>[l]^{~~B^\Fh_{\Fc^\ast}} (\Cb\ot \Fc^{\ot 2})^\Fh \ar[u]^{\p^\Fh_{{\Fg^\ast}}} \ar@<.6 ex>[r]^{~~~~~b^\Fh_{\Fc^\ast}} &\ar@<.6 ex>[l]^{~~B^\Fh_{\Fc^\ast}} \hdots&. }
\end{xy}
\end{align}

 The next step consists in identifying  $C^{\bullet, \bullet}((\Fg/\Fh)^\ast ,\Fc)$ with the
 $\Fc$-coinvariant version $C^{\bullet, \bullet}_{\rm coinv}((\Fg/\Fh)^\ast ,\Fc)$, defined as
  \begin{equation}
  C^{p,q}_{\rm coinv}((\Fg/\Fh)^\ast ,\Fc)= ((\wg^p{\Fq}^\ast\ot \Fc^{\ot q+1})^\Fh)^\Fc ;
  \end{equation}
 precisely, an element of $C^{p,q}((\Fg/\Fh)^\ast ,\Fc)^\Fc$ is an
  element $\a\ot\td f \in C^{p,q}_{\rm coinv}(\Fg^\ast, \Fc):=(\wdg^2{{\Fg}^\ast}\ot \Fc^{\ot q+1})^\Fc$  such that
    \begin{equation}
   \Lc_X(\a\ot \td f)=\i_X(\a)=0, \quad \text{for all}\;\;\;X\in \Fh,
  \end{equation}
  where $\Lc_X$ stands for Lie derivative of $X$, with $\Fh$ acting as before, and $\i_X$ is the contraction with respect to $X$.

  Again, since the structure of the Hopf algebra $\Hc$ is  $\Vc$-equivariant,
  the map $\Ic$  defined in \eqref{I-1} descends to  $C^{\bullet, \bullet}((\Fg/\Fh)^\ast ,\Fc)$
  and lands into  $C^{\bullet, \bullet}((\Fg/\Fh)^\ast ,\Fc)^\Fc$.
 It defines an  isomorphism, as the inverse $\Ic^{-1}$, defined in \eqref{I-1-inverse},
   also descends to the relative case. We denote the induced maps by $\Ic^H$ and $\Ic^{-1}_H$:
  \begin{equation}
  \xymatrix{
  C^{\bullet, \bullet}((\Fg/\Fh)^\ast ,\Fc)\ar@<.6 ex>[rr]^{~~\Ic^H}&&
  \ar@<.6 ex>[ll]^{~~\Ic^{-1}_H}C^{\bullet, \bullet}((\Fg/\Fh)^\ast ,\Fc)^\Fc.}
  \end{equation}

  As a result we get the following bicomplex
\begin{align}\label{relative-g*F}
\begin{xy} \xymatrix{  \vdots & \vdots
 &\vdots &&\\
 ((\wdg^2{\Fq^\ast}\ot \Fc)^\Fc)^\Fh  \ar[u]^{\p_\Fg^{\Fh,\rm coinv}}\ar@<.6 ex>[r]^{b_\Fc^{\Fh,\rm coinv}~~~~~~~}& \ar@<.6 ex>[l]^{~~B_\Fc^{\Fh,\rm coinv}} ((\wdg^2{\Fq^\ast}\ot\Fc^{\ot 2})^\Fc)^\Fh \ar[u]^{\p_\Fg^{\Fh,\rm coinv}} \ar@<.6 ex>[r]^{b_\Fc^{\Fh,\rm coinv}}& \ar@<.6 ex>[l]^{~~B_\Fc^{\Fh,\rm coinv}}((\wdg^2{\Fq^\ast}\ot\Fc^{\ot 3})^\Fc)^\Fh \ar[u]^{\p_\Fg^{\Fh,\rm coinv}} \ar@<.6 ex>[r]^{~~~~~~~~~b_\Fc^{\Fh,\rm coinv}} & \ar@<.6 ex>[l]^{~~B_\Fc^{\Fh,\rm coinv}}\hdots&  \\
 (({\Fq^\ast} \ot \Fc)^\Fc)^\Fh  \ar[u]^{\p_\Fg^{\Fh,\rm coinv}}\ar@<.6 ex>[r]^{b_\Fc^{\Fh,\rm coinv}~~~~~}& \ar@<.6 ex>[l]^{~~B_\Fc^{\Fh,\rm coinv}} (({\Fq^\ast}\ot\Fc^{\ot 2})^\Fc)^\Fh \ar[u]^{\p_\Fg^{\Fh,\rm coinv}} \ar@<.6 ex>[r]^{b_\Fc^{\Fh,\rm coinv}}& \ar@<.6 ex>[l]^{~~B_\Fc^{\Fh,\rm coinv}} (({\Fq^\ast}\ot \Fc^{\ot 3})^\Fc)^\Fh \ar[u]^{\p_\Fg^{\Fh,\rm coinv}} \ar@<.6 ex>[r]^{~~~~~b_\Fc^{\Fh,\rm coinv} }&\ar@<.6 ex>[l]^{~~B_\Fc^{\Fh,\rm coinv}} \hdots&  \\
 ((\Cb\ot \Fc)^\Fc)^\Fh  \ar[u]^{\p_\Fg^{\Fh,\rm coinv}}\ar@<.6 ex>[r]^{b_\Fc^{\Fh,\rm coinv}~~~~~~~}& \ar@<.6 ex>[l]^{~~B_\Fc^{\Fh,\rm coinv}} ((\Cb\ot \Fc^{\ot 2})^\Fc)^\Fh \ar[u]^{\p_\Fg^{\Fh,\rm coinv}}\ar@<.6 ex>[r]^{b_\Fc^{\Fh,\rm coinv}}& \ar@<.6 ex>[l]^{~~B_\Fc^{\Fh,\rm coinv}} ((\Cb\ot \Fc^{\ot 3})^\Fc)^\Fh \ar[u]^{\p_\Fg^{\Fh,\rm coinv}} \ar@<.6 ex>[r]^{~~~~~b_\Fc^{\Fh,\rm coinv}} &\ar@<.6 ex>[l]^{~~B_\Fc^{\Fh,\rm coinv}} \hdots&,  }
\end{xy}
\end{align}
with the coboundaries $\p_\Fg^{\Fh,\rm coinv}$, $b_\Fg^{\Fh,\rm coinv}$,
 and the boundary  $B_\Fg^{\Fh,\rm coinv}$ induced from
     $\p_\Fg^{\rm coinv}$, and $b_\Fc^{\rm coinv}$ and  $B_\Fc^{\rm coinv}$ respectively.
\medskip

We let $H$ act on $\Fq:=\Fg/\Fh$ by the adjoint representation, and on ${\Fq}^\ast$
by its transposed.
We also let $H$ act on $\FN$ from the right by the restriction of the action of $G$ on $\FN$. We
this understood, we define the bicomplex
\begin{equation}
C^{p,q}_{\rm pol}(\FN, (\Fg/\Fh)^\ast)=C^q_{\rm pol}(\FN, \wg^p{\Fq}^\ast))^H .
\end{equation}
A cochain $c\in C^{p,q}_{\rm pol}(\FN, (\Fg/\Fh)^\ast)$ is a  homogenous
 polynomial $q$-cochain on $\FN$ with values in $\wg^p{\Fq}^\ast$, which is
 invariant under the action of $H$;
 equivalently, it is a nonhomogeneous polynomial cochain
 $c\in C^{p,q}_{\rm pol}(\FN, \Fg^\ast):= C^q_{\rm pol}(\FN, \wg^p\Fg^\ast)$ such that,
\begin{align}
\begin{split}
& \i_X(c(\psi_0, \dots, \psi_q))=0, \quad  X\in \Fh,\\
&c(\psi_0\lt \phi, \dots, \psi_q\lt \phi)=\Ad_\phi c(\psi_0, \dots,\psi_q), \quad \phi\in H.
\end{split}
\end{align}
One notes that  $C^{\bullet, \bullet}_{H,\rm pol}(\FN,\Fq^\ast)$ is well-defined as $H\subset G_0$, the linear isotropy subgroup of $G$, and the
action of $G_0$ by left translation  commutes with action of $\FN$ on $G$.
The coboundaries  $\p_{\rm pol}$, $b_{\rm pol}$, and boundary $B_{\rm pol}$,  defined in \eqref{p-pol},  \eqref{homb}, and \eqref{t-s-pol}, induce their own counterparts on
$C^{\bullet,\bullet}_{\rm pol}(\FN, (\Fg/\Fh)^\ast)$. These
are denoted in the diagram below
by  $\p^H_{\rm pol}$, $b^H_{\rm pol}$, and  $B^H_{\rm pol}$.
\begin{align} \label{relative-g*N}
\begin{xy} \xymatrix{ \vdots & \vdots
 &\vdots  & \\
C_{\rm pol}^0 \big(\FN , \wdg^2{{\Fq}^\ast} \big)^H
  \ar@<.6 ex>[r]^{{b^H_{\rm pol}}} \ar[u]^{{\p^H_{\rm pol}}}&\ar@<.6 ex>[l]^{{B^H_{\rm pol}}}
C_{\rm pol}^1 \big(\FN , \wdg^2{{\Fq}^\ast}  \big)^H
  \ar@<.6 ex>[r]^{{b^H_{\rm pol}}} \ar[u]^{{\p^H_{\rm pol}}}
 & \ar@<.6 ex>[l]^{{B^H_{\rm pol}}}C_{\rm pol}^2 \big(\FN ,  \wg^2{{\Fq}^\ast}  \big)^H
 \ar[u]^{{\p^H_{\rm pol}}}  \ar@<.6 ex>[r]^{{~~~~~b^H_{\rm pol}}}& \ar@<.6 ex>[l]^{{~~~~~~B^H_{\rm pol}}}\hdots\\
C_{\rm pol}^0 \big(\FN , \Fg \big)^H
 \ar@<.6 ex>[r]^{{b^H_{\rm pol}}} \ar[u]^{{\p^H_{\rm pol}}}& \ar@<.6 ex>[l]^{{B^H_{\rm pol}}}C_{\rm pol}^1 \big(\FN , {{\Fq}^\ast}  \big)^H
\ar@<.6 ex>[r]^{{b^H_{\rm pol}}} \ar[u]^{{\p^H_{\rm pol}}}& \ar@<.6 ex>[l]^{{B^H_{\rm pol}}}C_{\rm pol}^2 \big(\FN ,  {{\Fq}^\ast}\big)^H
\ar[u]^{{\p^H_{\rm pol}}} \ar@<.6 ex>[r]^{{~~~~~b^H_{\rm pol}}}& \ar@<.6 ex>[l]^{{~~~~~B^H_{\rm pol}}}\hdots\\
C_{\rm pol}^0 \big(\FN , \Cb \big)^H \ar@<.6 ex>[r]^{{b^H_{\rm pol}}}\ar[u]^{{\p^H_{\rm pol}}}& \ar@<.6 ex>[l]^{{B^H_{\rm pol}}} C_{\rm pol}^1
\big(\FN , \Cb  \big)^H
 \ar@<.6 ex>[r]^{{b^H_{\rm pol}}}\ar[u]^{{\p^H_{\rm pol}}}& \ar@<.6 ex>[l]^{{B^H_{\rm pol}}}C_{\rm pol}^2 \big(\FN , \Cb \big)^H\ar[u]^{{\p^H_{\rm pol}}}
\ar@<.6 ex>[r]^{{~~~~~b^H_{\rm pol}}}& \ar@<.6 ex>[l]^{{~~~~~B^H_{\rm pol}}} \hdots  }
\end{xy}
\end{align}
We then identify $C^{\bullet,\bullet}_{\rm pol}(\FN, (\Fg/\Fh)^\ast)$
with $C^{\bullet,\bullet}_{\rm coinv}((\Fg/\Fh)^\ast ,\Fc)$,
via the isomorphism $\Jc^H$  induced by $\Jc$ (defined in \eqref{Jmap}).
\medskip

Finally, recalling the subcomplex $\, C^{p,q}_{\rm c-w}(\Fg^\ast, \Fc)$ of
$ C^{\bullet,\bullet}_{\rm coinv}(\Fg^\ast, \Fc)$, \cf \eqref{wedge-coinv},
formed of the $\Fc$-coinvariant cochains $(\wedge ^{p}\Fg^\ast\ot \wedge ^{q+1}\Fc)^{\Fc}$,
we define its relative version
$\, C^{p,q}_{\rm c-w}(\Fg^\ast,\Fh^\ast \Fc):= ((\wg^p\Fq^\ast\ot \wg^{q+1}\Fc)^\Fh)^\Fc$.
It consists of those $\a\ot \td f\in C^{p,q}_{\rm c-w}(\Fg^\ast, \Fc)$ such that
\begin{equation}
\i_X(\a)=\Lc_X(f\ot \td f)=0, \qquad X\in \Fh.
\end{equation}
The antisymmetrization map $\a_\Fc$ idenifies $C^{\bullet,\bullet}_{\rm c-w}((\Fg/\Fh)^\ast ,\Fc)$
with a subcomplex  of $C^{\bullet, \bullet}_{\rm coinv}((\Fg/\Fh)^\ast ,\Fc)$.
The restrictions of $\p_\Fg^{H,\rm coinv}$, and $b_\Fc^{H,\rm coinv}$ and  $B_\Fc^{H,\rm coinv}$ are obviously well-defined, and one notes again that, due to the antisymmetrization,
$B_\Fc^{H,\rm coinv}\mid_{C^{\bullet,\bullet}_{\rm c-w}((\Fg/\Fh)^\ast ,\Fc)}=0$. We denote the
induced coboundaries by  $\p^H_{\rm c-w}$ and $b^H_{\rm c-w}$, and record the resulting
bicomplex in the following diagram:

\begin{align}\label{relative-wedge-coinv}
\begin{xy} \xymatrix{  \vdots & \vdots
 &\vdots &&\\
 (\wdg^2\Fq^\ast)^\Fh  \ar[u]^{\p^H_{\rm c-w}}\ar[r]^{b^H_{\rm c-w}~~~~~~~}&  ((\wdg^2\Fq^\ast\ot\wdg^2 \Fc)^\Fc)^\Fh \ar[u]^{\p_{\rm  c-w}} \ar[r]^{b^H_{\rm c-w}}& ((\wdg^2\Fq^\ast\ot\wdg^3\Fc)^\Fc)^\Fh \ar[u]^{\p^H_{\rm c-w}} \ar[r]^{~~~~~~~~~~~~b^H_{\rm c-w}} & \hdots&  \\
 (\Fq^\ast)^\Fh  \ar[u]^{\p^H_{\rm c-w}}\ar[r]^{b^H_{\rm c-w}~~~~~}&  ((\Fq^\ast\ot\wdg^2 \Fc)^\Fc)^\Fh \ar[u]^{\p^H_{\rm c-w}} \ar[r]^{b^H_{\rm c-w}}& ((\Fq^\ast\ot\wdg^3\Fc)^\Fc)^\Fh \ar[u]^{\p^H_{\rm c-w}} \ar[r]^{~~~~~~~~b^H_{\rm c-w} }& \hdots&  \\
 \Cb  \ar[u]^{\p^H_{\rm c-w}}\ar[r]^{b^H_{\rm c-w}~~~~~~~}&  ((\Cb\ot\wdg^2 \Fc)^\Fc)^\Fh \ar[u]^{\p^H_{\rm c-w}} \ar[r]^{b^H_{\rm c-w}}& ((\Cb\ot\wdg^3\Fc)^\Fc)^\Fh \ar[u]^{\p^H_{\rm c-w}} \ar[r]^{~~~~~~~~b^H_{\rm c-w}} & \hdots&,  }
\end{xy}
\end{align}
The outcome can be summarized as follows.
\begin{proposition}
Let $H$ be a connected
Lie subgroup of the linear isotropy subgroup $G_0\subset G$.
\begin{itemize}
\item[a)] The bicomplexes \eqref{relative-UF+*}, \eqref{relative-g*F}, \eqref{relative-g*N}
are mutually isomorphic and are quasi-isomorphic to \eqref{relative-wedge-coinv}.
    \item[b)] If in addition $H$ is reductive (and therefore acts semisimply on $\Fc$) then
    the above bicomplexes are quasi-isomorphic to \eqref{relative-UF+} and hence to \eqref{relative-UF}.
\end{itemize}
\end{proposition}

\subsection{Relative characteristic maps}
As above, the standing assumption is
that $H\subset G_0$ is a connected subgroup of the
linear isotropy subgroup. We let $V$ denote the coset space
 $G/H$ by $V$. Assuming for the time being that $H$ is also compact,
 one can identify $C^{\infty}_c(V)$ with the functions
$f\in C_c^{\infty}(G)$ such that $f(kg)=f(g)$ for all $k\in H$.
Since right translation by $H$ commutes with the action of $\G$,
the crossed product $\Ac_{c,\G}(V):=C^{\infty}_c(V)\rtimes \G$ is well-defined.

Let $\Vc:=U(\Fh)$ be the enveloping algebra of $\Fh$, considered as Hopf subalgebra of $\Hc$.
Using the  $\d$-invariant trace \eqref{inv-tr}
on $\Ac_{c,\G}(G)$, one defines the relative  version of the characteristic map
define \eqref{char-Hopf} as follows
\begin{align}
\begin{split}
&\chi_H:C^\bullet(\Hc,\Vc, \Cb_\d)\ra C^{\bullet}(\Ac_{c,\G}(V)),\\
&\chi_H(\one \ot_\Hc\dot h^0\odots \dot h^q)(a^0\odots a^q)=\tau(h^0(a^0)h^1(a^0)\cdots h^q(a^q)).
\end{split}
\end{align}
Here $C^n(\Hc, \Vc, \Cb_\d):= \Cb_\d\ot_\Hc \Cc^{\ot n+1}\cong \Cb_\d\ot_\Vc \Cc^{\ot n}$, and its  Hopf cyclic structure is that defined in \cite[\S 5]{big}, corresponding
to the coalgebra $\Cc=: \Hc\ot_\Vc\Cb$
acted upon by $\Hc$ via left multiplication.

Without any extra work one constructs a bicocyclic module similar to \eqref{GJ-bicocyclic} for computing the cyclic cohomology of $\Ac_{c,\G}(V)$. We denote this bicocyclic module by
\begin{equation}\label{relative-GJ-bicocyclic}
C^{(p,q)}(C_c^{\infty}(V),\G):=\Hom(C^{\infty}_c(V)^{\ot p+1}\ot \Cb\G^{\ot q+1}, \Cb),
\end{equation}
where its cyclic structure are similar to that of \eqref{GJ-bicocyclic}. One then defines the following bicocyclic map similar to \eqref{char},

$\chi:C^{p,q}_\Hc(\Qc,\Fc,\Cb_\d)\ra C^{p,q}(C^{\infty}_c(V),\G)$, by the formula
\begin{align}\label{relative-char}
&\chi\left(\one\ot \td \dot u\ot \td f\right)(\td g\mid \td U )=\tau\left(\dot u^0(g_0)\dots \dot u^p(g_p)\dbar(S(f^0))(U_{\phi_0})  \dots \dbar(S(f^{q}))(U_{\phi_q})\right).
\end{align}

The above map induces  characteristic maps in the level of total and diagonal of the above bicomplexes. The counterpart of
Theorem \ref{thm:actmap} reads as follows.

\begin{theorem} \label{thm:-relative-actmap}
For any compact, connected Lie subgroup of $H \sbs G_0$, the following diagram is commutative:
\begin{equation}\label{chidiagram2}
\xymatrix{
\ar[d]^{\Ic^{-1}_\Hc}C^n(\Hc,\Vc,\Cb_\d)\ar[rr]^ {\chi_{\act}}&&
C^n(\Ac^H_{c,\G})\ar[d]^{\Ic^{-1}_{\G^H}}\\
\ar[d]_{\Psi^{-1}_{\acl}}C^n_\Hc(\Fc\acl\Qc, \Cb_\d)\ar[rr]^ {\chi_{\act}}&&
 C^n({\Ac^H_{c, \G}}^{\rm o})\ar[d]^{\Psi^{-1}_{\G^H}}\\
\ar[d]^{\overline{SH}} C^{n,n}_\Hc(\Qc,\Fc,\Cb_\d)\ar[rr]^{\chi_{\rm diag}}&&
 \ar[d]^{\overline{SH}} C^{n,n}(C_c^{\infty}(V),\G)\\
\underset{p+q=n}{\bigoplus}C^{p,q}_\Hc(\Qc,\Fc,\Cb_\d)\ar[rr]^{\chi_{\rm tot}}&&
 \underset{p+q=n}{\bigoplus}C^{p,q}(C_c^{\infty}(V),\G).}
\end{equation}
\end{theorem}

\bigskip

We now drop the compactness assumption on $H$, and recall, \cf  \eqref{Theta-formula},
 the map $\, \Theta: \left( \wedge^p \Fg^\ast\ot \wdg^{q+1}\Fc \right)^\Fc \ra
C_{\rm Bott}^{p,q}(G, \Gb)$ given by the formula
\begin{align*}
\begin{split}
&\Theta (\sum_I \a\ot f^0 \wdg \cdots \wdg f^q )(\phi_0,\dots,\phi_q)
\, = \, \\
&= \sum_{\s} (-1)^\s   \gbar(S( f^{\s(0)}))(\phi_0^{-1})\dots \gbar(S(f^{\s(q)}))(\phi_q^{-1}){{\td\a}} \,.
\end{split}
\end{align*}
 Here ${{\td\a}}$ is the left invariant form on $G$ associated to $\a \in \wdg^p \Fg^\ast$.
 We also recall  the other characteristic map
 $\, \chi^\dag: \wdg^\bullet\Fg^\ast\ot \Fc^{\ot q}\ra C^q(\Om_{c,\G}(G))$,
 \cf \eqref{chi-dag-def-1} or \eqref{chi-dag-def-2}.

\begin{theorem} \label{relTheta}
Let $H\subseteq G_0\subset G$ be a Lie subgroup of linear isotropy subgroup and
let $\Fh \subseteq \Fg_0$ be its Lie algebra.
\begin{itemize}
\item[a)] If $c\in C^{p,q}_{c-w}((\Fg/\Fh)^\ast ,\Fc)$ then
$\Theta(c)\in C_{\rm Bott}^{p,q}(V, \Gb)$, hence $\Theta$ induces a map of bicomplexes
\begin{equation}
\Theta^H:C^{\bullet,q}_{\rm c-w}((\Fg/\Fh)^\ast ,\Fc)\ra C^{\bullet}_{\rm Bott}(\Om_q(V).
\end{equation}

\item[b)] The map $\chi^\dag$ induces a similar map in the level of relative cohomology

\begin{equation}
\chi^H_\dag : C^{\bullet,q}((\Fg/\Fh)^\ast ,\Fc)\ra C^q(\Om_{c,\G}(V)).
\end{equation}
\item[c)] The following diagram is commutative
\begin{equation}
\begin{xy}
\xymatrix{  C^{\bullet,q}((\Fg/\Fh)^\ast ,\Fc)\ar[rr]^{\chi_\dag^H}&& C^q(\Om_{c,\G}(V))\\
C^{\bullet,q}_{\rm c-w}((\Fg/\Fh)^\ast ,\Fc)\ar[rr]^{\Theta^H}
\ar[u]^{\Ic^{-1}_H\circ \td \a_\Fc}& &C^{\bullet}_{\rm Bott}(\Om_q(V), \G)\ar[u]_{\Psi^H}
.}
\end{xy}
\end{equation}
\end{itemize}
\end{theorem}
\begin{proof}
\begin{itemize}\item [a)]Let $c:=\a\ot f^0\wdots f^q$. By definition we know for any $ X\in \Fh$
\begin{align}\label{proof-relative-Theta-1}
&\Lc_X(\a\ot f^0\wdots f^q)=0\\\label{proof-relative-Theta-2}
&\i_X(\a)=0.
\end{align}
It suffices to show that, for any $\phi_i\in \Gb$,
\begin{equation}
\Lc_X(\Theta(c)(\phi_0, \ldots,\phi_q))=\i_X(\Theta(c)(\phi_0, \ldots,\phi_q))=0, \quad \text{for any} \;\;X\in \Fh.
\end{equation}
The second condition being just  \eqref{proof-relative-Theta-2}, we only need to
verify the first one. From \eqref{proof-relative-Theta-1}, one sees that for any $X\in \Fh$,
\begin{align}\label{proof-relative-Theta-3}
\begin{split}
&\Lc_X(\a)\ot f^0\wdots f^q=-\a\ot \Lc_X(f^0\wdots f^q)=\\
&-\sum_{i=0}^q \a\ot f^0\wdots X\rt f^i\wdots f^q.
\end{split}
\end{align}
In the following calculation we use in the stated succession
 \eqref{proof-relative-Theta-3}, Lemma \ref{S-linear},  the fact that $\Db(X)=X\ot 1$ for any $X\in \Fg_0$,  \eqref{cocycle-condition}, and finally the fact that $\gbar(f)(\vp)=\ve(f)$ for any $\vp\in G$.
\begin{align*}
\begin{split}
&\Lc_X(\Theta(c)(\phi_0, \ldots,\phi_q))=\\
& \sum_{\s} (-1)^\s   \gbar(S( f^{\s(0)}))(\phi_0^{-1})\dots \gbar(S(f^{\s(q)}))(\phi_q^{-1})\Lc_X{{\td\a}}=\\
& -\sum_{\s,i} (-1)^\s   \gbar(S( f^{\s(0)}))(\phi_0^{-1})\dots\gbar(S(X\rt f^{\s(i)}))(\phi_i^{-1})\dots  \gbar(S(f^{\s(q)}))(\phi_q^{-1}){{\td\a}}=\\
& -\sum_{\s,i} (-1)^\s   \gbar(S( f^{\s(0)}))(\phi_0^{-1})\dots X(\gbar(S(f^{\s(i)})))(\phi_i^{-1})\dots  \gbar(S(f^{\s(q)}))(\phi_q^{-1})\td\a=\\
& -\sum_{\s,i} (-1)^\s   \gbar(S( f^{\s(0)}))(\phi_0^{-1})\dots\dt \gbar(S(f^{\s(i)}))(\phi_i^{-1}exp(tX))\dots  \gbar(S(f^{\s(q)}))(\phi_q^{-1}){{\td\a}}=\\
&-\sum_{\s,i} (-1)^\s   \gbar(S( f^{\s(0)}))(\phi_0^{-1})\dots\dt \gbar(S(f^{\s(i)}))(\phi_i^{-1})\dots  \gbar(S(f^{\s(q)}))(\phi_q^{-1}){{\td\a}}=0
\end{split}
\end{align*}

Since $\Theta$ is a map of bicomplexes, so is $\Theta^H$.

\item[b)] The proof is entirely similar to that of a).

\item[c)] The arguments of the proof of Proposition \ref{thm:dgmap} apply here too.

\end{itemize}
\end{proof}

\begin{remark} \label{rem:dgrelmap}
{\rm One has the following relative analogue of Remark
 \ref{rem:dgactmap} for
the transfer transfer of characteristic classes from Gelfand-Fuks cohomology to the
cyclic cohomology of action groupoids:}
\begin{equation}
\begin{xy}
\xymatrix{    & &
&& C^{\bullet}(\Ac_{c,\G}(V)) \\
&&  C^{\bullet,\bullet}((\Fg/\Fh)^\ast ,\Fc)\ar[rr]^{\chi^H_\dag}&& C^{\bullet}(\Om_{c,\G}(V)) \ar[u]^{\daleth^H}\\
C^{\bullet}_{\rm top}(\Fa, \Fh)\ar[rr]^{\Ec_{\rm LH}}&&C^{\bullet,\bullet}_{\rm c-w}((\Fg/\Fh)^\ast ,\Fc)\ar[rr]^{\Theta^H}
\ar[u]^{\Ic^{-1}\circ \td \a_\Fc}& &C^{\bullet}_{\rm Bott}(\Om_\bullet(V), \G)\ar[u]_{\Psi^H}\ar@/_3pc/[uu]_{\Phi^H} .
}
\end{xy}
\end{equation}
\end{remark}
The relative  version of  $\Phi$ has the expression
 \begin{align} \label{Phi-rel}
\begin{split}
&\Phi^H(\g)(x^0, \ldots, x^\ell)\;=\;
\lb_{p,q}\sum_{j=0}^l(-1)^{j(\ell-j)}\td\g(dx^{j+1}\cdots dx^\ell\; x^0\; dx^1\cdots dx^j) ,\\
&\text{where} \qquad \ell=n-p+q \quad \text{and} \quad \lb_{p,q}=\frac{p!}{(\ell+1)!}.
\end{split}
\end{align}

\subsection{Hopf cyclic and equivariant Chern classes}
In this final subsection we illustrate the versatility of the cohomological apparatus
developed above by indicating how to transfer the universal Hopf cyclic Chern classes
from relative Hopf cyclic cohomology
to the Bott bicomplex and to the
cyclic cohomology of convolution algebras of action groupoids.
\bigskip

{\bf 4.3.1.\,} Let $\Hc_n$  be
  the Hopf algebra of the general pseudogroup $\Gb_n$, \cf \cite{cm2}.
  Recall that, according to
  \cite{mr09},  $\Hc_n^{\rm cop} =\Fc (\Gb_n)\acl \Uc(\Gb_n)$.
 As an algebra, in the notation of \cite{mr09},
 $\Hc_n^{\rm cop}$ is generated by
 $$\{1\acl X_k, 1\acl Y^i_j,  \eta^i_{j,k\mid l_1, \ldots, l_m }\acl 1 \, \mid \, m\in \Zb^+ , \,
 1\le i,j,k, l_1, \ldots, l_m  \le n \} ;
 $$
 $\{X_1,\ldots, X_n\}$ is the standard basis of vector fields on $\Rb^n$,
 $\{Y^i_j \, \mid \,  1\le i,j \le n \}$ is
 the standard basis of the Lie algebra $\Fg\Fl_n= \Fg\Fl_n(\Rb)$,
 and the quotient of the affine group $G$ by the linear group $H= \GL_n(\Rb)$ is
 identified  with $V = \Rb^n$. We caution that
the notation of the generators $\eta^i_{j,k\mid l_1, \ldots, l_m } \in \Fc$ is slightly
different from that employed earlier in the paper, reflecting the use of double indices
 for the basis elements of $\Fg\Fl_n$.

In \cite[\S 3.4.1]{mr09} we have shown that the relative Hopf cyclic cohomology
 $HP^\bullet (\Hc_n^{\rm cop} , \Uc(\Fg \Fl_n); \Cb_{\d})$ is isomorphic to the relative
 Hochschild cohomology  $HH^\bullet (\Hc_n^{\rm cop} , \Uc(\Fg \Fl_n); \Cb_{\d})$.
 Moreover, we exhibited a basis for the latter, constructed as follows.

 For each partition $\lb =  (\lb_1 \geq  \ldots \geq  \lb_k)$
 of the set $ \{ 1, \dots , p \}$, where $\, 1 \leq p \leq n$,
 let $\lb \in S_p$ also denote a permutation whose
 cycles have lengths $\lb_1 \geq  \ldots \geq  \lb_k$,  \ie
 representing
 the corresponding conjugacy class $[\lb] \in [S_p] $.
Define
\begin{equation*}
C_{p, \lb} := \sum  (-1)^\mu \one \ot X_{\mu(p+1)}\wdots
 X_{\mu(n)}\ot\eta^{j_1}_{ \mu(1),j_{{\lb
(1)}}} \wdots  \eta^{j_{p}}_{\mu(p),j_{{\lb (p)}}},
\end{equation*}
where the summation is over  all $\, \mu \in S_n$ and all $ \, 1\le j_1,j_2,\dots , j_p\le n$.
Then $\{ C_{p, \lb} \, ; \, 1 \leq p \leq n , \quad [\lb] \in [S_p] \}$
 are Hochschild cocycles and their classes form a basis of
the  group  $HH^\bullet (\Hc_n^{\rm cop} , \Uc(\Fg \Fl_n); \Cb_{\d})$.
\footnote[1]{The statement of \cite[Theorem 3.25]{mr09} erroneously identified
 the basis elements as cyclic classes; they merely parametrize all the cyclic classes.}
  Except for extreme cases ($p=0,n$ or $\lb = \Id$) these cannot be viewed as
representing cyclic classes. Indeed, while $b_\Fc(C_{p, \lb})=0$, in general
$C_{p, \lb}$ is not closed under the Lie algebra coboundary operator $\p_V$.
 What is actually proved in \cite{mr09} is that the horizontal
 cohomology of \eqref{relative-UF+} computes
  $HH^\bullet(\Hc_n^{\rm cop}, \Uc(gl_n), \Cb_\d)$, which is generated by the
 above classes $C_{p, \lb}$, and that
 the vertical coboundary $\p_V$ vanishes at the level of  Hochschild cohomology.
We shall now indicate how the Hochschild cocycles $C_{p, \lb}$ can be completed to
$(b_\Fc , \p_V)$-cocycles.

Denote  by $\Fd$ the commutative Lie algebra generated by
$\{\eta^i_{j,k} \, \mid \,  1\le i,j, k\le n \}$, and let
 $F=\Sc(\Fd)$ be the coresponding symmetric Hopf algebra.
It is known (see \eg ~\cite[XVIII.5]{kass}) that the Hochschild cohomology
of the coalgebra $F$ can be identified with the cohomology of
the graded complex $\wdg^\bullet \Fd$
endowed with the $0$ differential. More precisely, with $\a:\wdg^\bullet\Fd\ra F^{\ot \bullet}$
denoting the antisymmetrization map and
$$\xymatrix{\mu: F^{\ot \bullet}\ar[r]^{~~~~p_1^{\ot \bullet}} &\Fd^{\ot \bullet }\ar[r] & \wdg^\bullet\Fd}\,,$$
 the natural projection map, there is a homotopy operator
 $h:F^{\ot \bullet} \ra F^{\ot \bullet -1}$ such that the following relations hold:
\begin{align} \label{ho}
&\mu\circ \a=\Id,\qquad  \a\circ \mu = \Id+ hb_\Fc+b_\Fc h.
\end{align}
By a slight abuse of notation, we keep the same symbols for the extension of
these operators tensored by the identity to the complex
obtained by tensoring with $\wdg^{\bullet} V$.
 We start by observing that
 \begin{equation*}
 \p_V(C_{p,\lb})\in (\wdg^{p-1} V\ot F^{\ot n-p})^{\Fg \Fl_n} ;
 \end{equation*}
the proof uses the facts that the action of the $X_i$'s on $\Fc$ is characterized by
$X_\ell\rt \eta^i_{j,k}=\eta^i_{j,k \ell}$,
 that  $C_{p,\lb}$ is antisymmetric, and finally the ``Bianchi'' identity \cite[(1.28)]{mr09}.
 Moreover, one actually has
 \begin{equation*}
  \p_V(C_{p,\lb})\in \ker\mu\cap \ker b_\Fc .
 \end{equation*}
 As $\p_V$ and $b_\Fc$ commute, it follows from \eqref{ho} that
 \begin{equation*}
 b_\Fc \big((h \td \p_V)(C_{p,\lb})\big) \, =\, - \td\p_V (C_{p,\lb}) ,
 \end{equation*}
  where ${\td \p_V}$ stands for the signed version  of the Lie algebra coboundary.
By induction, one checks that
the cochains
$$C^{(i)}_{p,\lb} := (h \, \td\p_V)^i(C_{p,\lb})\in (\wdg^{p-i}V\ot F^{\ot n-p-i})^{\Fg \Fl_n} ,
 $$
where $1 \le i \le m_p:=\min\{p,n-p\}$, satisfy
  $$b_\Fc (C^{(i+1)}_{p,\Lb}) \, = \, -  \td\p_V (C^{(i)}_{p,\lb}) .
  $$
  Assembling them into a total cochain
 \begin{equation}
 TC_{p,\lb}:= \,  \sum_{i=0}^{m_p}  C^{(i)}_{p,\lb} ,
 \end{equation}
 one has by the very construction
  \begin{equation}
 (b\,+\, {\td \p_V})(TC_{p,\lb})=0 .
 \end{equation}

\medskip

  The total cocycles $TC_{p,\lb}$, which live in the bicomplex \eqref{relative-UF+},
 represent a basis of the cohomology group
 $HP^\bullet (\Hc_n^{\rm cop} , \Uc(\Fg \Fl_n); \Cb_{\d})$. Via characteristic maps these
 total cocycles can be transported to
 define cyclic cocycles in the various models of cyclic cohomology for \'etale groupoids.
 For the clarity of the exposition, in what follows we shall only deal with the top
 component of these Chern classes $C_{p,\lb} \, = \, I (TC_{p,\lb})$.
 \medskip

 To this end we proceed to transfer the above Hochschild cocycles, as well as their cyclic
  completions, to the bicomplexes
\eqref{relative-UF+*} and \eqref{relative-g*F}.
First, we pass to \eqref{relative-UF+*} by means of
the inverse of  Poincar\'e  isomorphism $\FD_\Fq^{-1}$ in relative case \eqref{relative-theta}.
Denoting by $\{\t^1, \ldots, \t^n\}$ dual basis $\{X_1, \ldots, X_n\}$, one obtains
\begin{equation*}
C_{p, \lb}^{\eqref{relative-UF+*}}:=\FD_\Fq^{-1}(C_{p, \lb}) := \sum_{\mu\in S_n}
\t^{\mu(1)}\wdots \t^{\mu(p)}\ot \eta^{j_1}_{ \mu(1),j_{{\lb
(1)}}} \wdots \eta^{j_{p}}_{\mu(p),j_{{\lb (p)}}}.
\end{equation*}

We next move $C_{p, \lb}^{\eqref{relative-UF+*}}$ via the relative version
$\Ic^H$ of the
map $\Ic$ defined in \eqref{I-1}, and denote the classes
$\Ic^H(C_{p, \lb}^{\eqref{relative-UF+*}})$ by $C_{p, \lb}^{\eqref{relative-g*F}}$.
In the bicomplex \eqref{relative-g*F},
the Hochschild cocycle  $C_{p, \lb}^{\eqref{relative-g*F}}$ has the following expression:
\begin{align} \label{Eq-C-p-lb}
&C_{p, \lb}^{\eqref{relative-g*F}}= \sum_{\mu\in S_n}
 \t^{\mu(1)}\wdots \t^{\mu(p)}\ot  \eta^{j_1}_{ \mu(1),j_{{\lb
(1)}}} \wdots \eta^{j_{p}}_{\mu(p),j_{{\lb (p)}}}\wg 1\\ \notag
&= \sum_{1\le i_s, j_t\le n}  \t^{i_1}\wdots \t^{i_p}\ot  \eta^{j_1}_{ i_1,j_{{\lb
(1)}}} \wdots \eta^{j_{p}}_{i_p,j_{{\lb (p)}}}\wg 1. \notag
\end{align}
Indeed, the map
 $\Ic^H: C^{p,q}(\Fg^\ast, \Fh^\ast, \Fc)\ra C^{p,q}_{\rm coinv}(\Fg^\ast, \Fh^\ast, \Fc)$
is given by the formula
\begin{multline*}
\Ic^H(\a\ot f^1\odots f^q)=\\
\a\ns{0}\ot f^1\ps{1}\ot S(f^1\ps{2})f^2\ps{1}\odots S(f^{q-1}\ps{2})f^q\ps{1}\ot  S(f^q\ps{2}\a\ns{1}).
\end{multline*}
Since $\Db_{\Fg^\ast}(\t^i)=\t^i\ot 1$ and $\eta^i_{jk}$ are primitive,
 we may assume that $\a\ns{0}\ot\a\ns{1}=\a\ot 1$ and $\D(f^i)=f^1\ot 1+ 1\ot f^i$.
One easily sees that
\begin{equation*}
\Ic^H(\a\ot f^1\odots f^q)=\a\ot \prod_{i=1}^q (1^{\ot i-1}\ot f^i\ot 1^{\ot q-i+1}- 1^{\ot i}\ot f^i\ot 1^{\ot q-i}).
\end{equation*}
Here the multiplication takes place in the algebra $\Fc^{\ot q+1}$.
After cancelations due to the antisymmetric form of $C_{p, \lb}^{\eqref{relative-UF+*}}$,
the desired formula shows up.

To simplify the notation, we henceforth redenote
$C_{p, \lb}^{\eqref{relative-g*F}}$ by $\Cc_{p,\lb}$.
The cocycles   $\Cc_{p,\lb}$ are coinvariant, since they belong to
 the  image of $\Ic^H$.  They are also totally antisymmetric, and therefore belong to
 $C^{p,p}_{c-w}(\Fg^\ast,\Fh^\ast,\Fc)$. Therefore, they can be moved by means of
 the map of complexes
 $\Theta^H: C^{p,p}_{c-w}(\Fg^\ast,\Fh^\ast,\Fc)\ra C^{p,p}_{\rm Bott}(\Om(V), \G)$
   (defined for absolute case  in \eqref{Theta-formula}) given by
\begin{align*}
\begin{split}
&\Theta^H(\sum_I \a_I\ot \lu{I}f^0 \wdg \cdots \wdg \lu{I}f^q )(\phi_0,\dots,\phi_q)
\, = \, \\
&~~~~~~~~~~~~~~= \sum_{I, \s} (-1)^\s
 \gbar(S(\lu{I}f^{0}))(\phi_{\s(0)}^{-1})\dots
  \gbar(S(\lu{I}f^{q}))(\phi_{\s(q)}^{-1}){{\td\a_I}} \,.
\end{split}
\end{align*}
Thus, the expressions
\begin{align}  \label{absBottCh}
\Cc_{p, \lb}^{\rm Bott}(\phi_0, \ldots, \phi_p)\, :=\, \Theta^H(\Cc_{p, \lb})(\phi_0, \ldots, \phi_p)
\end{align}
represent the equivariant Chern classes in the Bott bicomplex
$C^{\bullet, \bullet}_{\rm Bott}(\Om(V), \G)$.
\medskip

We now define the {\em connection displacement form}
$\, Q = (Q^i_j): \G \ra \Om^1(G)  \ot \Fg\Fl_n$ by
\begin{equation} \label{curvG}
Q^i_j (\phi) :=\sum_k \gbar(\eta^i_{j,k})(\phi^{-1}) \, \td\t^k  \in  \Om^1(G) , \qquad \phi \in \G.
\end{equation}
Taking into account the switch of notation to fit the conventions in \cite{mr09}, and
using the equation (1.15) in {\em loc.~cit.}, one has
\begin{align*}
\gbar(\eta^i_{j,k})(\phi^{-1})  (x, {\bf y})\, = \, \g^i_{j k}(\phi) (x, {\bf y}) \, = \,
\sum_\mu \left( {\bf y}^{-1} \cdot
{\phi}^{\prime} (x)^{-1} \cdot \part_{\mu} {\phi}^{\prime} (x) \cdot
{\bf y}\right)^i_j \, {\bf y}^{\mu}_k \, .
\end{align*}
On the other hand, $ \, \displaystyle
\td\t^k \, = \,\sum_\nu \left( {\bf y}^{-1} \right)^k_\nu dx^\nu$.
Substituting these expressions in \eqref{curvG} one obtains the following
simplified formula for the connection form
\begin{align*}
Q (\phi) (x, {\bf y})\, = \, {\bf y}^{-1} \cdot
  \vartheta (\phi) (x)\cdot {\bf y} \in  \Om^1(G) \ot \Fg\Fl_n,
 \end{align*}
 where $\, \vartheta = (\vartheta^i_j)$ is the induced connection displacement
 form  on $V = \Rb^n$,
\begin{align} \label{simpcurvG}
   \vartheta (\phi) = {\phi'}^{-1} \cdot d \phi' \in \Om^1(V)  \ot \Fg\Fl_n .
\end{align}
The latter satisfies the structure equation
\begin{align} \label{streq}
  d\vartheta (\phi) + \vartheta (\phi) \wdg \vartheta (\phi)  \, = 0 ,
    \end{align}
 which is just another form for the first ``Bianchi'' identity \cite[(1.28)]{mr09}.
It also fulfills the cocycle equation
\begin{align} \label{dispcocy}
  \vartheta (\phi_1 \circ \phi_2) \, = \,
{\phi^\prime_2}^{-1} \,  \vartheta (\phi_1)\, {\phi^\prime_2} \, + \, \vartheta (\phi_2) .
    \end{align}

\begin{proposition} \label{prop:BottCh}
The equivariant Chern class $\Ch^{\rm Bott}_{p, \lb} :=\Theta^H (TC_{p, \lb})$
is represented by a total cocycles whose bottom component
$\Cc_{p, \lb}^{\rm Bott} : = I(\Ch^{\rm Bott}_{p, \lb})$
has the expression
\begin{align}  \label{BottCh}
&\Cc_{p, \lb}^{\rm Bott}(\phi_0, \ldots, \phi_p) =\\ \notag
& ~~~~~~~~\sum_{\s}(-1)^\s \Tr(\vartheta(\phi_{\s(1)})\cdots \vartheta(\phi_{\s(\lb_1)}))
\wdots \Tr(\vartheta(\phi_{\s(\lb_k)})\cdots \vartheta(\phi_{\s(p)})),
\end{align}
where $\s$ runs over the group of permutations of the set $\{0, 1, \ldots , p\}$.
\end{proposition}

\begin{proof} From \eqref{Eq-C-p-lb}
and the fact that the permutation $\lb$ is equal to the product of
cycles $(1,\ldots,\lb_1)\cdots(\lb_k, \ldots,p)$ it follows that
\begin{align*}
&\Theta^H(\Cc_{p, \lb})(\phi_0, \ldots, \phi_p):=\\
&\sum_\s (-1)^\s   \sum_{ 1\le i_s, j_t\le n} \gbar(\eta^{j_1}_{i_1, j_{\lb(1)}})(\phi_{\s(1)}^{-1})\cdots
\gbar(\eta^{j_p}_{i_p, j_{\lb(p)}})(\phi_{\s(p)}^{-1}) \td\t^{i_1}\wdots \td\t^{i_p}\\
&~~~~~~~~~~~~~~~= \sum_\s (-1)^\s \sum_{ 1\le j_t\le n} Q^{j_1}_{j_{\lb(1)}}(\phi_{\s(1)})\cdots
Q^{j_p}_{j_{\lb(p)}}(\phi_{\s(p)})\\
& ~~~=\sum_{\s}(-1)^\s \Tr(Q(\phi_{\s(1)})\cdots Q(\phi_{\s(\lb_1)}))
\wdots \Tr(Q(\phi_{\s(\lb_k)})\cdots Q(\phi_{\s(p)})).
\end{align*}
Applying the formula \eqref{simpcurvG} yields the desired expression.
\end{proof}
\medskip

   Note that, although $\s(0)$ does not appear in the
right hand side of formula \eqref{BottCh}, the expression does actually involve
all the variables $\phi_0, \ldots, \phi_p$.
\bigskip

{\bf 4.3.2.\,}  To clarify the relationship between
 the above expressions for equivariant Chern classes and the classical ones, we
 recall Bott's formula for the $\Diff$-equivariant Chern character of the
frame bundle to $\Rb^n$, following Getzler's account in
 \cite[Introduction]{getz}.

 Given diffeomorphisms $\phi_0, \ldots, \phi_p$, we associate to any point
 $\, \tb = (t_1, \ldots , t_p) \in \D^p \, , \,  0 \le t_1 \le \ldots \le t_p \le 1$
 the connection
 \begin{align}
 \begin{split}
\vartheta(\tb)= d +
 t_1\big(\vartheta(\phi_0)-\vartheta(\phi_1)\big)+ \cdots +
 t_p \big(\vartheta (\phi_{k-1})-\vartheta (\phi_p)\big) +\vartheta (\phi_p) .
\end{split}
\end{align}
 With
$$ \int_{\D^p} : \Om^{\bullet} (\D^p\times \Rb^n) \ra \Om^{\bullet -p}  (\Rb^n)
$$
denoting the normalized integral over the fibre of the projection
$\D^p\times  \Rb^n \ra  \Rb^n$, one defines
 the Chern-Simons form associated to the diffeomorphisms
 $(\phi_0,\ldots, \phi_p)$ by
 \begin{align} \label{CS}
\cs^{(p)} (\phi_0,\ldots, \phi_p) \, = \,  \int_{\D^p} \Tr e^{- \Om(\tb)}  ,
\end{align}
where $\, \Om (\tb) = \vartheta(\tb)^2 $ is the curvature form.
From Stokes' formula, which in this context assumes the form
\begin{align*}
  \int_{\D^p} \circ \, d \, -\, (-1)^p d \circ \int_{\D^p} \, +\, \int_{\p \D^p}=0 ,
\end{align*}
it follows that the above forms satisfy the cocycle identity
 \begin{align} \label{totalCS}
d \, \cs^{(p)} (\phi_0,\ldots,\phi_p)+
\sum_{i=0}^p(-1)^i \cs^{(p-1)} (\phi_0,\cdots, \hat\phi_i,\cdots, \phi_p) \, = \, 0 .
\end{align}
Thus, $\, \cs = \big(\cs^{(p)}\big)_{0 \le p \le n}$ defines a total cocycle in the Bott bicomplex.

Taking into account the structure equation \eqref{streq}, the curvature form
has  the expression
 \begin{align}\nonumber
  \Om (\tb) &=\sum_{i=1}^p dt_i \wdg \big(\vartheta (\phi_{i-1})-\vartheta (\phi_i)\big)
 -  \sum_{i=1}^p t_i  \big(\vartheta (\phi_{i-1}) \wdg \vartheta (\phi_{i-1})-\vartheta (\phi_i) \wdg \vartheta (\phi_i)\big)   \\\notag
& + \sum_{i =1}^p t_i \left(\big(\vartheta (\phi_{i-1})-\vartheta (\phi_i)\big) \wdg  \vartheta (\phi_p) +
\vartheta (\phi_p) \wdg \big(\vartheta (\phi_{i-1})-\vartheta (\phi_i)\big)\right)\\
&+ \sum_{i, j=1}^p t_i t_j \big(\vartheta (\phi_{i-1})-\vartheta (\phi_i)\big) \wdg \big(\vartheta (\phi_{j-1})-\vartheta (\phi_j)\big) .
   \end{align}
Writing
  \begin{align*}
\cs^{(p)} (\phi_0,\ldots, \phi_p) \, = \, \frac{ (-1)^p}{p!} \int_{\D^p} \Tr \left(\Om(\tb)^p \right) +
\sum_{p+1 \le k \le n} \frac{ (-1)^k}{k!} \int_{\D^p} \Tr \left(\Om(\tb)^k \right) ,
\end{align*}
one sees that the first term
involves only the product of the $p$ copies of
\begin{align*}
\begin{split}
&\sum_{i=1}^p dt_i \wdg \big(\vartheta (\phi_{i-1})-\vartheta (\phi_i)\big) \, = \,
\sum_{i=0}^p ds_i \wdg \vartheta (\phi_i) , \\
&\text{where} \quad  s_0 =t_1, \ldots , s_{i-1} = t_i - t_{i-1} , \ldots , s_p = 1 - t_p.
\end{split}
\end{align*}
Up to a normalizing factor $\nu_p$,  this term coincides
with $\, \Cc_{p, \lb(p)}^{\rm Bott}(\phi_0, \ldots, \phi_p)$,
where $\lb(p)$ stands for the full cycle $(1, \ldots , p)$. Thus,
 \begin{align*}
\nu_p \Cc_{p, \lb(p)}^{\rm Bott}(\phi_0, \ldots, \phi_p) +
\sum_{p+1 \le k \le n} \frac{ (-1)^k}{k!} \int_{\D^p} \Tr \left(\Om(\tb)^k \right)
\, = \, \cs^{(p)} (\phi_0,\ldots, \phi_p) .
\end{align*}
This offers an exact recipe for how to complete the Hopf Hochschild cocycle
$\, \sum_{0 \le p \le n} \nu_p C_{p, \lb(p)}$ to a total
 Hopf cyclic cocycle $\sum_{0 \le p \le n}\nu_p TC_{p, \lb(p)}$.
\bigskip

{\bf 4.3.3.\,}
Transported via the characteristic maps, the above classes can be concretely realized
in the cyclic bicomplex of \'etale groupoids.
 Thus, the expression
\begin{align}
&\Cc_{p,\lb}^{\Psi, \G} (\om_0U^\ast_{\phi_0},\;\ldots,\; \om_pU^\ast_{\phi_p} )\,=\\ \notag
&~~\displaystyle \int_{V}\Cc_{p, \lb}^{\rm Bott}(1,\;\phi_0, \ldots,\; \phi_{p-1}\cdots\phi_{0})\wdg
\om_0\wdg\phi_0^\ast\om_1\wdots (\phi_{p-1}\cdots \phi_0)^\ast\om_{p}, \\ \notag
&~~~~~~~~~~~~~~~~\text{if \quad   $\phi_p\cdots\phi_0= 1$ \quad
and \quad $0$ otherwise} , \notag
\end{align}
represents the bottom component
 $\displaystyle \Cc_{p,\lb}^{\Psi, \G} =
 \Psi^H (\Cc_{p, \lb}^{\rm Bott})$ of the Chern class $ \Psi^H (\Ch^{\rm Bott}_{p, \lb})$
  in the
cyclic cohomology bicomplex of the differential graded crossed product algebra
$C^q(\Om_{c,\G}(V))$.

\medskip

Similarly,
\begin{align}
 &\Cc_{p,\lb}^{\Upsilon, \G} (f^0, \ldots, f^{n-p}\mid \phi_0,\ldots,\ph_p )\,=\\ \notag
&\displaystyle \int_V f^0df^1\wdots df^{n-p} \wdg
\Cc_{p, \lb}^{\rm Bott}(\phi_0^{-1}, \phi_1^{-1}\phi_0^{-1},
 \ldots,\phi_{p-1}^{-1}\cdots\phi_0^{-1},1), \\ \notag
&~~~~~~~~~~~~~~~~~\text{if \quad $\phi_p\cdots\phi_0= 1$ \quad
and \quad $0$ otherwise}, \notag
\end{align}
represents the bottom component
 $\displaystyle
 \Cc_{p,\lb}^{\Upsilon, \G}= \Upsilon^H(\Cc_{p, \lb}^{\rm Bott})$ of the corresponding
 Chern class
 in the cyclic cohomology bicomplex  $C^{\bullet, \bullet}(C^{\infty}_c(V),\;\G)$.

 \medskip

 Finally, representatives for the Chern classes in the cyclic cohomology bicomplex
  $C^{\bullet}(\Ac_{c,\G}(V))$  can be obtained
  by transport via  Connes' map $\Phi^H$, \cf
  \eqref{Phi-rel}.

  \end{document}